\documentclass[reqno,10pt,centertags]{amsart}
\usepackage{amsmath,amsthm,amscd,amssymb,latexsym,esint,upref,stmaryrd,
enumerate,color,verbatim,yfonts}

\newcommand{\C}{{\mathbb C}}

\newcommand{\bbC}{{\mathbb{C}}}

\newcommand{\bbN}{{\mathbb{N}}}

\newcommand{\bbR}{{\mathbb{R}}}

\newcommand{\bsA}{{\boldsymbol{A}}}
\newcommand{\bsB}{{\boldsymbol{B}}}

\newcommand{\bsD}{{\boldsymbol{D}}}

\newcommand{\bsH}{{\boldsymbol{H}}}
\newcommand{\bsI}{{\boldsymbol{I}}}

\newcommand{\bsT}{{\boldsymbol{T}}}

\newcommand{\cA}{{\mathcal A}}
\newcommand{\cB}{{\mathcal B}}
\newcommand{\cC}{{\mathcal C}}
\newcommand{\cD}{{\mathcal D}}
\newcommand{\cE}{{\mathcal E}}

\newcommand{\cH}{{\mathcal H}}

\newcommand{\cK}{{\mathcal K}}
\newcommand{\cL}{{\mathcal L}}

\newcommand{\cS}{{\mathcal S}}

\newcommand{\Lf}{f_L}
\newcommand{\Lmodf}{|f|_L}
\newcommand{\LSf}{(Sf)_L}
\newcommand{\Lxi}{\xi_L}

\newcommand{\e}{\varepsilon}

\newcommand{\x}{\xi}

\DeclareMathOperator{\supp}{supp}

\DeclareMathOperator{\ran}{ran}
\DeclareMathOperator{\dom}{dom}

\DeclareMathOperator{\tr}{tr}

\DeclareMathOperator{\erf}{erf}
\DeclareMathOperator*{\nlim}{n-lim}
\DeclareMathOperator*{\slim}{s-lim}

\DeclareMathOperator*{\sgn}{sgn}

\renewcommand{\Re}{\text{\rm Re}}
\renewcommand{\Im}{\text{\rm Im}}
\renewcommand{\ln}{\text{\rm ln}}

\newcommand{\loc}{\text{\rm{loc}}}
\newcommand{\ind}{\operatorname{ind}}
\newcommand{\no}{\notag}
\newcommand{\lb}{\label}
\newcommand{\f}{\frac}

\newcommand{\ol}{\overline}

\newcommand{\wti}{\widetilde}
\newcommand{\Oh}{O}
\newcommand{\oh}{o}
\newcommand{\hatt}{\widehat} 
\newcommand{\bi}{\bibitem}

\renewcommand{\le}{\leqslant}

\let\geq\geqslant
\let\leq\leqslant

\makeatletter
\def\theequation{\@arabic\c@equation}

\allowdisplaybreaks 
\numberwithin{equation}{section}

\newtheorem{theorem}{Theorem}[section]
\newtheorem{proposition}[theorem]{Proposition}
\newtheorem{lemma}[theorem]{Lemma}
\newtheorem{corollary}[theorem]{Corollary}
\newtheorem{definition}[theorem]{Definition}
\newtheorem{hypothesis}[theorem]{Hypothesis}
\newtheorem{example}[theorem]{Example}

\theoremstyle{remark}
\newtheorem{remark}[theorem]{Remark}

\begin{document}

\numberwithin{equation}{section}
\allowdisplaybreaks

\title[On the Witten index]{On the Witten Index in Terms of \\ Spectral Shift Functions}

\author[A.\ Carey]{Alan Carey}  
\address{Mathematical Sciences Institute, Australian National University, Kingsley St., 
Canberra, ACT 0200, Australia}  
\email{acarey@maths.anu.edu.au}  
\urladdr{http://maths.anu.edu.au/~acarey/}
  
\author[F.\ Gesztesy]{Fritz Gesztesy}  
\address{Department of Mathematics,
University of Missouri, Columbia, MO 65211, USA}
\email{gesztesyf@missouri.edu}
\urladdr{http://www.math.missouri.edu/personnel/faculty/gesztesyf.html}

\author[D.\ Potapov]{Denis Potapov}
\address{School of Mathematics and Statistics, UNSW, Kensington, NSW 2052,
Australia} 
\email{d.potapov@unsw.edu.au}

\author[F.\ Sukochev]{Fedor Sukochev}
\address{School of Mathematics and Statistics, UNSW, Kensington, NSW 2052,
Australia} 
\email{f.sukochev@unsw.edu.au}

\author[Y.\ Tomilov]{Yuri Tomilov}
\address{Faculty of Mathematics and Computer Science, Nicholas 
Copernicus University, ul.\ Chopina 12/18, 87-100 Torun, Poland, and Institute of Mathematics, Polish Academy of Sciences. \'Sniadeckich str. 8, 00-956 Warsaw, Poland}
\email{tomilov@mat.uni.torun.pl}

\dedicatory{Dedicated with admiration to Sergio Albeverio on the occasion of 
his 75th birthday.}
\thanks{To appear in {\it J. Analyse Math.}}

\date{\today}

\begin{abstract} 
We study the model operator $\bsD_\bsA^{} = (d/dt) + \bsA$ in $L^2(\bbR;\cH)$ associated 
with the operator path $\{A(t)\}_{t=-\infty}^{\infty}$, where $(\bsA f)(t) = A(t) f(t)$ for 
a.e.\ $t\in\bbR$, and appropriate $f \in L^2(\bbR;\cH)$ (with $\cH$ a separable, complex 
Hilbert space). Denoting by $A_{\pm}$ the norm resolvent limits of $A(t)$ as $t \to \pm \infty$, 
our setup permits $A(t)$ in $\cH$ to be an unbounded, relatively trace class  
perturbation of the unbounded self-adjoint operator $A_-$, and no discrete spectrum 
assumptions are made on $A_{\pm}$. 

Introducing 
$\bsH_1 = \bsD_\bsA^* \bsD_\bsA^{}$,  $\bsH_2 = \bsD_\bsA^{} \bsD_\bsA^*$, 
the resolvent and semigroup regularized Witten indices of $\bsD_\bsA^{}$, denoted by 
$W_r(\bsD_{\bsA}^{})$ and $W_s(\bsD_{\bsA}^{})$, are defined by  
\begin{align*}
& W_r(\bsD_\bsA^{}) = \lim_{\lambda \uparrow 0} (- \lambda) 
\tr_{L^2(\bbR;\cH)}\big((\bsH_1 - \lambda \bsI)^{-1}
- (\bsH_2 - \lambda \bsI)^{-1}\big),       \\
& W_s(\bsD_\bsA^{}) = \lim_{t \uparrow \infty} \tr_{L^2(\bbR;\cH)} 
\big(e^{- t \bsH_1} - e^{-t \bsH_2}\big),     
\end{align*} 
whenever these limits exist. These regularized indices coincide with the Fredholm 
index of $\bsD_\bsA^{}$ whenever the latter is Fredholm. 

In situations where $\bsD_\bsA^{}$ ceases to be a Fredholm operator in $L^2(\bbR;\cH)$
we compute its resolvent (resp., semigroup) regularized Witten index in terms of the spectral 
shift function $\xi(\,\cdot\,;A_+,A_-)$ associated with the pair $(A_+, A_-)$ as follows: Assuming 
$0$ to be a right and a left Lebesgue 
point of $\xi(\,\cdot\,\, ; A_+, A_-)$, denoted by $\Lxi(0_+; A_+,A_-)$ and $\Lxi(0_-; A_+, A_-)$, 
we  prove that $0$ is also a right Lebesgue point of 
$\xi(\,\cdot\,\, ; \bsH_2, \bsH_1)$, denoted by $\Lxi(0_+; \bsH_2, \bsH_1)$, and that 
\begin{align*} 
W_r(\bsD_\bsA^{}) &= W_s(\bsD_\bsA^{})    \\ 
& = \Lxi(0_+; \bsH_2, \bsH_1)       \\ 
& = [\Lxi(0_+; A_+,A_-) + \Lxi(0_-; A_+, A_-)]/2,     
\end{align*}
the  principal result of this paper.

In the special case where $\dim(\cH) < \infty$, we prove that the Witten indices 
of $\bsD_\bsA^{}$ are either integer, or half-integer-valued. 
\end{abstract}

\maketitle

\newpage 

{\scriptsize{\tableofcontents}}

\section{Introduction}  \lb{s1}

To describe the principal results on the Witten index for a class of non-Fredholm operators in
this paper, we first briefly recall the setup and principal results on Fredholm indices of a particular
model operator in \cite{GLMST11}. We were motivated there by the approach in \cite{RS95} 
where operators of this general form are argued to
provide models for more complex situations. They arise in connection with investigations of the
Maslov index, Morse theory, Floer homology, Sturm oscillation theory, etc.

To introduce the model, let $\{A(t)\}_{t\in\bbR}$ be a family of self-adjoint operators in 
the complex, separable Hilbert space $\cH$, subject to a relative trace class condition 
described in detail in Hypothesis \ref{h2.1}, and denote by $\bsA$ the operator in 
$L^2(\bbR;\cH)$ defined by   
\begin{align}
&(\bsA f)(t) = A(t) f(t) \, \text{ for a.e.\ $t\in\bbR$,}   \no \\
& f \in \dom(\bsA) = \bigg\{g \in L^2(\bbR;\cH) \,\bigg|\,
g(t)\in \dom(A(t)) \text{ for a.e.\ } t\in\bbR;     \lb{1.1} \\
& \quad t \mapsto A(t)g(t) \text{ is (weakly) measurable;} \,  
\int_{\bbR} dt \, \|A(t) g(t)\|_{\cH}^2 < \infty\bigg\}.   \no 
\end{align}
The relative trace class setup employed in \cite{GLMST11} ensures that $A(t)$ has self-adjoint 
limiting operators
\begin{equation} 
  A_+=\lim_{t\to+\infty}A(t), \quad A_-=\lim_{t\to-\infty}A(t)    \lb{1.2}
\end{equation} 
in $\cH$ in the norm resolvent sense (detailed in \eqref{2.6}). Next, we introduce the model 
operator 
\begin{equation}
\bsD_\bsA^{} = \f{d}{dt} + \bsA,
\quad \dom(\bsD_\bsA^{})= \dom(d/dt) \cap \dom(\bsA_-),     \lb{1.5}   \\
\end{equation}
and the associated nonnegative, self-adjoint operators 
\begin{equation}
\bsH_1 = \bsD_\bsA^* \bsD_\bsA^{}, \quad \bsH_2 = \bsD_\bsA^{} \bsD_\bsA^*,   
\lb{1.6}
\end{equation}
in $L^2(\bbR;\cH)$. (Here $\bsA_-$  in $L^2(\bbR;\cH)$ represents the self-adjoint constant 
fiber operator defined according to \eqref{1.1}, with $A(t)$ replaced by the asymptote $A_-$, 
cf.\ \eqref{2.13}.)  

Assuming that $A_-$ and $A_+$ are boundedly invertible, we also recall (cf.\ \cite{GLMST11}) that
$\bsD_\bsA^{}$ is a Fredholm operator in $L^2(\bbR;\cH)$. Moreover, as shown in 
\cite{GLMST11} (and earlier in \cite{Pu08} under a simpler set of hypotheses on the family 
$A(\cdot)$), the Fredholm index of $\bsD_\bsA^{}$ then can be computed as follows,   
\begin{align}
\ind (\bsD_\bsA^{}) & = \xi(0_+; \bsH_2, \bsH_1)    \lb{1.9a} \\
& = \xi(0; A_+, A_-)   \lb{1.10} \\  
& = {\tr}_{\cH}(E_{A_-}((-\infty,0))-E_{A_+}((-\infty,0)))     \lb{1.12} \\ 
& = \pi^{-1} \lim_{\varepsilon \downarrow 0}\Im\big(\ln\big({\det}_{\cH}
\big((A_+ - i \varepsilon I_{\cH})(A_- - i \varepsilon I_{\cH})^{-1}\big)\big)\big).        \lb{1.13}
\end{align}
Here $\{E_{S}(\lambda)\}_{\lambda\in\bbR}$ denotes the family of spectral projections associated 
with the self-adjoint operator $S$, $\tr_{\cH}(\cdot)$ and $\det_{\cH} (\cdot)$ represent the trace 
and Fredholm determinant, and $\xi(\, \cdot \, ; S_2,S_1)$ denotes the spectral shift function for 
the pair of self-adjoint operators $(S_2, S_1)$. Whenever $S_j$, $j=1,2$, are bounded from below, 
we adhere to the normalization
\begin{equation}
\xi(\lambda ; S_2,S_1) = 0 \, \text{ for } \, \lambda < \inf(\sigma(S_1) \cup \sigma(S_2)),  
\end{equation}  
in particular, $\xi(\lambda; \bsH_2, \bsH_1) = 0$, $\lambda < 0$.

Given these preparations, we can now describe the new direction developed in this paper which  
focuses on the model operator $\bsD_\bsA^{}$ in $L^2(\bbR;\cH)$ whenever the latter ceases to 
be Fredholm. However, before turning our attention to the specific model operator 
$\bsD_\bsA^{}$, we digress a bit and first recall the definition of the Witten index as studied 
in \cite{BGGSS87} and \cite{GS88} for a linear, closed, densely defined operator $T$ in 
some complex, separable Hilbert space $\cH$. (Details for the following material 
on $T$ are presented in Section \ref{s3}.) We start with the resolvent regularized  
Witten index, denoted by $W_r(T)$, 
\begin{equation}
W_r(T) = \lim_{\lambda \uparrow 0} (- \lambda) 
\tr_{\cH}\big((T^*T - \lambda I_{\cH})^{-1} - (TT^* - \lambda I_{\cH})^{-1}\big),     \lb{1.13a}
\end{equation} 
whenever the limit exists. The corresponding semigroup (heat kernel) regularized Witten 
index, denoted $W_s(\bsD_\bsA^{})$, is 
analogously defined by 
\begin{equation}
W_s(T) = \lim_{t \uparrow \infty} \tr_{\cH} \big(e^{- t T^*T} - e^{-t TT^*}\big),     \lb{1.13b}
\end{equation}
again, whenever the limit exists. 

If $T$ is Fredholm (and of course the necessary trace class conditions in \eqref{1.13a} 
and \eqref{1.13b} are satisfied), one has consistency with the Fredholm index, $\ind(T)$ of $T$, 
that is,
\begin{equation} 
\ind(T) = W_r(T) = W_s(T). 
\end{equation}

In general (i.e., if $T$ is not Fredholm), $W_r(T)$ is not necessarily 
integer-valued; in fact, one can show that it can take on any prescribed real number. The 
intrinsic value of $W_r(T)$ then lies in its stability properties with respect to additive perturbations,   
analogous to stability properties of the Fredholm index. Indeed, as long as one replaces the 
familiar relative compactness assumption on the perturbation in connection with the 
Fredholm index, by appropriate relative trace class conditions in connection with 
the Witten index, stability of the Witten index was proved in \cite{BGGSS87} and \cite{GS88} 
(see also \cite{CP74}). 

Thus, in cases where $T$ ceases to be Fredholm, the Fredholm index of $T$ is naturally replaced 
by the regularized Witten index $W_r(T)$, respectively, $W_s(T)$, assuming both regularized 
indices coincide.  

Under appropriate spectral assumptions on $T^*T$ and $TT^*$, the following connection 
between Fredholm, respectively, Witten indices and the underlying spectral shift function for the 
pair $(TT^*, T^*T)$,  
\begin{equation}
W_r(T) = W_s(T) = \xi(0_+; TT^*, T^*T)    \lb{1.13c}
\end{equation} 
(cf.\ also \eqref{1.9a} and \eqref{1.15}) goes back to \cite{BGGSS87}, \cite{GS88}, and 
in a somewhat different context to \cite{CP86} and \cite[Ch.\ X]{Mu87}. Incidentally, \eqref{1.13c}  
shows the sought after consistency between $W_r(T)$ and $W_s(T)$. 

Originally, index regularizations such as \eqref{1.13a} were studied in the context of 
supersymmetric quantum mechanics in the physics literature in the 1970's and 1980's, see, 
for instance, \cite{AC84}, \cite{BB84}, \cite{Ca78}, \cite{Hi83}, \cite{HT82}, \cite{NS86}, 
\cite{NT85}, \cite{We79}, \cite{Wi82}. In particular, Callias' paper \cite{Ca78} was very 
influential and inspired a fair number of mathematical investigations. For relevant work we 
refer, for instance, to \cite{An89}, \cite{An90}, \cite{An93}, \cite{BGGSS87}, \cite{BMS88}, 
\cite{BS78}, \cite{Bu92}, \cite{GS88}, \cite{Ko11}, \cite{Mu87}, \cite{Mu88}, \cite{Mu98}, 
and the references therein. The intimate connection between the Fredholm index of 
$\bsD_\bsA^{}$ and the spectral flow associated with the operator path $\{A(t)\}_{t \in \bbR}$ 
is not discussed in this paper, but for relevant results we refer, for 
instance, to \cite{DS94}, \cite{Fl88}, \cite{GLMST11}, \cite{Le05}, \cite{RS95}, and the 
extensive list of references therein.  

At this point we leave the general discussion of the general operator $T$ and return to our 
specific model operator $\bsD_\bsA^{}$ in $L^2(\bbR;\cH)$. Assuming that $0$ is a right 
and a left Lebesgue point of $\xi(\,\cdot\,\, ; A_+, A_-)$ (denoted by $\Lxi(0_+; A_+,A_-)$ 
and $\Lxi(0_-; A_+, A_-)$, respectively), we  will prove in Section \ref{s4} that it then is 
also a right Lebesgue point of $\xi(\,\cdot\,\, ; \bsH_2, \bsH_1)$ (denoted by 
$\Lxi(0_+; \bsH_2, \bsH_1)$). Under this right/left Lebesgue point 
assumption on $\xi(\,\cdot\,\, ; A_+, A_-)$, the analog of \eqref{1.5} and \eqref{1.6} in the 
general case where $\bsD_\bsA^{}$ ceases to be Fredholm, and hence the  principal new 
result of the present paper, then reads as follows, 
\begin{align} 
W_r(\bsD_\bsA^{}) &= W_s(\bsD_\bsA^{})  \lb{1.14} \\ 
& = \Lxi(0_+; \bsH_2, \bsH_1)      \lb{1.15} \\ 
& = [\Lxi(0_+; A_+,A_-) + \Lxi(0_-; A_+, A_-)]/2.    \lb{1.16}
\end{align}

In the special case where $\dim(\cH) < \infty$, we will prove in addition, that the Witten indices 
in \eqref{1.14}--\eqref{1.16} are either integer, or half-integer-valued.

The additional formula \eqref{1.10} in terms of the spectral shift function 
$\xi(\, \cdot \,; A_+,A_-)$ for the pair 
of the pair of boundedly invertible asymptotes $(A_+, A_-)$, and hence in the case where 
$\bsD_\bsA^{}$ is Fredholm, is of fairly recent origin and due to \cite{Pu08} and (under the 
general assumptions in Hypothesis \ref{h2.1}) due to \cite{GLMST11}. Together with the new 
result \eqref{1.16} proven in this paper, we now succeeded to extend this circle of ideas  
also to the non-Fredholm case for the model operator $\bsD_\bsA^{}$. We emphasize that 
no discrete spectrum assumptions are ever made on $A_{\pm}$.

Next, we briefly describe the content of each section. Section 2 provides a succinct summary of 
the results in \cite{GLMST11}, focusing on the case where $\bsD_\bsA^{}$ is a Fredholm operator. 
In this context we note that $\bsD_\bsA^{}$ is Fredholm if and only if $A_-$ and $A_+$ are 
boundedly invertible in $\cH$ (cf.\ Theorem \ref{t2.6}). While the ``if'' part of this statement was 
known (see, \cite{GLMST11}), the ``only if'' part is one of the new results in this paper. The 
resolvent and semigroup 
regularized Witten indices are studied in great detail in Section \ref{s3}. Our principal Section \ref{s4} 
is then devoted to the computation of the Witten index of $\bsD_\bsA^{}$ when the latter ceases to 
be a Fredholm operator under the assumption that $0$ is a right and a left Lebesgue point of 
$\xi(\, \cdot \,; A_+, A_-) $. Section \ref{s5} provides a complete treatment of the particular case 
$\dim(\cH) < \infty$ (no Lebesgue point assumptions are necessary in this case). 
As the Witten index, in the non-Fredholm case, is not invariant under (relatively) compact perturbations it cannot have any connection to K-theory. Our final Section \ref{s6} argues that cyclic homology is the appropriate substitute for K theory for non-Fredholm operators. 
A brief Appendix \ref{sA} collects analytical tools (such as the definition of 
right/left Lebesgue points, basic facts on self-adjoint rank-one perturbations of self-adjoint operators, 
and some results on truncated Hilbert transforms and the Poisson and conjugated Poisson 
operator for the open upper half-plane) used at various places in this paper. Appendix \ref{sB} 
considers more general trace relations than the standard one in \eqref{2.19} and 
establishes the connection with the Abel transform underlying formula \eqref{2.37a}. 
Appendix \ref{sC} offers an alternative proof of the important estimate \eqref{3.42}. 

Finally, we briefly summarize some of the notation used in this paper: 
Let $\cH$ be a separable, complex Hilbert space, $(\cdot,\cdot)_{\cH}$ the scalar product in $\cH$
(linear in the second argument), and $I_{\cH}$ the identity operator in $\cH$.

Next, if $T$ is a linear operator mapping (a subspace of) a Hilbert space into another, then 
$\dom(T)$ and $\ker(T)$ denote the domain and kernel (i.e., null space) of $T$. 
The closure of a closable operator $S$ is denoted by $\ol S$. 
The spectrum, essential spectrum, discrete spectrum, point spectrum (i.e., the set of eigenvalues), 
and resolvent set 
of a closed linear operator in a Hilbert space will be denoted by $\sigma(\cdot)$, 
$\sigma_{ess}(\cdot)$, $\sigma_{d}(\cdot)$, $\sigma_{p}(\cdot)$, and $\rho(\cdot)$, 
respectively. 

The strongly right continuous family of spectral projections of a self-adjoint operator $S$ in $\cH$ will be denoted by $E_S(\lambda)$, $\lambda \in \bbR$. (In 
particular, $E_S(\lambda) = E_S((-\infty,\lambda])$, 
$E_S((-\infty, \lambda)) = \slim_{\varepsilon\downarrow 0} E_S(\lambda - \varepsilon)$, and $E_S((\lambda_1,\lambda_2]) = E_S(\lambda_2) - E_S(\lambda_1)$, 
$\lambda_1 < \lambda_2$, $\lambda, \lambda_1, \lambda_2 \in\bbR$.)

The convergence in the strong operator topology (i.e., pointwise limits) will be denoted by $\slim$. 
Similarly, limits in the norm topology are abbreviated by $\nlim$. 

If $T$ is a Fredholm operator, its Fredholm index is denoted by $\ind(T)$.  

The Banach spaces of bounded and compact linear operators on a separable, complex Hilbert space 
$\cH$ are denoted by $\cB(\cH)$ and $\cB_\infty(\cH)$, respectively; the corresponding 
$\ell^p$-based trace ideals will be denoted by $\cB_p (\cH)$, $p>0$. 
Moreover, ${\det}_{\cH}(I_\cK-A)$, 
and $\tr_{\cH}(A)$ denote the standard Fredholm determinant and the corresponding trace 
of a trace class operator $A\in\cB_1(\cH)$. 

We will use the abbreviation $\bbC_+ = \{z\in\bbC\,|\, \Im(z) > 0\}$ for the open complex upper-half plane.

\section{The Fredholm Index of $\bsD_{\bsA}^{}$ 
and the Spectral Shift Function}  \lb{s2} 

In this section we recall some of the principal results obtained in \cite{GLMST11} 
on the Fredholm index of the model operator $\bsD_\bsA^{}$ (cf.\ \eqref{2.DA}). 
Since in subsequent sections we will go beyond the Fredholm property of 
$\bsD_\bsA^{}$ and aim at computing its Witten index, a summary of the results 
in \cite{GLMST11} will prove to be a useful basis for our point of departure into new territory.

Throughout this section we make the following assumptions: 

\begin{hypothesis} \lb{h2.1}
Suppose $\cH$ is a complex, separable Hilbert space. \\
$(i)$ Assume $A_-$ is self-adjoint on $\dom(A_-) \subseteq \cH$. \\
$(ii)$ Suppose there exists a family of operators $B(t)$, $t\in\bbR$,
closed and symmetric in $\cH$, with $\dom(B(t)) \supseteq \dom(A_-)$, $t\in\bbR$. \\
$(iii)$ Assume there exists a family of operators $B'(t)$,
$t\in\bbR$, closed and symmetric in $\cH$, with 
$\dom(B'(t)) \supseteq \dom(A_-)$, such that the family
$B(t)(|A_-| + I_{\cH})^{-1}$, $t\in\bbR$, is weakly locally absolutely continuous, and for a.e.\ $t\in\bbR$,  
\begin{equation}
\frac{d}{dt} (g,B(t)(|A_-| + I_{\cH})^{-1}h)_{\cH}
=(g,B'(t) (|A_-| + I_{\cH})^{-1}h)_{\cH}, \quad g, h\in\cH.     \lb{2.1}
\end{equation} 
$(iv)$ Assume that $B'(t)(|A_-| + I_{\cH})^{-1}\in\cB_1(\cH)$,  $t\in\bbR$, and
\begin{equation}  \lb{2.2}
\int_\bbR dt \, \big\|B'(t) (|A_-| + I_{\cH})^{-1}\big\|_{\cB_1(\cH)} < \infty.
\end{equation}
$(v)$ Suppose that the families 
\begin{equation}
\big\{\big(|B(t)|^2 + I_{\cH} \big)^{-1}\big\}_{t\in\bbR} \, \text{ and } \, 
\big\{\big(|B'(t)|^2 + I_{\cH} \big)^{-1}\big\}_{t\in\bbR}      \lb{2.3} 
\end{equation}
are weakly measurable $($cf.\ \cite[Definition\ A.3\,$(ii)$]{GLMST11}$)$.
\end{hypothesis}

Assuming Hypothesis \ref{h2.1}, we introduce the family of operators 
$\{A(t)\}_{t\in\bbR}$ in $\cH$ by 
\begin{equation}
A(t) = A_- + B(t), \quad \dom(A(t)) = \dom(A_-), \; t\in\bbR,    \lb{2.4}
\end{equation}
and recall from \cite{GLMST11} that $A(t)$, $t \in \bbR$, are self-adjoint and that 
there exists a self-adjoint operator $A_+$ in $\cH$ such that
\begin{equation}
\dom(A_+) = \dom(A_-)     \lb{2.5}
\end{equation}
and
\begin{align}
& \nlim_{t\to \pm \infty}(A(t) - z I_{\cH})^{-1} = (A_{\pm} - z I_{\cH})^{-1}, \quad z\in\bbC\backslash\bbR,     \lb{2.6} \\
& \big[(A_+ - z I_{\cH})^{-1} - (A_- - z I_{\cH})^{-1}\big] \in \cB_1(\cH), \quad 
z \in \rho(A_+) \cap \rho(A_-).     \lb{2.7} 
\end{align}
In addition, we also introduce
\begin{equation}
B_- = 0, \quad B_+ = \ol{(A_+ - A_-)}, \quad \dom(B_+) \supseteq \dom(A_-), 
\end{equation}
and note that 
\begin{equation}
A_+ = A_- + B_+, \quad \dom(A_+) = \dom(A_-).    \lb{2.A+B+}
\end{equation} 

We also recall the estimate
\begin{equation} 
\big\|B(t) (|A_-| + I_{\cH})^{-1}\big\|_{\cB_1(\cH)}
\leq \int_{-\infty}^t ds \, \big\|B'(s) (|A_-| + I_{\cH})^{-1}\big\|_{\cB_1(\cH)} 
\leq C, \quad t \in \bbR,   \lb{2.7A}
\end{equation} 
for some finite constant $C>0$. 

Next, let $\bsA$ in $L^2(\bbR;\cH)$ be the operator associated with the family 
$\{A(t)\}_{t\in\bbR}$ in $\cH$ by
\begin{align}
&(\bsA f)(t) = A(t) f(t) \, \text{ for a.e.\ $t\in\bbR$,}   \no \\
& f \in \dom(\bsA) = \bigg\{g \in L^2(\bbR;\cH) \,\bigg|\,
g(t)\in \dom(A(t)) \text{ for a.e.\ } t\in\bbR,     \lb{2.bfA} \\
& \quad t \mapsto A(t)g(t) \text{ is (weakly) measurable,} \,  
\int_{\bbR} dt \, \|A(t) g(t)\|_{\cH}^2 < \infty\bigg\}.    \no
\end{align}

To state the principal results of \cite{GLMST11}, we start by introducing in 
$L^2(\bbR;\cH)$ the operator
\begin{equation}
\bsD_\bsA^{} = \f{d}{dt} + \bsA,
\quad \dom(\bsD_\bsA^{})= \dom(d/dt) \cap \dom(\bsA_-).   \lb{2.DA}
\end{equation}
The operator $d/dt$ in $L^2(\bbR;\cH)$  is defined by
\begin{align}
\begin{split}
& \bigg(\f{d}{dt}f\bigg)(t) = f'(t) \, \text{ for a.e.\ $t\in\bbR$,} 
\lb{2.ddt}  \\
& \, f \in \dom(d/dt) = \big\{g \in L^2(\bbR;\cH) \, \big|\,
g \in AC_{\loc}(\bbR; \cH), \, g' \in L^2(\bbR;\cH)\big\},
\end{split}
\end{align}
and we recall that 
\begin{align}
& g \in AC_{\loc}(\bbR; \cH) \, \text{ if and only if $g$ is of the form }        \lb{2.ac} \\
& \quad  g(t) = g(t_0) + \int_{t_0}^t ds \, h(s), \; t, t_0 \in \bbR, \, \text{ for some } \, 
h \in L^1_{\loc}(\bbR;\cH), \text{ and } \, g' = h \, a.e.   \no 
\end{align} 
(The integral in \eqref{2.ac} is of course a Bochner integral.) In addition, 
$\bsA$ is defined in \eqref{2.bfA} ($\bsB$, $\bsA' = \bsB'$ are analogously defined) 
and $\bsA_-$  in $L^2(\bbR;\cH)$ represents 
the self-adjoint (constant fiber) operator defined according to
\begin{align}
&(\bsA_- f)(t) = A_- f(t) \, \text{ for a.e.\ $t\in\bbR$,}   \no \\
& f \in \dom(\bsA_-) = \bigg\{g \in L^2(\bbR;\cH) \,\bigg|\,
g(t)\in \dom(A_-) \text{ for a.e.\ } t\in\bbR,    \no \\
& \quad t \mapsto A_- g(t) \text{ is (weakly) measurable,} \,  
\int_{\bbR} dt \, \|A_- g(t)\|_{\cH}^2 < \infty\bigg\}.    \lb{2.13}
\end{align} 

Assuming Hypothesis \ref{h2.1}, one can show that the operator
$\bsD_\bsA^{}$ is densely defined and closed in $L^2(\bbR; \cH)$. Similarly, 
the adjoint operator $\bsD_\bsA^*$ of $\bsD_\bsA^{}$ in $L^2(\bbR; \cH)$ is then given 
by (cf.\ \cite{GLMST11})
\begin{equation}
\bsD_\bsA^*=- \f{d}{dt} + \bsA, \quad
\dom(\bsD_\bsA^*) = \dom(d/dt) \cap \dom(\bsA_-) = \dom(\bsD_\bsA^{}).    \lb{2.14} 
\end{equation}

Using these operators, we define in $L^2(\bbR;\cH)$ the nonnegative 
self-adjoint operators
\begin{equation} \lb{2.15}
\bsH_1=\bsD_\bsA^* \bsD_\bsA^{},\quad \bsH_2=\bsD_\bsA^{} \bsD_\bsA^*.
\end{equation} 
Finally, let us define the functions
\begin{align} \lb{2.16}
g_z(x) & = x(x^2-z)^{-1/2}, \quad z\in\C\backslash [0,\infty), \; x\in\bbR. 
\end{align}

The following principal result relates the trace of the difference of the 
resolvents of $\bsH_1$ and $\bsH_2$ in $L^2(\bbR;\cH)$, and the trace 
of the difference of $g_z(A_+)$ and $g_z(A_-)$ in $\cH$.

\begin{theorem} [\cite{GLMST11}] \lb{t2.2}
Assume Hypothesis \ref{h2.1} and define the operators
$\bsH_1$ and $\bsH_2$ as in \eqref{2.15} and the function 
$g_z$ as in \eqref{2.16}. Then 
\begin{align}
& \big[(\bsH_2 - z \bsI)^{-1}-(\bsH_1 - z \bsI)^{-1}\big] \in \cB_1\big(L^2(\bbR;\cH)\big),  
\quad z\in\rho(\bsH_1) \cap \rho(\bsH_2),      \lb{2.17}  \\
& [g_z(A_+)-g_z(A_-)] \in \cB_1(\cH),  \quad 
z\in\rho\big(A_+^2\big) \cap \rho\big(A_-^2\big),     \lb{2.18}
\end{align}
and the following trace formula holds,  
\begin{align}
\begin{split}
   \lb{2.19}
   \tr_{L^2(\bbR;\cH)}\big((\bsH_2 - z \bsI)^{-1}-(\bsH_1 - z \, 
\bsI)^{-1}\big) = \frac{1}{2z} \tr_\cH \big(g_z(A_+)-g_z(A_-)\big),&  \\  
z\in\bbC\backslash [0,\infty).&  
\end{split}
\end{align} 
\end{theorem}

Next, one also needs to introduce the spectral shift function 
$\xi(\,\cdot\,; \bsH_2,\bsH_1)$ associated with the pair $(\bsH_2, \bsH_1)$. 
Since $\bsH_2\geq 0$ and $\bsH_1\geq 0$, and 
\begin{equation}
\big[(\bsH_2 + \bsI)^{-1} - (\bsH_1 + \bsI)^{-1}\big] \in \cB_1 
\big(L^2(\bbR;\cH)\big),
\end{equation}
by Theorem \ref{t2.2}, one uniquely introduces $\xi(\,\cdot\,; \bsH_2,\bsH_1)$ by requiring that
\begin{equation}
\xi(\lambda; \bsH_2,\bsH_1) = 0, \quad \lambda < 0,    \lb{2.27}
\end{equation}
and by 
\begin{equation}
\tr_{L^2(\bbR;\cH)} \big((\bsH_2 - z \bsI)^{-1} - (\bsH_1 - z \, 
\bsI)^{-1}\big)
= - \int_{[0, \infty)}  \frac{\xi(\lambda; \bsH_2, \bsH_1) \, 
d\lambda}{(\lambda -z)^2}, \quad z\in\bbC\backslash [0,\infty), 
\lb{2.28}
\end{equation}
following \cite[Sect.\ 8.9]{Ya92}. 

\begin{lemma} [\cite{GLMST11}] \lb{l2.3}
Assume Hypothesis \ref{h2.1}. Then 
\begin{equation}
\xi(\,\cdot\,; \bsH_2, \bsH_1) \in L^1\big(\bbR; (|\lambda| + 1)^{-1} 
d\lambda\big)   \lb{2.30}
\end{equation}
and 
\begin{equation}
\xi(\lambda; \bsH_2, \bsH_1) = \pi^{-1} \lim_{\varepsilon \downarrow 0} 
\Im\big(\ln\big(\wti D_{\bsH_2/\bsH_1}(\lambda + i \varepsilon)\big)\big) 
\, \text{ for a.e.\ } \, \lambda \in \bbR,   
\end{equation}
where we used the abbreviation 
\begin{align}
& \wti D_{\bsH_2/\bsH_1} (z) = {\det}_{L^2(\bbR;\cH)} 
\big((\bsH_1 - z \bsI)^{-1/2} (\bsH_2 - z \bsI) (\bsH_1 - z \bsI)^{-1/2}\big)  \no \\
& \quad = {\det}_{L^2(\bbR;\cH)} 
\big(\bsI + 2 (\bsH_1 - z \bsI)^{-1/2} \bsA' (\bsH_1 - z \bsI)^{-1/2}\big), 
\quad z \in \rho(\bsH_1).   
\end{align}
\end{lemma}

Next, we will introduce the spectral shift function associated with the pair 
$(A_+, A_-)$ via the invariance principle: By \eqref{2.18}, and the fact that 
$g_{-1}(A_{\pm})$ are self-adjoint (cf.\ \eqref{2.16} for the definition of $g$), Krein's 
trace formula in its simplest form yields (cf.\ \cite[Theorem\ 8.2.1]{Ya92})
\begin{equation}\lb{2.20}
\tr_{\cH}\big(g_{-1}(A_+) - g_{-1}(A_-)\big)
= \int_{[-1,1]} \xi(\omega; g_{-1}(A_+), g_{-1}(A_-)) \, d\omega.
\end{equation}
Next, we define 
(cf.\ also \cite[eq.\ 8.11.4]{Ya92})
\begin{equation}
\xi(\nu; A_+, A_-) := \xi(g_{-1} (\nu); g_{-1}(A_+), 
g{-1}(A_-)), \quad \nu\in\bbR,  \lb{2.21}
\end{equation}
 where
$\xi(\,\cdot\,;g_{-1}(A_+), g_{-1}(A_-))$ is the spectral shift function associated with the pair
$(g_{-1}(A_+), g_{-1}(A_-))$ uniquely determined by the requirement 
(cf.\ \cite[Sects.\ 9.1, 9.2]{Ya92})
\begin{equation}
\xi(\,\cdot\,;g_{-1}(A_+), g_{-1}(A_-)) \in L^1(\bbR; d\omega).  \lb{2.22}
\end{equation} 
This then implies
\begin{equation}
\tr_{\cH}\big(g_{z}(A_+) - g_{z}(A_-)\big)
  = - z \int_{\bbR} \frac{\xi(\nu; A_+, A_-) \, d\nu}{(\nu^2 - z)^{3/2}},
\quad  z\in\bbC\backslash [0,\infty).         \lb{2.24}
\end{equation}

It remains to detail the precise connection between $\xi$ and Fredholm 
perturbation  determinants associated with the pair $(A_-, A_+)$. Let 
\begin{equation}
D_{T/S}(z) = {\det}_{\cH} ((T-z I_{\cH})(S- z I_{\cH})^{-1}) = {\det}_{\cH}(I+(T-S)(S-z I_{\cH})^{-1}), 
\quad z \in \rho(S), 
\end{equation}
denote the perturbation determinant for the pair of operators $(S,T)$ in $\cH$, 
assuming $(T-S)(S-z_0)^{-1} \in \cB_1(\cH)$ for some (and hence for all) 
$z_0 \in \rho(S)$. 

\begin{theorem} [\cite{GLMST11}] \lb{t2.4} 
Assume Hypothesis \ref{h2.1}, then
\begin{equation} 
\xi(\lambda; A_+,A_-) = \pi^{-1}\lim_{\e\downarrow 0} 
\Im(\ln(D_{A_+/A_-}(\lambda+i\e))) \, 
\text{ for a.e.\ } \, \lambda\in\bbR,    \lb{2.34}
\end{equation} 
where we make the choice of branch of $\ln(D_{A_+/A_-}(\cdot))$ on $\bbC_+$ such that 
\begin{equation}
\lim_{\Im(z) \to +\infty}\ln(D_{A_+/A_-}(z))=0.
\end{equation} 
Suppose in addition that $0\in\rho(A_-)\cap\rho(A_+)$, then for a continuous representative of  $\xi(\,\cdot\,; A_+,A_-)$ in a neighborhood of $\lambda = 0$ the equality 
\begin{equation} 
\xi(0; A_+,A_-)= \pi^{-1} \lim_{\varepsilon \downarrow 0}
\Im(\ln(D_{A_+/A_-}(i\e)))   \lb{2.35}
\end{equation} 
holds.
\end{theorem}

We note that the representation \eqref{2.34} for $\xi(\,\cdot\,; A_+,A_-)$ was proved in \cite[Theorem\ 7.6]{GLMST11} only under the additional hypothesis that $0\in\rho(A_-)\cap\rho(A_+)$, by appealing to the decomposition 
\cite[eq.\ (7.50)]{GLMST11}, involving $A_-^{-1}$. But this additional assumption is clearly unnecessary by using the corresponding decomposition
\begin{align}
\begin{split} 
& (A_+ - A_-) (A_- - z I_{\cH})^{-1}     \\ 
& \quad = [(A_+ - A_-) (|A_-| + I_{\cH})^{-1}]
[(|A_-| + I_{\cH})(A_- - z I_{\cH})^{-1}],   \quad z \in \bbC\backslash\bbR, 
\end{split} 
\end{align} 
instead.

In addition, we recall that,   
\begin{align} 
\begin{split} 
\frac{d}{dz}\ln(D_{A_+/A_-}(z)) &= - {\tr}_{\cH}\big((A_+ - z I_{\cH})^{-1}-(A_- - z I_{\cH})^{-1}\big)   \\
& = \int_{\bbR}\frac{\xi(\lambda; A_+, A_-)\,d\lambda}{(\lambda-z)^2},    \quad 
z \in \bbC\backslash\bbR,     \lb{2.36}  
\end{split} 
\end{align} 
the trace formula associated with the pair $(A_+, A_-)$.  

Given these preparations, one obtains Pushnitski's formula \cite{Pu08} relating the spectral 
shift functions $\xi(\cdot \, ; \bsH_2, \bsH_1)$ and 
$\xi(\cdot \, ; A_+,A_-)$ under our general assumptions Hypothesis \ref{h2.1}: 

\begin{theorem} [\cite{GLMST11}] \lb{t2.5}
Assume Hypothesis \ref{h2.1} and define the spectral shift function 
$\xi(\,\cdot\,; \bsH_2,\bsH_1)$ according to \eqref{2.27} and \eqref{2.28}.
Then one has for a.e.\ $\lambda>0$,
\begin{equation} 
\xi(\lambda; \bsH_2, \bsH_1)=\frac{1}{\pi}\int_{-\lambda^{1/2}}^{\lambda^{1/2}}
\frac{\xi(\nu; A_+,A_-) \, d\nu}{(\lambda-\nu^2)^{1/2}},   \lb{2.37a}
\end{equation} 
with a convergent Lebesgue integral on the right-hand side of \eqref{2.37a}. 
\end{theorem}

Next, we turn to the connection between the spectral shift function and the Fredholm 
index of $\bsD_\bsA^{}$. (We recall that $\bsD_\bsA^{}$ is densely defined and closed in 
$L^2(\bbR;\cH)$, cf.\ \cite{GLMST11}.) 

First, we state a new criterion for the Fredholm 
property of $\bsD_\bsA^{}$. The following result is a generalization of 
\cite[Propositions\ 7.11, 7.12, Corollary\ 7.13]{LT05} and 
\cite[Theorem\ 1.2]{LP08} (in the context where $\bsA$ is self-adjoint). It shows, in 
particular, that the spectral assumptions 
in \cite[Theorem 3.12]{RS95} (see also \cite{Ra04}), as well as similar assumptions in related 
papers, are, in fact, necessary. 

If $T$ is a densely defined and closed operator in $\cH$, one recalls that
\begin{align}
& \text{$T$ is said to be upper-semi-Fredholm if $\ran(T)$ is closed}   \no \\
& \quad \text{and $\dim(\ker(T)) < \infty$,} \\
& \text{$T$ is said to be lower-semi-Fredholm if $\ran(T)$ is closed}   \no \\
& \quad \text{and $\dim(\ker(T^*)) < \infty$,} \\
& \text{$T$ is said to be Fredholm if $\ran(T)$ is closed}   \no \\
& \quad \text{and $\dim(\ker(T)) + \dim(\ker(T^*)) < \infty$.}
\end{align}
Equivalently, $T$ is Fredholm if $\dim(\ker(T)) + {\rm codim}(\ran(T)) < \infty$. 

\begin{theorem} \lb{t2.6}
Assume Hypothesis \ref{h2.1}. Then $\bsD_\bsA^{}$ Fredholm if and only if 
$0 \in \rho(A_+)\cap\rho(A_-)$. More precisely, if $0 \in \sigma(A_+)$ $($or $0 \in \sigma(A_-)$$)$, 
then $\bsD_\bsA^{}$ is neither upper-semi-Fredholm, nor lower-semi-Fredholm.
\end{theorem} 
\begin{proof}
That $\bsD_\bsA^{}$ is closed and densely defined in $L^2(\bbR; \cH)$, as well as the 
sufficiency of the condition $0 \in \rho(A_+)\cap\rho(A_-)$ for the Fredholm property of 
$\bsD_\bsA^{}$ have been proven in \cite[Lemma\ 4.2, Remark\ 4.3]{GLMST11}. It remains to show that the condition 
$0 \in \rho(A_+)\cap\rho(A_-)$ is also necessary for the Fredholm property of $\bsD_\bsA^{}$. 

Arguing by contradiction, let $0 \in \sigma (A_+)$. Then there exists a sequence 
\begin{equation} 
\{x_n\}_{n \in \bbN} \subset \dom(A_+), \text{ with $\|x_n\|_{\cH} = 1$, $n \in \bbN$, such that 
$\|A_+ x_n\|_{\cH} \underset{n\to\infty}{\longrightarrow} 0$.} 
\end{equation}
For each $n\in\bbN$, consider a smooth function $\chi_n:\bbR\to[0,1]$
such that
\begin{equation}
\chi_n(t)= \begin{cases} 1, & |t|\leq n, \\ 0, & |t|\geq n+1, \end{cases}
\quad  |\chi^{\prime}_n(t)|\leq 2, \; t\in\bbR.
\end{equation}
Next, we define $\{{\bf x}_n\}_{n \in \bbN} \subset \dom(\bsD_{\bsA}^{})$ by 
\begin{equation}
{\bf x}_n(t)= \|\chi_n\|^{-1}_{L^2(\bbR; dt)} \chi_n (t-2n-2) x_n, \quad t\in \bbR \; n \in \bbN,
\end{equation}
and note that $\|{\bf x}_n \|_{L^2(\bbR;\cH)}=1$, $n \in \bbN$, 
and
\begin{equation} \lb{4.57}
{\bf x}_n \underset{n\to\infty}{\longrightarrow} 0 \, \text{ weakly,} 
\end{equation}
since for each $n \in \bbN$ the support of ${\bf x}_n$ is contained in $[n+1,3n+3]$. 
Moreover,
\begin{align}
& \|\bsD_{\bsA}^{}{\bf x}_n\|^2_{L^2(\bbR;\cH)}   \no \\
& \quad = \|\chi_n\|^{-2}_{L^2(\bbR; dt)} \left \|\chi^{\prime}_n(\cdot-2n-2)x_n 
+ \chi_n(\cdot-2n-2) A(\cdot)x_n \right \|^2_{L^2(\bbR;\cH)}    \no \\
& \quad \leq 2\|\chi_n\|^{-2}_{L^2(\bbR; dt)} \|\chi^{\prime}_n(\cdot-2n-2)x_n \|^2_{L^2(\bbR;\cH)}    
\no \\
& \qquad + 2\|\chi_n\|^{-2}_{L^2(\bbR; dt)} 
\left \| \chi_n(\cdot-2n-2) A(\cdot)x_n \right \|^2_{L^2(\bbR;\cH)}  \no \\
& \quad \leq 2\|\chi_n\|^{-2}_{L^2(\bbR; dt)} \|\chi^{\prime}_n \|^2_{L^2(\bbR; dt)} 
+ 2 \sup_{t \in [n+1, 3n+3]} \|A(t)x_n\|_{\cH}^2     \no \\
& \quad \leq  2\|\chi_n\|^{-2}_{L^2(\bbR; dt)} \|\chi^{\prime}_n \|^2_{L^2(\bbR; dt)} 
+ 4 \|A_+ x_n\|_{\cH}^2 + 4\sup_{t \in [n+1, 3n+3]} \|B(t)x_n\|_{\cH}^2   \no \\
& \quad =  2\|\chi_n\|^{-2}_{L^2(\bbR; dt)} \|\chi^{\prime}_n \|^2_{L^2(\bbR; dt)} 
+ 4 \|A_+ x_n\|_{\cH}^2     \no \\
&\quad  
+4\sup_{t \in [n+1, 3n+3]} \|B(t)(A_++iI_\cH)^{-1}(A_++iI_\cH)x_n\|_{\cH}^2  \no\\
& \quad \leq  2\|\chi_n\|^{-2}_{L^2(\bbR; dt)} \|\chi^{\prime}_n \|^2_{L^2(\bbR; dt)} 
+ 4 \|A_+ x_n\|_{\cH}^2     \no \\
& \quad + 4 \|(A_+ + i I_{\cH})x_n\|_{\cH}^2
\sup_{t \in [n+1, 3n+3]} \|B(t)(A_+ + i I_{\cH})^{-1}\|^2_{\cB(\cH)}\no\\
& \quad \leq  2\|\chi_n\|^{-2}_{L^2(\bbR; dt)} \|\chi^{\prime}_n \|^2_{L^2(\bbR; dt)} 
+ 4 \|A_+ x_n\|_{\cH}^2     \no \\
&\qquad + 8 \big[\|A_+ x_n\|_{\cH}^2 
+ 1\big]\sup_{t \in [n+1, 3n+3]}\|B(t)(A_+ + i I_{\cH})^{-1}\|_{\cB(\cH)}^2.
\end{align}

Since
\begin{equation}
\|\chi_n\|^{-2}_{L^2(\bbR; dt)} \|\chi^{\prime}_n \|^2_{L^2(\bbR; dt)} 
\underset{n \to \infty}{\longrightarrow} 0,
\end{equation}
by construction, and
\begin{align}
& \sup_{t \in [n+1, 3n+3]}\|B(t)(A_+ + i I_{\cH})^{-1}\|_{\cB(\cH)}   \no \\
& \quad = \sup_{t \in [n+1, 3n+3]}\|B(t) (A_- + i I_{\cH})^{-1} 
(A_- + i I_{\cH}) (A_+ + i I_{\cH})^{-1}\|_{\cB(\cH)}  \no \\
& \quad \leq \sup_{t \in [n+1, 3n+3]}\|B(t) (A_- + i I_{\cH})^{-1}\|_{\cB(\cH)}
\|(A_- + i I_{\cH}) (A_+ + i I_{\cH})^{-1}\|_{\cB(\cH)}   \no \\
& \quad \leq C \sup_{t \in [n+1, 3n+3]}\|B(t) (A_- + i I_{\cH})^{-1}\|_{\cB(\cH)}
\underset{n \to \infty}{\longrightarrow} 0,    \lb{4.60} 
\end{align}
by hypothesis  (cf.\ \eqref{2.7A}), and the fact that $\dom(A_-) 
= \dom(A_+)$ implies $\|(A_- + i I_{\cH}) (A_+ + i I_{\cH})^{-1}\|_{\cB(\cH)} \leq C$,  
applying the closed graph theorem. Thus, one obtains 
\begin{equation}
\|\bsD_{\bsA}^{}{\bf x}_n\|^2_{L^2(\bbR;\cH)} \underset{n \to \infty}{\longrightarrow} 0.  \lb{4.61} 
\end{equation}

At this point we mention the following sufficient condition for $T$ to be an upper-semi-Fredholm 
operator, to be found in \cite{Wo59} (where properties of $T^*T$ are exploited). 
For convenience of the reader we briefly state it together with its short proof:  
\begin{align}
& \text{Let $T$ be densely defined and closed in $\cK$. Then $T$ is {\bf not} upper-semi-Fredholm} 
\no \\  
& \quad \text{if there exists a sequence $\{y_n\}_{n\in\bbN} \subset \cK$, 
$\|y_n\|_{\cK} = 1$, $n \in\bbN$, such that}   \no \\
& \quad \text{$y_n \underset{n\to\infty}{\longrightarrow} 0$ weakly and 
$\|T y_n\|_{\cK} \underset{n\to\infty}{\longrightarrow} 0$.}    \lb{4.62}
\end{align}
For the proof of \eqref{4.62} it suffices to note that if $T$ is upper semi-Fredholm, then there exists 
$S \in \cB(\cK)$ and $K \in \cB_{\infty}(\cK)$ such that  $ST = I_{\cK} + K$ on $\dom(T)$ 
(cf.\ \cite[Ch.\ 7]{Sc02}). Then one has $\|STy_n \| \underset{n \to \infty}{\longrightarrow} 0$, and,  
 in view of compactness of $K$, 
$\|ST y_n\|_{\cK} = \|(I_{\cK} + K) y_n\|_{\cK} \underset{n \to \infty}{\longrightarrow} 1$, a contradiction. 

In view of \eqref{4.57} and \eqref{4.61}, the latter property implies that $\bsD_{\bsA}^{}$ is not 
an upper-semi-Fredholm, let alone a Fredholm operator. 
Similarly, passing to adjoint operators, $\bsD_{\bsA}^{}$ is not a lower-semi-Fredholm 
operator (using the fact that the adjoint of a closed, densely defined, upper-semi-Fredholm 
operator is lower-semi-Fredholm and employing \eqref{2.14}).

Finally, considering
$\{{\bf \widetilde x}_n\}_{n \in \bbN} \subset \dom(\bsD_{\bsA}^{})$ defined by
\begin{equation}
{\bf \widetilde x}_n(t)= \|\chi_n\|^{-1}_{L^2(\bbR; dt)} \chi_n (t+2n+2) x_n, \quad n \in \bbN,
\end{equation}
and noting $\|{\bf \widetilde x}_n \|_{L^2(\bbR;\cH)}=1$, $n \in \bbN$, and
\begin{equation}
{\bf \widetilde x}_n \underset{n \to \infty}{\longrightarrow} 0 \, \text{ weakly,} 
\end{equation}
entirely analogous arguments show that $0 \in \sigma (A_-)$ implies once more that 
$\bsD_{\bsA}^{}$ is not upper-semi-Fredholm (and again not lower-semi-Fredholm 
by passing to adjoints).
\end{proof} 

\begin{remark} \lb{r2.6a} 
Note that apart from a ``minor'' regularity of $\bsA$, implying, in
particular, \eqref{2.14}, and the property \eqref{4.60} (and its
analogue for $n \to -\infty$) no other properties of
$\bsD_\bsA^{}$ were used in the proof of Theorem \ref{t2.6}. Thus,
Theorem \ref{t2.6} holds in greater generality and complements
\cite[Theorem\ 3.12]{RS95}. However, for clarity of
exposition, here and in the sequel, we prefer to stick to
Hypothesis \ref{h2.1}, although, some of the results on
$\bsD_\bsA^{}$ might hold under milder assumptions on $\bsA$. 
$\Diamond$
\end{remark}

To the best of our knowledge, the ``only if'' characterization of the Fredholm property of 
$\bsD_\bsA^{}$ under Hypothesis \ref{h2.1} is new. In this context we also refer to 
\cite[Propositions 7.11, 7.12,  Corollary 7.13 ]{LT05} 
and \cite[Theorem\ 1.2]{LP08}. (In the context of stronger 
assumptions than those in Hypothesis \ref{h2.1} we also refer to \cite{CGPST13}.)  

Thus, if $0$ belongs to either of $\sigma(A_+)$ or $\sigma(A_-)$
then $\bsD_\bsA^{}$ ceases to be Fredholm. On the other hand,
under Hypothesis \ref{h2.1}, formulas \eqref{2.20}, \eqref{2.30}
and \eqref{2.37a} still hold, and this brings the spectral shift
function into our study of fine spectral properties
$\bsD_\bsA^{}$ in terms of those of $A_\pm$. 

In fact, Theorem \ref{t2.6} yields a complete description of the essential
spectrum of $\bsD_\bsA^{}$:  Define the essential spectrum of a
densely defined and closed operator $T$ in the complex, separable Hilbert space $\cK$ as
\begin{equation}
\sigma_{ess}(T) = \{\lambda \in \bbC \, | \, T-\lambda I_{\cK} \, \,
\text{is not Fredholm}\}.
\end{equation}

\begin{corollary} \lb{c2.8} 
Assume Hypothesis \ref{h2.1}. Then
\begin{equation}\lb{essential}
\sigma_{{\rm ess}}(\bsD_\bsA^{})= (\sigma(A_+) + i \, \bbR) \cup 
(\sigma(A_-) + i \, \bbR).
\end{equation}
\end{corollary}
\begin{proof} 
Observe that if $M_\mu$ is a unitary operator on $L^2(\bbR;\cH)$
defined by $(M_\mu f)(s)=e^{-i\mu s} f(s)$, $f \in L^2(\bbR;\cH)$, $\mu \in \bbR$, 
then clearly $M_{\mu}^{-1}(\bsD_\bsA^{} - \lambda \bsI) M_\mu = \bsD_\bsA^{}
-(\lambda + i\mu) \bsI$. 
Hence, the essential spectrum of $\bsD_\bsA^{}$ is invariant with
respect to shifts along the imaginary axis. Moreover, if $\lambda
\in \bbR$, then setting $(A-\lambda I_{\cH})_-:=A_--\lambda I_{\cH}$, 
and leaving $B(t)$ unchanged one obtains in \eqref{2.6} that 
$(A-\lambda I_{\cH})_+=A_+-\lambda I_{\cH}$,  
and hence Theorem \ref{t2.6} implies that
$\bsD_\bsA^{}-\lambda \bsI = \bsD_{\bsA-\lambda \bsI}^{}$ is Fredholm if only
if $\lambda \in \rho (A_+)\cap \rho (A_-)$. Thus, a combination of
the two observations above yields
\begin{equation}
\sigma_{ess}(\bsD_\bsA^{})=\{\lambda + i\mu \in \bbC \, | \, \lambda \in
\sigma (A_+), \, \mu \in \bbR \} \cup \{\lambda+i\mu \in \bbC \, | \, \lambda
\in \sigma (A_-), \, \mu \in \bbR \},
\end{equation}
and hence \eqref{essential}. 
\end{proof} 

For invariance of the spectrum (and sums of its subsets) of $\bsD_\bsA^{}$ (in fact, for more general 
evolution operators) with respect to translations parallel to $i \, \bbR$, we refer to 
\cite[Proposition~2.36]{CL99}.

The computation of the Fredholm index of $\bsD_\bsA^{}$ in terms of spectral shift functions 
is provided next.  

\begin{theorem} [\cite{GLMST11}] \lb{t2.7}
Assume Hypothesis \ref{h2.1} and suppose that $0 \in \rho(A_+)\cap\rho(A_-)$. Then $\bsD_\bsA^{}$ is a Fredholm operator in $L^2(\bbR; \cH)$ and the following equalities hold for its Fredholm index:
\begin{align}
\ind (\bsD_\bsA^{}) & = \xi(0_+; \bsH_2, \bsH_1)    \lb{2.38} \\  
& = \xi(0;A_+,A_-)     \lb{2.39} \\ 
& = {\tr}_{\cH}(E_{A_-}((-\infty,0))-E_{A_+}((-\infty,0)))     \lb{2.39a} \\ 
& = \pi^{-1} \lim_{\varepsilon \downarrow 0}\Im\big(\ln\big({\det}_{\cH}
\big((A_+ - i \varepsilon I_{\cH})(A_- - i \varepsilon I_{\cH})^{-1}\big)\big)\big),       \lb{2.40}
\end{align}
with a choice of branch of $\ln({\det}_{\cH}(\cdot))$ on $\bbC_+$ analogous to that in Theorem \ref{t2.4}. 
\end{theorem}

\begin{remark} \lb{r2.8} (Topological invariance.) One notes the 
remarkable independence of 
\begin{equation}
\tr_{L^2(\bbR;\cH)}\big((\bsH_2 - z \bsI)^{-1} 
- (\bsH_1 - z \bsI)^{-1}\big), \quad \xi(\lambda; \bsH_2, \bsH_1), 
\, \text{ and } \, \ind (\bsD_\bsA^{})    \lb{2.41}
\end{equation} 
of the details of the operator path $\{A(t)\}_{t\in\bbR}$. Indeed, only the asymptotes 
$A_{\pm}$ of the family $A(t)$, $t \in \bbR$, enter $\xi(\,\cdot\,; A_+,A_-)$ and hence by \eqref{2.19}, \eqref{2.37a}, and \eqref{2.39}, the quantities in 
\eqref{2.41}. This is sometimes dubbed topological invariance in the pertinent 
literature (see, e.g., \cite{BGGSS87}, \cite{Ca78}, \cite{GLMST11}, \cite{GS88}, 
and the references therein). $\Diamond$
\end{remark}

The remainder of this paper will primarily focus on the case where $\bsD_\bsA^{}$ ceases 
to be a Fredholm operator in $L^2(\bbR; \cH)$.

\section{The Resolvent and Semigroup Regularized Witten Indices} \lb{s3}

The Witten index of an operator $T$ in $\cH$ is a natural substitute for the 
Fredholm index of $T$ in cases where the operator $T$ ceases to have the 
Fredholm property. In particular, as shown in \cite{GS88}, the Witten index possesses stability 
properties with respect to additive perturbations, however, necessarily under considerably 
stronger hypotheses (very roughly speaking, relative trace class type perturbations) than in the 
case of Fredholm indices (were relatively compact perturbations can be handled). This notion 
became popular in connection with a variety of 
examples in supersymmetric quantum mechanics in the 1980's, and after 
briefly recalling some of the basic definitions we will derive new properties of 
the Witten index of $T$ in this section.

We start with the following facts on trace class properties of resolvent and semigroup 
differences (some being well-known, such as item $(i)$ below, cf.\ \cite[p.\ 178]{We80}): 

\begin{lemma} \lb{l3.1}
Suppose that $0 \leq S_j$, $j=1,2$, are nonnegative, self-adjoint operators 
in $\cH$. Then the following assertions $(i)$--$(iii)$ hold: \\
$(i)$ If $\big[(S_2 - z_0 I_{\cH})^{-1} - (S_1 - z_0 I_{\cH})^{-1}\big] \in \cB_1(\cH)$ for 
some $z_0 \in \rho(S_1)\cap \rho(S_2)$, then
\begin{equation}
\big[(S_2 - z I_{\cH})^{-1} - (S_1 - z I_{\cH})^{-1}\big] \in \cB_1(\cH) 
\, \text{ for all $z \in \rho(S_1)\cap \rho(S_2)$.}
\end{equation} 
$(ii)$ If  $\big[e^{- z S_2} - e^{- z S_1}\big] \in \cB_1(\cH)$ for $z$ in a neighborhood of 
zero in the open right half-plane, then 
\begin{equation}
\big[e^{- z S_2} - e^{- z S_1}\big] \in \cB_1(\cH), \quad \Re(z) > 0. 
\end{equation}  
Moreover, if $\big[e^{- t_0 S_2} - e^{- t_0 S_1}\big] \in \cB_1(\cH)$ for some $t_0 > 0$, then 
\begin{equation}
\big[e^{- t S_2} - e^{- t S_1}\big] \in \cB_1(\cH) \, \text{ for all $t \geq t_0$.}   \lb{3.3a}
\end{equation}
$(iii)$ If $\big[(S_2 - z_0 I_{\cH})^{-1} - (S_1 - z_0 I_{\cH})^{-1}\big] \in \cB_1(\cH)$ for some 
$($and hence for all\,$)$ $z_0 \in \rho(S_1)\cap \rho(S_2)$, then 
\begin{equation}
\big[e^{- t S_2} - e^{- t S_1}\big] \in \cB_1(\cH) \, \text{ for all $t > 0$.}   \lb{3.4a} 
\end{equation}
\end{lemma}
\begin{proof}
The proof of item $(i)$ follows from either of the elementary identities  
\begin{align}
& \big[(S_2 - z I_{\cH})^{-1} - (S_1 - z I_{\cH})^{-1}\big]   \no \\
& \quad = (S_1 - z_0 I_{\cH}) (S_1 - z I_{\cH})^{-1}
\big[(S_2 - z_0 I_{\cH})^{-1} - (S_1 - z_0 I_{\cH})^{-1}\big]     \no \\
& \qquad \times (S_2 - z_0 I_{\cH})(S_2 - z I_{\cH})^{-1},    \lb{3.5a} \\
& \quad = \big[I_{\cH} +(z - z_0)(S_1 - z I_{\cH})^{-1}\big] 
\big[(S_2 - z_0 I_{\cH})^{-1} - (S_1 - z_0 I_{\cH})^{-1}\big]      \lb{difresolv} \\
& \qquad \times \big[I_{\cH} + (z - z_0) (S_2 - z I_{\cH})^{-1}\big],  
\quad z, z_0 \in \rho(S_1)\cap \rho(S_2).     \no 
\end{align}
The proof of the first part of item $(ii)$ follows from
\begin{equation}
\big[e^{- 2z S_1} - e^{- 2z S_2}\big] 
= e^{-z S_1 } \big[e^{- z S_1} - e^{- z S_2}\big] 
+ \big[e^{- z S_1} - e^{- z S_2}\big] e^{-z S_2},  \quad \Re(z) >0,
\end{equation}
while the second part of $(ii)$ (which is proved, e.g., in \cite{GS88}) follows from the general fact 
(cf.\ \cite[Proposition\ 8.10.2]{Ya92}), 
\begin{align}
& \text{if $A$ and $B$ are self-adjoint in $\cH$ and $A, B \geq \varepsilon_0 I_{\cH}$,}   \no \\
& \quad \text{with $\big[B^{- \tau_0} - A^{- \tau_0}\big] \in \cB_1(\cH)$ for some 
$\varepsilon_0, \tau_0 > 0$,}    \no \\
& \quad \text{then $\big[B^{- \tau} - A^{- \tau}\big] \in \cB_1(\cH)$ for all $\tau \geq  \tau_0$.} 
\end{align}
(The last assertion also readily follows from a slight variation of Theorem \ref{t6.2} below; in addition, see \cite[Theorem 4]{PS09} and \cite[Lemma 6]{PS10}.) \\
To prove item $(iii)$ we recall the following fact (cf.\ \cite[Sect.\ 8.3]{Ya92}):  
\begin{align}
& \text{If $0 \leq A, B \in \cB(\cH)$, with $[B - A] \in \cB_1(\cH)$, and 
$f \in C^1(\bbR)$ such that}     \no \\
& \quad \text{for some finite complex measure $d m$, $f'$ admits the representation} \no \\  
& \quad f'(\lambda) = \int_{\bbR} dm(t) \, e^{i t \lambda}, \;\; |m|(\bbR) < \infty,  
\, \text{ then } \, [f(B) - f(A)] \in \cB_1(\cH).    \lb{3.9a} 
\end{align}
In particular, this applies to $f \in \cS(\bbR)$ (the Schwartz class) and hence to 
$f \in C_0^{\infty}(\bbR)$ (and more generally, if $f'$ lies in the $L^2$-based Sobolev space, 
$f' \in W^{\alpha,2}(\bbR)$ for some $\alpha > 1/2$, cf., \cite[Corollary~7.9.4]{Ho83}). As a special case of \eqref{3.9a} we now suppose that 
$I_{\cH} \geq A_1^{-1}, B_1^{-1} \in \cB(\cH)$ with $\big[B_1^{-1} - A_1^{-1}\big] \in \cB_1(\cH)$. 
Introducing 
\begin{equation}
g(\lambda,t) = \begin{cases} e^{- 1/\lambda}, & \lambda \in [0, 2/t], \\ 0, & \lambda \leq 0, 
\end{cases}  \quad g(\, \cdot\, , t) \in C_0^{\infty}(\bbR) \text{ real-valued}, \; t>0, 
\end{equation}
one obtains that
\begin{equation}
g\big((t B_1)^{-1},t\big) - g\big((t A_1)^{-1},t\big) = \big[e^{- t B_1} - e^{- t A_1}\big] \in \cB_1(\cH), 
\quad t > 0,  
\end{equation}
employing the fact that $\sigma\big((t A_1)^{-1}\big) \cup \sigma\big((t B_1)^{-1}\big) \subseteq [0, 1/t]$, 
and identifying $f$ with $g$, $A$ with $A_1^{-1}$, and $B$ with $B_1^{-1}$ in \eqref{3.9a}. 
Finally, identifying $S_1 + I_{\cH}$ with $A_1$ and $S_2 + I_{\cH}$ with $B_1$ yields \eqref{3.4a} 
(multiplied by the common factor $e^{-t}$, $t > 0$). 
\end{proof}

We note that the converse to Lemma \ref{l3.1}\,$(iii)$ is false as will be illustrated in 
Example \ref{e3.6}. We also remark that \eqref{3.4a}, in fact, extends to the entire open right 
half-plane as will be proven in Theorem \ref{t3.9}. 

In addition to \eqref{difresolv} we also note that if $S_2$ is relatively trace class 
with respect to $S_1$, that is, if 
\begin{equation}
\dom(S_2) \supseteq \dom(S_1) \, \text{ and } \, 
(S_2 - S_1) (S_1 - z I_{\cH})^{-1} \in \cB_1(\cH), \quad z \in \rho(S_1),  
\end{equation}
then the difference of the resolvents of $S_2$ and $S_1$ can be written in the form
\begin{align} \lb{difresolvrel}
& (S_2 - z I_{\cH})^{-1} - (S_1 - z I_{\cH})^{-1} =  - (S_1 - z I_{\cH})^{-1} 
\big[(S_2 -  S_1) (S_2 - z_0 I_{\cH})^{-1}\big]     \no \\
& \hspace*{4.95cm} \times (S_2 - z_0 I_{\cH}) (S_2 - z I_{\cH})^{-1}  \\
& \quad = - (S_1 - z I_{\cH})^{-1} \big[(S_2 -  S_1) (S_2 - z_0 I_{\cH})^{-1}\big] 
\big[-I_{\cH} + (z - z_0) (S_2 - z I_{\cH})^{-1}\big].   \no 
\end{align}

\begin{remark} \lb{r3.2} 
Thus, the trace class property of 
$\big[(S_2 - z I_{\cH})^{-1} - (S_1 - z I_{\cH})^{-1}\big]$ for all $z \in
\rho(S_1)\cap \rho(S_2)$, is guaranteed by that property being satisfied
for just one point in $\rho(S_1) \cap \rho(S_2)$. However, the analog for 
semigroups does not hold. In other words, if 
$\big[e^{- t_0 S_1} - e^{- t_0 S_2}\big] \in \cB_1(\cH)$ for some $t_0 > 0$, 
it can of course happen that 
$\big[e^{- t S_1} - e^{- t S_2}\big] \notin \cB_1(\cH)$ for $0< t < t_0$. We describe 
an elementary example next: Let $M$ be a self-adjoint operator in $\cH$ 
with purely discrete spectrum 
$\sigma(M) = \{n\}_{n \in \bbN}$, such that the multiplicity of the $n$th eigenvalue equals $n$. Introduce
\begin{equation}
T_1 = M, \; T_2 = (M + I_{\cH}), S_1 = \ln(T_1), \; S_2 = \ln(T_2).
\end{equation} 
Then
\begin{equation}
n \big[(n+1)^{-\beta} - n^{-\beta}\big] + \beta n^{\beta} \underset{n \to \infty} 
= \Oh\big(n^{-1 - \beta}\big), \quad \beta > 0, 
\end{equation}
yields 
\begin{equation}
\big[T_1^{- \beta} - T_2^{- \beta}\big] \in \cB_1(\cH) \, 
\text{ if and only if $\beta > 1$,}
\end{equation}
and hence one infers that 
\begin{equation}
\big[e^{- t S_1} - e^{-t S_2}\big] =
\big[e^{- t \ln(T_1)} - e^{-t \ln(T_2)}\big] = 
\big[T_1^{- t} - T_2^{- t}\big] \in \cB_1(\cH) \, 
\text{ if and only if $t > 1$.}
\end{equation}
$\hfill \Diamond$
\end{remark}

At this point we introduce resolvent and semigroup regularized Witten indices 
as follows: 

\begin{definition} \lb{d3.3} 
Let $T$ be a closed, linear, densely defined operator in $\cH$. \\ 
$(i)$ Suppose that for some $($and hence for all\,$)$ 
$z \in \bbC \backslash [0,\infty)$,  
\begin{equation} 
\big[(T^* T - z I_{\cH})^{-1} - (TT^* - z I_{\cH})^{-1}\big] \in \cB_1(\cH).   \lb{3.1} 
\end{equation}  
Then one introduces the resolvent regularization 
\begin{equation}
\Delta_r(T, \lambda) = (- \lambda) \tr_{\cH}\big((T^* T - \lambda I_{\cH})^{-1}
- (T T^* - \lambda I_{\cH})^{-1}\big), \quad \lambda < 0.       \lb{3.2} 
\end{equation} 
The resolvent regularized Witten index $W_r (T)$ of $T$ is then defined by  
\begin{equation} 
W_r(T) = \lim_{\lambda \uparrow 0} \Delta_r(T, \lambda),      \lb{3.3}
\end{equation}
whenever this limit exists. \\
$(ii)$ Suppose that for some $t_0 > 0$  
\begin{equation}
\big[e^{- t_0 T^* T} - e^{- t_0 TT^*}\big] \in \cB_1(\cH).    \lb{3.4} 
\end{equation} 
Then $\big(e^{-t T^*T} - e^{-t TT^*}\big) \in \cB_1(H)$ for all $t >t_0$ 
$($cf.\ \eqref{3.3a}$)$, and one introduces the semigroup regularization 
\begin{equation}
\Delta_s(T, t) = \tr_{\cH}\big(e^{-t T^*T} - e^{-t TT^*}\big), \quad t > 0.     \lb{3.5}   
\end{equation} 
The semigroup regularized Witten index $W_s (T)$ of $T$ is then defined by  
\begin{equation} 
W_s(T) = \lim_{t \uparrow \infty} \Delta_s(T, t),       \lb{3.6}
\end{equation}
whenever this limit exists.
\end{definition} 

Here, in obvious notation, the subscript ``$r$'' (resp., ``$s$'') indicates the use of 
the resolvent (resp., semigroup or heat kernel) regularization. More generally, 
one could replace the particular pair of self-adjoint operators $(TT^*,T^*T)$ 
in $\Delta_r$, $\Delta_s$ by a general pair of self-adjoint operators 
$(S_2,S_1)$ in $\cH$ and study the correspondingly defined Witten indices 
$W_r(S_2,S_1)$, $W_s(S_2,S_1)$. Since we will eventually focus primarily on the 
Witten index of $\bsD_\bsA^{}$, we refrain from introducing this generalization.  

\begin{remark} \lb{r3.4} 
In general (i.e., if $T$ is not Fredholm), $W_r(T)$ (resp., $W_s(T)$) is not necessarily 
integer-valued; in fact, it can take on any prescribed real number. As a concrete 
example, we mention the two-dimensional  magnetic field system discussed by 
Aharonov and Casher \cite{AC79} which demonstrates that that the 
resolvent and semigroup regularized Witten indices have the meaning of 
(non-quantized) magnetic flux $F \in \bbR$ which indeed can take on any 
prescribed real number (cf.\ the analysis in \cite{An90a}, \cite{BGGSS87}).   
$\Diamond$
\end{remark} 

\begin{remark} \lb{r3.5} 
It was shown in \cite{BGGSS87} and \cite{GS88} that $W_r(T)$ 
(resp., $W_s(T)$) has stability properties with respect to additive perturbations 
analogous to the Fredholm index, as long as one replaces the familiar 
relative compactness assumption on the perturbation in connection with the 
Fredholm index, by appropriate relative trace class conditions in connection with 
the Witten index. In this context we also refer to \cite{CP74}. $\Diamond$
\end{remark} 

A natural question is whether the above definitions of the resolvent and 
semigroup regularized Witten index are in fact equivalent. Example \ref{e3.6} below 
demonstrates that this cannot hold in general. Therefore, the real issue is the following: 
suppose that both the semigroup 
(for sufficiently large $t>0$) and the resolvent differences are trace class, can one 
prove the equivalence of Definition \ref{d3.3}\,$(i)$ and \ref{d3.3}\,$(ii)$? 
We cannot answer this question at the moment in full generality but refer to  
Remark \ref{r3.7}. 

\begin{example} \lb{e3.6}
Let $A$ and $B$ be self-adjoint operators in $\cH$, introduce the $2 \times 2$ 
block matrix operators in $\cH \oplus \cH$ by
\begin{align}
\begin{split} 
T &= \begin{pmatrix} 0 & B \\ A & 0 \end{pmatrix}, \quad \dom(T) = \dom(A) \oplus \dom(B), \\
T^* &= \begin{pmatrix} 0 & A \\ B & 0 \end{pmatrix}, \quad \dom(T) = \dom(B) \oplus \dom(A), 
\end{split} 
\end{align} 
such that
\begin{align}
\begin{split} 
T^* T &= \begin{pmatrix} A^2 & 0 \\ 0 & B^2 \end{pmatrix}, \quad \dom(T^*T) 
= \dom\big(A^2\big) \oplus \dom\big(B^2\big), \\
T T^* &= \begin{pmatrix} B^2 & 0 \\ 0 & A^2 \end{pmatrix}, \quad \dom(T) 
= \dom\big(B^2\big) \oplus \dom\big(A^2\big), 
\end{split} 
\end{align} 
and hence
\begin{align}
& (T^*T - z I_{\cH \oplus \cH})^{-1} - (TT^* - z I_{\cH \oplus \cH})^{-1}    \no \\
& \quad = \big[(A^2 -z I_{\cH})^{-1} - (B^2 - z I_{\cH})^{-1}\big] 
\begin{pmatrix} I_{\cH} & 0 \\ 0 & - I_{\cH}\end{pmatrix}, \quad 
z \in \rho\big(A^2\big) \cap \rho\big(B^2\big),    \lb{3.A} \\
& e^{-t T^*T} - e^{-t TT^*} = \big[e^{-t A^2} - e^{-t B^2}\big] 
\begin{pmatrix} I_{\cH} & 0 \\ 0 & - I_{\cH}\end{pmatrix}, \quad t \geq 0.    \lb{3.B} 
\end{align}
In the remainder of this example, we use following concrete choices for $A$ and $B$,
\begin{align}
\begin{split} 
& A = [-  \Delta + V]^{1/2}, \quad 0 \leq V \in C_0^{\infty} (\bbR^n), \quad 
B = (- \Delta)^{1/2}, \\ 
& \dom(A) = \dom(B) = W^{1,2}(\bbR^n), \; n \in \bbN, 
\end{split} 
\end{align}
with $\cH = L^2(\bbR^n; d^n x)$, $n \in \bbN$, and to avoid trivialities, we assume that there exists an open $($nonempty\,$)$ 
ball in $\bbR^n$ on which $V$ is strictly positive. Then it is known $($cf., e.g., 
\cite[Theorem~21.4]{Si79}$)$ that 
\begin{equation}
\big[e^{-t A^2} - e^{-t B^2}\big] \in \cB_1\big(L^2(\bbR^n; d^n x)\big), 
\quad t \geq 0, \; n \in \bbN.     \lb{3.C} 
\end{equation}
Next, we investigate the resolvent difference of $A^2$ and $B^2$ and hence employ the known 
resolvent equation $($employing $I = I_{L^2(\bbR^n; d^n x)}$ for brevity\,$)$,
\begin{align}
& \big(A^2 - z I\big)^{-1} - \big(B^2 - z I\big)^{-1}     \\
& \quad = - \big(B^2 - z I\big)^{-1} v 
\big[I + v \big(B^2 - z I\big)^{-1} v\big]^{-1} v \big(B^2 - z I\big)^{-1}, \quad 
z \in \bbC \backslash [0,\infty),     \no  
\end{align}
where $V = v^2$, $0 \leq v \in C_0^{\infty}(\bbR^n)$. In particular, 
\begin{equation}
\big(A^2 - z I\big)^{-1} - \big(B^2 - z I\big)^{-1} = - C(z)^* C(z), \quad 
z < 0,     \lb{3.D}
\end{equation}
where we introduced 
\begin{equation}
C(z) = \big[I + v \big(B^2 - z I\big)^{-1} v\big]^{-1/2} v \big(B^2 - z I\big)^{-1}, \quad 
z \in \bbC \backslash [0,\infty). 
\end{equation}
Since 
\begin{equation}
\big[I + v \big(B^2 - z I\big)^{-1} v\big]^{\pm 1/2} \in \cB\big(L^2(\bbR^n; d^n x)\big), 
\quad z < 0, 
\end{equation}
one concludes that given $p \geq 1$, $z < 0$, 
\begin{align} 
C(z) \in \cB_p\big(L^2(\bbR^n; d^n x)\big) \, \text{if and only if } \, 
v \big(B^2 - z I\big)^{-1} \in \cB_p\big(L^2(\bbR^n; d^n x)\big). 
\end{align}
An application of \cite[Theorem~4.1, Proposition~4.4]{Si05} yields 
\begin{equation}
v \big(B^2 - z I\big)^{-1} \in \cB_2\big(L^2(\bbR^n; d^n x)\big), \; z < 0, \, 
\text{ if and only if $1 \leq n \leq 3$,}
\end{equation}
equivalently,
\begin{equation}
C(z) \in \cB_2\big(L^2(\bbR^n; d^n x)\big), \; z < 0, \, \text{ if and only if $1 \leq n \leq 3$.}
\end{equation}
Thus, by \eqref{3.D} $($and a standard application of resolvent identities\,$)$, 
\begin{equation}
\big[\big(A^2 - z I\big)^{-1} - \big(B^2 - z I\big)^{-1}\big] \notin 
\cB_1\big(L^2(\bbR^n; d^n x)\big), \quad z \in \bbC \backslash [0,\infty), \; n \geq 4.      \lb{3.E}
\end{equation}
Choosing $n \geq 4$, combining \eqref{3.A}, \eqref{3.B}, \eqref{3.C}, and \eqref{3.E} yields 
explicit examples  where the semigroup difference of $T^*T$ and $TT^*$ is trace class, 
but their resolvent difference is not. 
\end{example}

\begin{remark} \lb{r3.7} 
Assuming that both the semigroup and the resolvent differences are trace class, then the 
existence of the limit in \eqref{3.6} implies that in \eqref{3.3} by an Abelian theorem for 
Laplace transforms (see, e.g., \cite[Corollary 1a, p.\ 182]{Wi41}). The opposite
implication, however, is not obvious. If $0 \leq S_j$, $j=1,2$, are nonnegative and self-adjoint 
in $\cH$, and  
\begin{equation}
\big[(S_2 - z_0 I_{\cH})^{-1} - (S_1 - z_0 I_{\cH})^{-1}\big] \in \cB_1(\cH) \, 
\text{ for some } \, z_0 \in \rho(S_1) \cap \rho(S_2), 
\end{equation}
then  by Theorem \ref{t3.9} below, $\big(e^{-\cdot S_2}-e^{-\cdot S_1}\big)$ extends to the right
half-plane as a $\cB_1(\cH)$-valued analytic function, and thus
${\tr}_{\cH}\big(e^{-\cdot S_2}-e^{-\cdot S_1}\big)$ extends to the right
complex half-plane as an analytic function too. If 
${\tr}_{\cH}\big(e^{-\cdot S_2}-e^{-\cdot S_1}\big)$ is bounded in some 
sector around the half-line $[0,\infty)$, then by  a Tauberian theorem 
(cf., e.g., \cite[Theorem\ 2.6.4\,(b)]{ABHN01}) one obtains that \eqref{3.3}
implies \eqref{3.6}. This holds, in particular, if 
\begin{equation}
\sup_{t > 0}\frac{1}{t} \int_{0}^{t} |\xi(s; S_2,S_1)|\,ds < \infty,
\end{equation}
see the proof of Theorem \ref{t3.10} below. 

We also note that the following equivalence, 
\begin{align}
\begin{split} 
& \lim_{t \downarrow 0 } {\tr}_{\cH}\big[e^{- t S_2} - e^{- t S_1}\big] \, 
\text{ exists if and only if } \\
& \quad   
\lim_{\lambda \downarrow - \infty} (- \lambda)  
{\tr}_{\cH}\big[(S_2 - \lambda I_{\cH})^{-1} - (S_1 - \lambda I_{\cH})^{-1}\big]
\, \text{ exists}, 
\end{split} 
\end{align}
holds unrestrictedly by \cite[Theorem\ 2.6.4\,(a)]{ABHN01}. 
Finally, we note that it can be shown in a rather straightforward manner, 
and by a different method (see the proof of \cite[Proposition\ 2.1]{BW02}), that
\begin{equation}
\limsup_{t \to\infty}\f{1}{t} \ln \big\|e^{-t S_2}-e^{-t S_1}\big\|_{\cB_1(\cH)} = 0.  
\end{equation}  
$\hfill \Diamond$
\end{remark}

Before we continue this circle of ideas, we recall the known consistency between 
the Fredholm and Witten index whenever $T$ is Fredholm:

\begin{theorem} $($\cite{BGGSS87}, \cite{GS88}.$)$  \lb{t3.8} 
Suppose that $T$ is a Fredholm operator in $\cH$. \\
$(i)$ If \eqref{3.1} holds, then the resolvent regularized Witten index $W_r(T)$ 
exists, equals the Fredholm index, $\ind (T)$, of $T$, and
\begin{equation} 
W_r(T) =  \ind (T) = \xi(0_+; T T^*, T^* T).    \lb{3.7}
\end{equation}
$(ii)$ If \eqref{3.4} holds, then the semigroup regularized Witten index $W_s(T)$ 
exists, equals the Fredholm index, $\ind (T)$, of $T$, and
\begin{equation} 
W_s(T) =  \ind (T) = \xi(0_+; T T^*, T^* T).    \lb{3.8}
\end{equation}
\end{theorem}

Here $\xi(\, \cdot \,; S_2, S_1)$ denotes as usual the spectral shift function 
for the pair of self-adjoint operators $(S_2, S_1)$ in $\cH$ and we always use 
the convention,
\begin{equation}
\xi(\lambda ; S_2, S_1) = 0 \, \text{ for $\lambda < \inf(\sigma(S_1) \cup 
\sigma(S_2))$},     \lb{3.9} 
\end{equation}
whenever $S_j$, $j=1,2$, are bounded from below.

We mention, in passing, that the principal reason for the equalities in \eqref{3.7} 
is the fact that $T$ Fredholm is equivalent to $T^*T$ and $TT^*$ Fredholm, 
which in turn is equivalent to the existence of $\varepsilon_0 > 0$ such that 
\begin{equation}
\inf\big(\sigma_{ess} (T^*T)\big) = \inf\big(\sigma_{ess} (TT^*)\big) 
= \varepsilon_0 > 0.
\end{equation}
This implies the existence of norm convergent Laurent expansions of the 
resolvents of $(T^*T - z I_{\cH})^{-1}$ and $(TT^* - z I_{\cH})^{-1}$, respectively, 
around $z=0$ (see, e.g., \cite[Sect.\ 10.2]{Ba85}, \cite[Sect.\ III.6.5]{Ka80}) whose 
principal part consists of 
$P_{T^*T}(\{0\}) z^{-1}$ and $P_{TT^*}(\{0\}) z^{-1}$ (with $P_S(\{0\})$ the Riesz 
or spectral projection of the self-adjoint operator $S$ in $\cH$ associated with 
the isolated spectral point $\{0\}$; we do not exclude the possibility that  
$P_S(\{0\}) = 0$). Thus, one obtains 
\begin{align}
\begin{split}
\ind(T) = \ind(T^*T) &= \tr_{\cH}\big(P_{T^*T}(\{0\})\big) 
- \tr_{\cH}\big(P_{TT^*}(\{0\})\big)   \\ 
&= \xi(0_+; TT^*, T^*T),   
\end{split} 
\end{align}
since (cf.\ \eqref{3.9}) 
\begin{equation}
\xi(\lambda; TT^*, T^*T) = \begin{cases} \xi(0_+; TT^*, T^*T)  
& \text{for a.e.\ $\lambda \in (0,\varepsilon_0)$,}   \\
0 & \text{for $\lambda < 0$,} 
\end{cases} 
\end{equation}
as $\xi(\, \cdot \,; TT^*, T^*T)$ is piecewise constant a.e.\ below the infimum of the essential 
spectrum of $TT^*$ (reps., $T^*T$) by \cite[Proposition\ 8.2.8 and p.\ 300]{Ya92} and the 
fact that nonzero eigenvalues of $T T^*$ and $T^* T$ coincide and also their multiplicities 
coincide (cf.\ \cite{De78}). 

In the remainder of this section we develop a Laplace transform approach
to the study of $\xi(\,\cdot\,, S_2,S_1)$, still assuming that $0 \leq S_j$, $j=1,2$, 
are self-adjoint nonnegative in $\cH$. We will distinguish between the resolvent comparable case \eqref{difresolv} and the relatively trace class case 
\eqref{difresolvrel} since the trace norm asymptotics as $|z| \to \infty$ and 
$|z| \to 0$ of the difference of resolvents behaves differently in each of these cases.

For $\varphi \in (0, \pi/2)$ we introduce the sector  
\begin{equation} 
S_{\varphi} = \{z \in \bbC \,|\, |\arg (z)|<\varphi\},    \lb{3.39}
\end{equation} 
and for notational simplicity we frequently use the abbreviations 
\begin{align}
R_j (z) &= (S_j - z I_{\cH})^{-1}, \quad z \in \rho(S_j), \; j=1,2,    \\
\Delta_R(S_2,S_1;z) &= (S_2 - z I_{\cH})^{-1} - (S_1 - z I_{\cH})^{-1}. 
\end{align}

First, we note that by \eqref{difresolv}
\begin{align} \lb{estimrescomp}
& \|\Delta_R(S_2,S_1;z)\|_{\cB_1(\cH)} =\begin{cases}{\rm O} (1),& \arg (z)={\rm const}, \; 
|z| \to \infty,  \\
 {\rm O}(|z|^{-2}),& \arg (z)={\rm const}, \; |z| \to 0,
\end{cases}   \no \\
& \quad \text{in the resolvent comparable case,}
\end{align}
while by \eqref{difresolvrel}, 
\begin{align} \lb{estimrescomp1}
& \|\Delta_R(S_2,S_1;z)\|_{\cB_1(\cH)} = {\rm
O}\left(|z|^{-1}+|z|^{-2}\right),     \quad 
\arg (z)={\rm const}\neq 0,   \no \\
& \quad \text{$|z| \to \infty$ or $|z| \to 0$, \, in the relatively trace class case.}
\end{align}
Unfortunately, the  estimate  \eqref{estimrescomp} is rather poor at infinity
since we did not obtain any decay there.

To overcome this problem in the resolvent comparable case we will 
now pass to a related expression with
somewhat better estimates. First, one notes that $\Delta_R (S_2,S_1; \, \cdot\,)$ is a 
$\cB_1(\cH)$-valued function analytic in $\bbC \backslash  [0,\infty)$. The 
analyticity  of $\Delta_R(S_2,S_1;\, \cdot \,)$ can be proved in several ways. For
instance, it suffices to note that the set of rank-one operators
is contained in $\cB(\cH)=(\cB_1(\cH))^*$ and separates
$\cB_1(\cH)$  (see, e.g., \cite[p. 335]{GLMST11}). Then the
analyticity follows from \cite[Theorem A7]{ABHN01}, see also
\cite{AN00}. Therefore, one can differentiate $\Delta_R$ in the
trace norm and the trace norm of the derivative can be estimated
as
\begin{align} \lb{resolsquare}
& \big\|\big(\Delta_R (S_2,S_1;z)\big)' \big\|_{\cB_1(\cH)} = 
\big\|R_2(z)^2 - R_1(z)^2\big\|_{\cB_1(\cH)}   \no \\
& \quad = \|R_2(z)[R_2(z) - R_1(z)] - [R_1(z) - R_2(z)]R_1(z)\|_{\cB_1(\cH)} \no\\
& \quad = {\rm O}\left(|z|^{-3} + |z|^{-1}\right ), \quad \arg (z)= {\rm
const}\neq 0, \; |z| \to \infty \; \text{or} \; |z| \to 0.     
\end{align}

The following result contains an important improvement over Lemma \ref{l3.1}\,$(iii)$: 

\begin{theorem} \lb{t3.9}
Suppose that $0 \leq S_j$, $j=1,2$, are nonnegative, self-adjoint operators 
in $\cH$ and assume that 
\begin{equation} 
\big[(S_2 - z_0 I_{\cH})^{-1} - (S_1 - z_0 I_{\cH})^{-1}\big] \in \cB_1(\cH)  
\, \text{ for some } \, z_0 \in \rho(S_1) \cap \rho(S_2).   
\end{equation} 
Then
\begin{equation}
\big[e^{- z S_2} - e^{- z S_1}\big] \in \cB_1(\cH), \quad \Re(z) >0, 
\end{equation}
and 
\begin{equation}
\tr_{\cH}\big(e^{- \cdot S_2} - e^{- \cdot S_1}\big) \, \text{ is analytic in $\Re(z) >0$.} 
\end{equation} 
Moreover, for each $\varphi \in (0,\pi/2)$, there exists
$C_\varphi>0$ such that
\begin{equation} \lb{3.42}
\big\|e^{- z S_2} - e^{- z S_1}\big\|_{\cB_1(\cH)} \leq C_\varphi
(|z|^{-1}+|z|), \quad  z \in S_{\varphi}.
\end{equation}
\end{theorem}
\begin{proof}
Let us first assume that $S_2$ and $S_1$ are $\cB_1(\cH)$ resolvent 
comparable. Let
\begin{equation}
F(z):= - z \big[e^{- z S_2} - e^{- z S_1}\big], \quad \Re(z) >0,
\end{equation}
and fix $\gamma \in (0,\pi/2)$ and $\varepsilon \in
(0,(\pi/2)-\gamma)$.  Define a clockwise oriented path
$\Gamma=\Gamma(\gamma,\delta)$ consisting of three pieces
$\Gamma_0, \Gamma_{\pm}$:
\begin{equation}
\Gamma_{\pm}= \big\{re^{i\pm \gamma} \,\big|\, r \geq \delta \big\},  \quad  
\Gamma_0= \big\{\delta e^{i\theta} \,\big|\, \gamma \leq \theta \leq 2\pi -
\gamma\big\}.
\end{equation}
We note the well-known representation
\begin{equation} \lb{semigroup}
e^{- z S_j} = \frac{1}{2\pi i}\int_{\Gamma} d \lambda \, e^{-\lambda z} 
R_j(\lambda),  \quad \Re(z) >0, \; j=1,2, 
\end{equation}
where the integral converges absolutely in the uniform topology of 
$\cB(\cH)$
and does not depend on $\delta$ and $\gamma$ by Cauchy's theorem.
Then integrating the above equality by parts one obtains
\begin{equation}
z e^{- z S_j}=\frac{1}{2\pi i} \int_{\Gamma} d\lambda \, e^{-\lambda z} 
R_j(\lambda)^2,  \quad \Re(z) > 0, \; j=1,2, 
\end{equation}
and thus after subtraction
\begin{equation} \lb{representf}
F(z)=\frac{1}{2\pi i}\int_{\Gamma} d\lambda \, e^{-\lambda z} 
(\Delta_R(S_2,S_1))'(\lambda), \quad \Re(z) > 0,
\end{equation}
with the integral converging absolutely in the uniform topology of
$\cB(\cH)$ (although, for our purpose, the strong topology would be
sufficient). Equality \eqref{representf} can also be checked 
directly by taking the Laplace transform on both sides and using
the uniqueness theorem for Laplace transforms. One notes that the
function
\begin{equation}
(\Delta_R (S_2,S_1))': \begin{cases} 
\bbC \backslash  [0,\infty) \to \cB_1(\cH), \\
z \mapsto R_2(z)^2 - R_1(z)^2,    
\end{cases} 
\end{equation}
is clearly measurable since it is analytic on 
$\bbC \backslash  [0,\infty)$ as mentioned above.

Due to the exponential factor in \eqref{representf}, the integral
in \eqref{representf} converges absolutely as a $\cB_1(\cH)$-valued
integral. Thus, it defines a $\cB_1(\cH)$-valued function analytic
in $\Re(z) > 0$.  To prove analyticity it suffices to observe
that due to \eqref{estimrescomp} one can differentiate under the
integral sign. (Alternatively, one can use Morera's theorem.) Since
taken in the uniform topology of $\cB(\cH)$, the integral in 
\eqref{representf} coincides
with $F(z)$, it yields the same result while considered in the 
$\cB_1(\cH)$-norm. Thus, $F$ is a $\cB_1(\cH)$-valued function, 
analytic in $\bbC\backslash  [0,\infty)$. We then proceed
with estimates of $\| F\|_{\cB_1(\cH)}$  in the right complex half-plane.
For this purpose we estimate the integrand in \eqref{representf} on
each of the pieces $\Gamma_0$, $\Gamma_{\pm}$ separately using
 \eqref{resolsquare}.
Our reasoning follows a classical argument given in the proof of
\cite[Theorem 2.6.1]{ABHN01} closely:

If $z \in S_{(\pi/2)-\gamma -\varepsilon}$ and $\lambda \in \Gamma$ then, since $-(\pi/2)+\varepsilon \leq \arg (z) \pm \gamma \leq (\pi/2)-\varepsilon$, we have
\begin{equation}
\Re(\lambda) z = r |z| \cos(\arg (z)\pm \gamma) \geq r|z|\sin(\varepsilon),
\end{equation}
and therefore, 
\begin{equation}
\| e^{-\lambda z} (\Delta_R(S_2,S_1))'(\lambda)\|_{\cB_1(\cH)} 
\leq e^{-r|z|\sin (\varepsilon)}\left(\frac{1}{r}+\frac{1}{r^3}\right), \quad z \in \Gamma_{\pm}.
\end{equation}
Next, choose $\delta=1/|z|$. 
Then on $\Gamma_0$ we have $|\lambda|=\frac{1}{|z|}e^{i\theta}$, $\theta \in (-\gamma,\gamma)$, 
and hence
\begin{align} \lb{gamma0}
\begin{split} 
& \bigg\|\int_{\Gamma_0} d\lambda \, e^{-\lambda z} (\Delta_R(S_2,S_1))'(\lambda) 
\bigg\|_{\cB_1(\cH)} \leq C_0 |z|^2\int_{\gamma}^{2\pi - \gamma} 
d\theta \, e^{-\cos(\arg (z) +\theta)}   \\
& \quad \leq C_0 \frac{2\pi}{e} |z|^2, 
\end{split}
\end{align}
for some $C_0 > 0$, independent of $\gamma, \varepsilon$ and $\delta$. 

If $\lambda \in \Gamma_{\pm}$ then $\lambda=r e^{\pm i \gamma}$, 
$r > 0$, and
\begin{align} \lb{gammapm}
&\bigg\|\int_{\Gamma_{\pm}} d\lambda \, e^{-\lambda z} (\Delta_R(S_2,S_1))'(\lambda)
\bigg\|_{\cB_1(\cH)}
\le\int_{1/|z|}^{\infty} dr \, e^{-r|z|\sin (\varepsilon)}
\bigg(\frac{1}{r}+\frac{1}{r^3}\bigg)        \no \\
& \quad \leq \bigg(\int_{1}^{\infty}\frac{du \, e^{-u\sin(\varepsilon)}}{u} + |z|^2 
\int_{1}^{\infty}\frac{du \, e^{-u\sin(\varepsilon)}}{u^3}\bigg)   \no \\
& \quad \leq \bigg(\int_{\sin(\varepsilon)}^1 \f{dw \, e^{-w}}{w} + \int_1^{\infty} \f{dw \, e^{-w}}{w} 
+ |z|^2 \int_{1}^{\infty}\frac{du}{u^3}\bigg)   \no \\
& \quad \leq \ln(1/\sin(\varepsilon)) + C_1 + 2^{-1} |z|^2, 
\end{align}
for some $C_1>0$, independent of $\gamma, \varepsilon$ and $\delta$. Hence, by 
\eqref{gamma0} and \eqref{gammapm} one concludes that
\begin{equation}
\|F(z)\|_{\cB_1(\cH)} \leq 
\bigg(\frac{C_0}{e} + \frac{1}{2 \pi}\bigg)|z|^2 + \bigg(\frac{\ln(1/\sin(\varepsilon))}{\pi} 
+ \f{C_1}{\pi}\bigg),
\quad z \in S_{(\pi/2)-\gamma-\varepsilon},
\end{equation}
for $\gamma \in (0,\pi/2)$ and $\varepsilon \in
(0,(\pi/2)-\gamma)$.  By arranging $\varphi=
(\pi/2)-\gamma-\varepsilon$ for appropriate $\gamma$ and $\varepsilon$, this implies \eqref{3.42}. 

If $S_2$ is relatively trace class with respect to
$S_1$ then one can repeat the above reasoning using merely
\eqref{semigroup} and omitting the integration by parts step.
\end{proof}

Due to its importance, we also offer an alternative proof based on real variable methods of the estimate \eqref{3.42} in Appendix \ref{sB}. 

\begin{remark} \lb{r3.10}
The result of Theorem \ref{t3.9} is sharp in the sense that it fails when $\varphi=\pi/2$. Indeed, consider 
two diagonal operators  $S_0=\{2\pi n\}_{n\geq 0}$ and $S_1:=\{2\pi n+\pi\}_{n\geq 0}$. 
Clearly, $e^{iS_0}=1$ and $e^{iS_1}=-1$. Thus, the difference $e^{iS_0}- e^{iS_1}=2 I_{\cH}$ 
is not even a compact operator (assuming $\dim(\cH) = \infty$). On the other hand, a direct 
computation yields that 
$\big[(S_0 + I_{\cH})^{-1} -(S_1 + I_{\cH})^{-1}\big] \in \cB_1(\cH)$. 

In this context we note that the $\varphi$-dependence of $C_{\varphi}$ in \eqref{3.42} 
is rather intricate (and definitely not uniform with respect to $\varphi$): Indeed, it is well-known that 
the trace ideal $\cB_1(\cH)$ possesses the so-called Fatou property (see, e.g., 
\cite[Theorem\ 2.7\,(d)]{Si05}). In particular, if a sequence of operators 
$\{T_n\}_{n\in\bbN} \subset \cB_1(\cH)$ converges to an operator $T \in \cB(\cH)$ weakly as 
$n \to \infty$, and $\sup_{n \in \bbN} \|T_n\|_{\cB_1(\cH)} < \infty$, then $T \in \cB_1(\cH)$. Since 
for every self-adjoint operator $S$ in $\cH$, 
$\lim_{\varepsilon \downarrow 0}\big\|e^{(\varepsilon+i)S} - e^{i S}\big\| = 0$, the difference 
$e^{(\varepsilon+i) S_0} - e^{(\varepsilon+i) S_1}$ converges in $\cB(\cH)$-norm to 
$e^{i S_0} - e^{i S_1}$ as $\varepsilon \downarrow 0$. Since the latter operator is not trace class 
(in fact, it is not even compact) if $\dim(\cH) = \infty$ by the above argument, one infers that 
$\sup_{\varepsilon \to 0} \big\|e^{(\varepsilon+i)S_0} - e^{(\varepsilon+i) S_1}\big\|_{\cB_1(\cH)}= \infty$, 
thus showing the absence of a uniform bound on $C_{\varphi}$. (The absense of a uniform bound on 
$C_{\varphi}$ can also be shown directly for the concrete pair of diagonal operators $S_1$ and $S_2$ 
above. However, we preferred an abstract argument which might apply in other situations as well.) 
$\Diamond$
 \end{remark}

It is easy to see that $\big(e^{- z S_1}\big)_{\Re(z) >0}$ and
$\big(e^{- z S_2}\big)_{\Re(z)>0}$ are semigroups which are bounded
and analytic in $\Re(z) > 0$. In the next theorem we express the trace of 
$\big[e^{- z S_2} - e^{- z S_1}\big]$, $\Re(z) > 0$, in terms of
the Laplace transform of $\xi(\,\cdot\,; S_2, S_1)$ and provide a
``heat semigroup'' formula for the Witten index. We denote 
\begin{equation} 
\Xi(r; S_2,S_1)=\int_{0}^{r}\xi(s; S_2, S_1)\,ds, \quad r >0. 
\end{equation}

\begin{theorem} \lb{t3.10}
Suppose that $0 \leq S_j$, $j=1,2$, are nonnegative, self-adjoint operators 
in $\cH$ and assume that 
\begin{equation} \lb{traceclassasump}
\big[(S_2 - z_0 I_{\cH})^{-1} - (S_1 - z_0 I_{\cH})^{-1}\big] \in \cB_1(\cH)  
\, \text{ for some } \, z_0 \in \rho(S_1) \cap \rho(S_2),     
\end{equation}
and that 
\begin{equation} \lb{integrabilityxi}
\int_{0}^{\infty}\frac{|\xi (s; S_2, S_1)|\, ds}{s + 1} < \infty.
\end{equation}
Then the following items $(i)$ and $(ii)$ hold: \\ 
$(i)$ For every $\Re(z) > 0$ one has
\begin{equation} \lb{laplacetrans}
{\tr}_{\cH}\big(e^{-z S_2} - e^{- z S_1} \big) = 
- z \int_{0}^{\infty} \xi (s; S_2, S_1)\, ds \, e^{-zs}.
\end{equation}
$(ii)$ If $0$ is a right Lebesgue point of $\xi (\, \cdot \,; S_2, S_1)$ $($cf.\ our discussion 
in Appendix \ref{sA}$)$, then
\begin{equation} \lb{witten}
\lim_{z \to \infty} {\tr}_{\cH} \big(e^{-z S_2} - e^{- z S_1} \big) 
= - \Xi'(0_+; S_2,S_1) = - \Lxi(0_+; S_2,S_1), 
\end{equation}
uniformly for $z$ in any sector $S_\varphi$, $\varphi \in (0,\pi/2)$. 
Moreover,
\begin{equation} \lb{derwitten}
\lim_{z \to \infty} \frac{d^n}{dz^n} {\tr}_{\cH} \big(e^{- z S_2} - 
e^{- z S_1}\big)=0,\quad n \in \bbN,
\end{equation}
uniformly for $z$ in any sector $S_\varphi$, $\varphi \in (0,\pi/2)$. 
\end{theorem}
\begin{proof} 
To prove item $(i)$, one notes that by Theorem \ref{t3.9}, 
$\big[e^{-\cdot S_2} - e^{-\cdot S_1}\big]$ is a $\cB_1(\cH)$-valued
function analytic in $\Re(z) > 0$. Since the trace is a
continuous functional on $\cB_1(\cH)$, 
${\tr}_{\cH} \big(e^{- \cdot S_2} - e^{- \cdot S_1}\big)$ is an
analytic function in $\Re(z) > 0$ as well. As
\eqref{laplacetrans} holds for $z > 0$, and both sides
of \eqref{laplacetrans} are functions analytic in $\Re(z) > 0$,
the identity \eqref{laplacetrans} is satisfied for $\Re(z) > 0$ by the
uniqueness theorems for analytic functions.

To justify item $(ii)$, one first observes that
\begin{equation}
\frac{1}{t} {\tr}_{\cH}\big(e^{- t S_2} - e^{- t S_1}\big) 
= - \int_{0}^{\infty} \xi (s; S_2, S_1)\, ds \, e^{-ts} 
= - \int_{0}^{\infty}d \, \Xi (s; S_2,S_1) \, e^{-ts}.
\end{equation}
By the hypothesis that $0$ is a right Lebesgue point of $\xi (\, \cdot \,; S_2, S_1)$, 
one obtains (cf.\ \eqref{5.7}) that 
\begin{equation}
\lim_{r \downarrow 0+} \frac{\Xi (r; S_2,S_1)}{r} = \Xi'(0_+; S_2,S_1) = \Lxi(0_+; S_2,S_1) 
\end{equation} 
exists and then an Abelian theorem for Laplace transforms (cf.\ 
\cite[Theorem\ 1, p.\ 181]{Wi41}) yields that 
\begin{align} 
\begin{split} 
& \lim_{r \downarrow 0+} \frac{\Xi (r; S_2,S_1)}{r} = \Xi'(0_+;S_2,S_1)    \\ 
& \quad \text{ implies } \, 
- \lim_{t \to \infty} {\tr}_{\cH} \big(e^{- t S_2} - e^{- t S_1}\big) = \Xi'(0_+; S_2,S_1).
\end{split} 
\end{align}

Again, since $0$ is a right Lebesgue point of $\xi (\, \cdot \,; S_2, S_1)$, one 
concludes that 
\begin{equation}
\sup_{r \in (0,1]} \frac{1}{r}\int_{0}^{r} |\xi (s; S_2, S_1)|\, ds <\infty.
\end{equation} 
Next, from
\begin{equation}
\int_{0}^{\infty}\frac{|\xi (s; S_2, S_1)| \, ds}{s + 1} < \infty,
\end{equation}
it follows that
\begin{align}
& \sup_{t > 0}\frac{1}{t} \int_{0}^{t} |\xi (s; S_2, S_1)|\, ds   \no \\
& \quad \leq \sup_{t \in (0,1]}
\frac{1}{t}\int_{0}^{t} |\xi (s; S_2, S_1)|\, ds 
+ \sup_{t > 1}\frac{1}{t}\int_{1}^{t} |\xi (s; S_2, S_1)|\, ds    \no \\
& \quad \leq \sup_{t \in (0,1]}
\frac{1}{t}\int_{0}^{t} |\xi (s; S_2, S_1)|\, ds + 2 \int_{1}^{t} 
\frac{|\xi (s; S_2, S_1)| \, ds}{s + 1}     \no \\
& \quad \leq \sup_{t \in (0,1]} \frac{1}{t}\int_{0}^{t} |\xi (s; S_2, S_1)|\, ds 
+ 2 \int_{1}^{\infty} \frac{|\xi (s; S_2, S_1)| \, ds}{s + 1} < \infty.
\end{align}
Hence,
\begin{align}
& \big|{\tr}_{\cH} \big( e^{- z S_2}  - e^{- z S_1} \big)\big| 
\leq |z| \int_{0}^{\infty} |\xi (t; S_2, S_1)|\, dt \, e^{-t \Re(z)}   \no \\
&\quad = |z| \Re(z) \int_{0}^{\infty} dt \bigg( t e^{-t \Re(z)} \frac{1}{t} 
\int_{0}^{t} |\xi (s; S_2, S_1)|\, ds\bigg)    \no \\
&\quad \leq C |z| \Re(z) \int_{0}^{\infty} dt \, t e^{-t \Re(z)} \leq C \frac{|z|}{\Re(z)},
\end{align}
for some $C>0$.  Thus the function $f(z):= {\tr}_{\cH} \big( e^{- z S_2}
- e^{- z S_1} \big)$, $\Re(z) > 0$, is bounded in every
sector $S_{\varphi}$, $\varphi \in(0,(\pi/2))$, and tends to the 
limit $\Xi'(0_+; S_2,S_1)$ along the real axis. Then by a standard
application of Vitali's theorem to the sequence  of functions
$\{f_n\}_{n\in\bbN} = \{f(n \cdot)\}_{n \in \bbN}$, analytic on
$S_{\varphi}\cap \{z \in\bbC \,|\, 1<|z|<2\}$, one concludes that
$\lim_{n \to \infty} f_n(z) = \Xi'(0_+; S_2,S_1)$ and $\lim_{n \to \infty} f_n^{(k)}(z)=0$ for every
$k \in \bbN$, 
uniformly in $z$ in compact subsets of $S_{\varphi}$ (see, e.g., \cite[Proposition
2.6.3]{ABHN01}). Therefore, the limit in \eqref{witten} exists uniformly
in $S_{\varphi'}, \varphi'< \varphi$. It remains to recall
that the choice of $\varphi \in(0,(\pi/2))$ was arbitrary.

Similarly, $\lim_{n \to \infty} f^{(k)}(z) = 0$ for every $k \in \bbN$ 
uniformly in any sector $S_\varphi$, $\varphi \in (0,(\pi/2))$.
Hence, differentiating  $\big[e^{-\cdot S_2} - e^{-\cdot S_1}\big]$ in the
trace norm, and using the fact that
 the result must coincide with the derivative
of $\big[e^{-\cdot S_2} - e^{-\cdot S_1}\big]$ in the uniform (or strong) operator 
topology, one obtains for every $k \in \bbN$,
\begin{align}
\lim_{z \to \infty} f^{(k)}(z)& = \lim_{z \to \infty} \frac{d^k}{dz^k}
{\tr}_{\cH} \big(e^{-z S_2} - e^{-z S_1}\big)\\
& = \lim_{z \to \infty}(-1)^{k+1} {\tr}_{\cH} \big(S_2^k e^{-z S_2} - S_1^k  e^{-z S_1}\big)=0,
\end{align}
uniformly in any sector $S_\varphi$, $\varphi \in (0,(\pi/2))$, implying \eqref{derwitten}.
\end{proof}

\begin{remark} \lb{r3.11} 
We recall in connection with the model operator $\bsD_\bsA^{}$ that 
condition \eqref{integrabilityxi} is satisfied by $\xi(\,\cdot\,;\bsH_2,\bsH_1)$ as 
recorded in Lemma \ref{l2.3}. $\Diamond$
\end{remark}

\section{The Witten Index of $\bsD_\bsA^{}$ 
in Terms of Spectral Shift Functions} \lb{s4}

This section is devoted to a derivation of the Witten index of $\bsD_\bsA^{}$.

Since our analysis in this section heavily relies on the use of right and left Lebesgue points 
of spectral shift functions, we refer once more to our brief collection of pertinent facts on 
Lebesgue points in Appendix \ref{sA}. 

We start with the following fact.

\begin{lemma} \lb{l4.1} 
Introduce linear operators $S$ and $T$ by 
\begin{align}
& S: \begin{cases}
L^1_{loc}(\bbR; d\nu) \to L^1_{loc}((0,\infty); d\lambda),    \\[1mm] 
f \mapsto \frac1{\pi}\int_0^{\lambda^{1/2}} d \nu \, (\lambda-\nu^2)^{-1/2} f(\nu), \quad \lambda>0,
\end{cases}   \\[2mm] 
& T: \begin{cases} 
L^1\big(\bbR;(1+\nu^2)^{-3/2}d\nu\big) \to L^1_{loc}((0,\infty); d\lambda),    \\[1mm] 
f \mapsto \lambda \int_{\bbR} d \nu \, (\nu^2+\lambda)^{-3/2} f(\nu),\quad\lambda > 0.
\end{cases}  
\end{align}
$(i)$ If $0$ is a right Lebesgue point for $f \in L^1_{loc}(\bbR; d\nu)$, then it is also right Lebesgue 
point for $S f$ and
\begin{equation} \lb{approx lim}
\LSf (0_+) = \frac12\Lf (0_+).
\end{equation}
$(ii)$ If $0$ is left and right Lebesgue point for $f\in L^1\big(\mathbb{R},(1+\nu^2)^{-3/2}d\nu\big)$, then
\begin{equation} 
\lim_{\varepsilon\downarrow0}(T f)(\varepsilon)= \Lf (0_+) + \Lf (0_-).
\end{equation} 
\end{lemma}
\begin{proof} $(i)$ Since $S 1=1/2$, we assume, without loss of generality, that $\Lf (0_+) = 0$. One then infers from Fubini's theorem that
\begin{align} 
& \bigg|\frac1h\int_0^h d \lambda \, (Sf)(\lambda)\bigg|
\leq \frac1{\pi h} \int_0^h d\lambda  \int_0^{\lambda^{1/2}}\frac{d \nu \, |f(\nu)|}{(\lambda-\nu^2)^{1/2}}  
\no \\
& \quad =\frac1{\pi h}\int_0^{h^{1/2}} d \nu \, |f(\nu)| \int_{\nu^2}^h\frac{d\lambda}{(\lambda-\nu^2)^{1/2}}  
=\frac{2}{\pi h}\int_0^{h^{1/2}} d \nu \, |f(\nu)| (h-\nu^2)^{1/2}   \no \\
& \quad \leq \frac{2}{\pi h^{1/2}}\int_0^{h^{1/2}} d\nu \, |f(\nu)|.
\end{align} 
Since by assumption $\Lf (0_+) = 0$, it follows that
\begin{equation} 
\frac{1}{h^{1/2}}\int_0^{h^{1/2}} d \nu \, |f(\nu)| \underset{h\downarrow 0}{\longrightarrow} 0, 
\end{equation} 
concluding the proof of item $(i)$.  \\
$(ii)$ Since $T \chi_{(0,\infty)} = T\chi_{(-\infty,0)}=1$, we assume, without loss 
of generality, that $0$ is a Lebesgue point for $f \in L^1\big(\bbR;(1+\nu^2)^{-3/2}d\nu\big)$, 
that $\Lf (0)=0$, and that $f$ is supported on $(0,\infty)$. 

Fix $\delta>0$ and select $\tau > 0$ such that
\begin{equation} \lb{our t def}
\frac1h \int_0^h d\nu \, |f(\nu)| < \delta, \quad 0 < h < \tau.
\end{equation}
For every $0 < h<\tau$, one has 
\begin{equation} 
(Tf)(h)=\int_0^h\frac{d \nu \, h^2f(\nu)}{(\nu^2+h^2)^{3/2}} 
+ \int_h^\tau \frac{d \nu \, h^2f(\nu)}{(\nu^2+h^2)^{3/2}}
+ \int_{\tau}^{\infty}\frac{d \nu \, h^2f(\nu)}{(\nu^2+h^2)^{3/2}}.     \lb{4.8}
\end{equation} 
Evidently,
\begin{equation} 
\bigg|\int_0^h\frac{d \nu \, h^2f(\nu)}{(\nu^2+h^2)^{3/2}}\bigg|\leq\frac1h\int_0^h d \nu \, |f(\nu)|  
\underset{h\downarrow 0}{\longrightarrow} 0, 
\end{equation} 
and
\begin{equation} 
\bigg|\int_{\tau}^{\infty}\frac{d \nu \, h^2f(\nu)}{(\nu^2+h^2)^{3/2}}\bigg| 
\leq h^2\int_{\tau}^{\infty}\frac{d \nu \, |f(\nu)|}{\nu^3}\underset{h\downarrow 0}{\longrightarrow} 0.
\end{equation} 
Next, we estimate the second term on the right-hand side of \eqref{4.8}. For brevity, set
\begin{equation} 
F(t)=\int_0^t d\nu \, |f(\nu)|, \quad t>0.
\end{equation} 
Then, 
\begin{equation} 
\bigg|\int_h^{\tau}\frac{d \nu \, h^2f(\nu)}{(\nu^2+h^2)^{3/2}}\bigg| 
\leq\int_h^{\tau}\frac{d \nu \, h^2|f|(\nu)}{\nu^3}
=\frac{h^2F(\nu)}{\nu^3}\bigg|_h^{\tau}+3\int_h^{\tau}\frac{d \nu \, h^2F(\nu)}{\nu^4}.
\end{equation} 
By \eqref{our t def}, one concludes that $0\leq F(\nu)\leq\delta\nu$ for $0\leq\nu\leq \tau$. It follows that 
\begin{equation} 
\bigg|\int_h^{\tau}\frac{d \nu \, h^2f(\nu)}{(\nu^2+h^2)^{3/2}}\bigg|
\leq 2\delta+3h^2\delta\int_h^{\tau}\frac{d\nu}{\nu^3}\leq \frac{7}{2}\delta.
\end{equation} 
It follows that
\begin{equation} 
\limsup_{h\downarrow0}|(T f)(h)| \leq \frac{7}{2}\delta.
\end{equation} 
Since $\delta$ is arbitrarily small, the assertion in item $(ii)$ follows.
\end{proof}

In the next result we use the notation ${\rm Hol}(\mathbb{C}\backslash [0,\infty))$ to denote holomorphic functions on the set $\mathbb{C}\backslash [0,\infty)$.

\begin{lemma} \lb{l4.2} Introduce the linear operator  $\mathbf{T}$ by 
\begin{equation} 
\mathbf{T}: \begin{cases} 
L^1\big(\bbR; (1+\nu^2)^{-3/2}d\nu\big) \to {\rm Hol} (\bbC \backslash [0,\infty))  \\[1mm] 
f \mapsto -z\int_{\mathbb{R}} d \nu \, (\nu^2-z)^{-3/2} f(\nu), \quad z \in \mathbb{C}\backslash [0,\infty).
\end{cases} 
\end{equation} 
If $0$ is a left and a right Lebesgue point for $f\in L^1\big(\bbR; (1+\nu^2)^{-3/2}d\nu\big)$, then with 
$S_{\varphi}$, $ \varphi \in (0,\pi/2)$, the sector introduced in \eqref{3.39}, 
\begin{equation} 
\lim_{z\to0, \, z\in \bbC \backslash S_{\varphi}}(\mathbf{T}f)(z)=\Lf(0_+) + \Lf(0_-).
\end{equation} 
\end{lemma}
\begin{proof} The functions $\mathbf{T}\chi_{(0,\infty)}$ and $\mathbf{T}\chi_{(-\infty,0)}$ are holomorphic on $\mathbb{C}\backslash [0,\infty)$ and are equal to $1$ on $(-\infty, 0)$. By the uniqueness theorem for holomorphic functions one concludes that $\mathbf{T}\chi_{(0,\infty)}=\mathbf{T}\chi_{(-\infty,0)}=1$. Since the assertion is linear with respect to $f$, we may assume, without loss of generality, that $0$ is Lebesgue point for $f$ and that $\Lf (0)=0$. This implies $\Lmodf (0)=0$.  \\
Next, one notes the existence of a constant $C_{\varphi} > 0$ such that
\begin{equation} 
\Big|\frac1{(\nu^2-z)^{3/2}}\Big|\leq\frac{C_{\varphi}}{(\nu^2+|z|)^{3/2}}, 
\quad z \in \bbC \backslash S_{\varphi}.
\end{equation} 
In fact, one can choose $C_{\varphi} = \{2/[1 - \cos(\varphi)]\}^{3/4}$. 
It follows that 
\begin{equation} 
|(\mathbf{T}f)(z)|\leq C_{\varphi} \, (T |f|)(|z|),\quad z \in \bbC \backslash S_{\varphi},
\end{equation} 
with $T$ the operator defined in Lemma \ref{l4.1}. By Lemma \ref{l4.1}\,$(ii)$, 
$(T |f|)(|z|) \underset{z \to 0}{\longrightarrow} 0$. 
\end{proof}

Combining Lemmas \ref{l4.1} and \ref{l4.2} and Theorem \ref{t3.10} one obtains the principal result of this section:

\begin{theorem} \lb{t5.5} Assume Hypothesis \ref{h2.1} and let $\varphi\in(0,\pi/2)$. In addition, suppose 
that $0$ is a right and a left Lebesgue point of $\xi(A_+, A_-)$. Then 
\begin{align} 
& \text{$0$ is a right Lebesgue point of $\xi(\, \cdot \, ; \bsH_2,\bsH_1)$ and}    \no \\
& \quad \Lxi(0_+; \bsH_2, \bsH_1) = [\Lxi(0_+; A_+,A_-) + \Lxi(0_-; A_+, A_-)]/2.   \lb{55a}
\end{align}
Moreover, the following assertions hold: 
\begin{align} 
& W_r(\bsD_\bsA) = \lim_{z \to 0, \, z \in \bbC \backslash  S_{\varphi}}z\tr_{L^2(\bbR;\cH)}
\big((\bsH_2 - z \bsI)^{-1} - (\bsH_1 - z \bsI)^{-1}\big),   \lb{55b} \\[1mm] 
& W_r(\bsD_\bsA)=- \lim_{z \to 0, \, z \in \bbC \backslash  S_{\varphi}} 
\f{z}{2} \int_{\bbR} \frac{\xi(\nu; A_+, A_-) \, d\nu}{(\nu^2 - z)^{3/2}},     \lb{55c} \\[1mm] 
& W_r(\bsD_\bsA)=\lim_{z \to \infty, \, z \in S_{\varphi}} 
\tr_{L^2(\bbR;\cH)}\big(e^{- z \bsH_1} - e^{-z \bsH_2}\big),    \lb{55d} \\[1mm] 
& W_r(\bsD_\bsA) = [\Lxi(0_+; A_+,A_-) + \Lxi(0_-; A_+, A_-)]/2 = W_s(\bsD_\bsA).  \lb{55e} 
\end{align} 
\end{theorem} 
\begin{proof} 
Rewiting \eqref{2.37a} in the form,
\begin{equation} 
\xi(\lambda; \bsH_2, \bsH_1) = \frac{1}{\pi}\int_0^{\lambda^{1/2}}
\frac{d \nu \, [\xi(\nu; A_+,A_-) + \xi(-\nu; A_+,A_-)]}{(\lambda-\nu^2)^{1/2}},  
\quad \lambda > 0,     \lb{55f}
\end{equation} 
an application of Lemma \ref{l4.1}\,$(i)$ to the particular function $f(\nu) = \xi(\nu,A_+,A_-)+\xi(-\nu,A_+,A_-)$, 
$\nu > 0$, yields \eqref{55a}. \\
The identity,
\begin{equation}
z \tr_{L^2(\bbR;\cH)} \big((\bsH_2 - z \bsI)^{-1} - (\bsH_1 - z \bsI)^{-1}\big) = 
- \f{z}{2} \int_{\bbR} \frac{\xi(\nu; A_+, A_-) \, d\nu}{(\nu^2 - z)^{3/2}}, 
\quad z \in \bbC \backslash [0,\infty),   \lb{55g} 
\end{equation}
follows from \eqref{2.19} and \eqref{2.24} (the latter were proven in \cite{GLMST11}). The limit on the right-hand side of  \eqref{55c} exists by Lemma \ref{l4.2}. Consequently, by \eqref{55g}, also the limit on the right-hand side of  \eqref{55b} exists and equals that in \eqref{55c}. 
By definition \eqref{3.3}, this limit equals $W_r(\bsD_\bsA)$.  \\
The first equality in \eqref{55e} follows from Lemma \ref{l4.2} and \eqref{55c}; the second  
equality in \eqref{55e} follows from Theorem \ref{t3.10} and Remark \ref{r3.11} combined with \eqref{55a}. \\
Finally, \eqref{55d} follows from \eqref{55e} and Theorem \ref{t3.10}.
\end{proof}

In Remark \ref{r3.4} we recalled the fact that the Witten index, in general, can be any prescribed 
real number. Next we show that this also applies to the special case of the Witten 
index of $\bsD_\bsA^{}$:

\begin{remark} \lb{r5.6}
There exist pairs of self-adjoint operators $(A_+, A_-)$ in $\cH$ such that $0$ is a 
Lebesgue point of $\xi(\,\cdot\,; A_+,A_-)$ and that 
\begin{align}
\begin{split}
W_r(\bsD_\bsA^{}) &= W_s(\bsD_\bsA^{}) = \Lxi(0_+; \bsH_2, \bsH_1) 
= \Lxi(0; A_+,A_-)   \\
&= \text{{\bf any} prescribed real number.}     \lb{5.37} 
\end{split}
\end{align}
A simple concrete example is the following: Consider $A_{\pm} \in \cB(\cH)$ 
with $[A_+ - A_-] \in \cB_1(\cH)$, and introduce 
\begin{equation}
B(t) = \f{e^t}{e^t + 1} [A_+ - A_-] \in \cB_1(\cH), \quad t \in \bbR.
\end{equation}
Then 
\begin{equation}
B'(t) = \f{e^t}{(e^t + 1)^2} [A_+ - A_-] \in \cB_1(\cH), \quad t \in \bbR,
\end{equation}
shows that Hypothesis \ref{h2.1} is satisfied. Moreover, since by 
\cite[Proposition\ A.8]{GPS08} {\it any} integrable function 
$\xi \in L^1(\bbR; dt)$ of compact support arises as the 
spectral shift function for a pair of bounded, self-adjoint operators 
$(A_+,A_-)$ in $\cH$ with $[A_+ - A_-] \in \cB_1(\cH)$, Theorem \ref{t5.5} implies  
\eqref{5.37}. $\Diamond$
\end{remark}

It remains to illustrate the nature of the assumption that  $0$ is a right and left Lebesgue 
point of $\xi(\,\cdot\,\, ; A_+, A_-)$. We start with simplest case where $A_{\pm}$ have discrete 
spectrum in an open neighborhood of $0$:

\begin{remark} \lb{r5.7}
By \cite[Proposition\ 8.2.8]{Ya92}, if $A_{\pm}$ have discrete spectrum in an open neighborhood 
of $0$, then $\xi(\, \cdot \,; A_+,A_-)$ has a right and left limit at any point of this open neighborhood 
and in particular any point in that open neighborhood is a right and a left Lebesgue point of 
$\xi(\, \cdot \,; A_+,A_-)$. $\Diamond$
\end{remark}

Next, to turn the discussion to absolutely continuous spectra of $A_{\pm}$ near $0$, we refer 
the reader to our quick review of some basic facts of self-adjoint rank-one perturbations of a 
self-adjoint operator $A_0$ in $\cH$ in Appendix \ref{sA}. Given the preparations in 
Appendix \ref{sA}, we can continue our discussion of the right and left Lebesgue point 
hypothesis on $\xi(\, \cdot \, ; A_+, A_-)$ under the assumption of the presence of purely absolutely 
continuous spectrum of $A_{\pm}$ in a neighborhood of $0$ as follows:

\begin{proposition} \lb{p5.8}
There exist pairs of bounded self-adjoint operators $(A_+, A_-)$ in $\cH$ such that $A_+ - A_-$ 
is of rank one, and $A_{\pm}$ both have purely absolutely continuous spectrum in a fixed 
neighborhood $(-\varepsilon_0, \varepsilon_0)$, for some $\varepsilon_0 > 0$, yet 
$\xi(\, \cdot \,; A_+,A_-)$ may or may not have a right and/or a left Lebesgue point at $0$. 
\end{proposition}
\begin{proof}
To illustrate this fact, we employ the rank-one formalism summarized in Appendix \ref{sA} 
(cf.\ \eqref{5.40}--\eqref{5.59}) and identify $A_- = A_0$, 
$A_+ = A_{\alpha_0}$ in \eqref{5.40} for an 
appropriate $\alpha_0 > 0$ to be fixed subsequently. In particular, we choose $A_0$ to be the operator of multiplication with the variable $\lambda$ in the Hilbert space $L^2(\bbR; d\omega_0(\lambda))$, where $\supp (d \omega)$ is compact (and 
hence $A_0$ is bounded). 
We employ \eqref{5.53} by choosing 
$\xi_0 = \xi(\, \cdot \, ; A_+, A_-) \in L^\infty(\bbR; d\nu)$ of compact support, satisfying 
\begin{equation} 
c_0 \leq \xi_0 \leq d_0 \, \text{ on } \, (-\varepsilon_0, \varepsilon_0) 
\, \text{ for some } \, 0 < c_0 < d_0 < 1.   \lb{5.33a} 
\end{equation} 
Then $A_- = A_0$ has purely absolutely continuous spectrum in $(-\varepsilon_0, \varepsilon_0)$ 
by \eqref{5.52}. Next, we choose  
$\xi_0$ to be smooth on $\bbR \backslash \{0\}$, but temporarily leave open the 
behavior of $\xi_0$ at $\nu = 0$ (except for the constraint \eqref{5.33a}). We note that by \eqref{5.43}, \eqref{5.44}, 
the choice of $\xi_0$ also implies a 
certain choice of $\alpha_0 > 0$. Employing the elementary fact (cf.\ 
\eqref{5.41}, \eqref{5.42}, \eqref{5.55}, \eqref{5.56}) that 
\begin{align}
\begin{split} 
& \Im(1 + \alpha_0 F_0(\nu + i \varepsilon)) = \alpha_0 \Im(F_0(\nu + i \varepsilon))    \\
& \quad = i \alpha_0 \exp(\pi (Q_{\varepsilon} \xi_0)(\nu)) \sin(\pi (P_{\varepsilon} \xi_0)(\nu)),  
\quad \nu \in \bbR, \; \varepsilon > 0,    
\end{split} 
\end{align}
one notes that 
\begin{equation}
\lim_{\varepsilon \downarrow 0} (P_{\varepsilon} \xi_0)(\nu) = \xi_0(\nu) \in [c,d], \quad 
\nu \in (- \varepsilon_0, \varepsilon_0) \backslash \{0\},  
\end{equation}
and hence 
\begin{equation}
0 < \lim_{\varepsilon \downarrow 0} \sin(\pi (P_{\varepsilon} \xi_0)(\nu)) \, 
\text{ exists finitely for all $\nu \in (- \varepsilon_0, \varepsilon_0) \backslash \{0\}$}
\end{equation}
since by hypothesis $0 < c < d < 1$. By our hypotheses on $\xi_0$ also 
\begin{equation}
\lim_{\varepsilon \downarrow 0} (Q_{\varepsilon} \xi_0)(\nu) \, 
\text{ exists finitely for all $\nu \in (- \varepsilon_0, \varepsilon_0) \backslash \{0\}$}
\end{equation} 
as a consequence of \eqref{5.57}--\eqref{5.59}. Thus, by \eqref{5.48}--\eqref{5.50} also 
$A_+ = A_{\alpha_0}$ has purely absolutely continuous spectrum in 
$(- \varepsilon_0, \varepsilon_0)$ (we also note that the singular 
continuous spectrum of a self-adjoint operator is necessarily supported on an uncountable set). To 
conclude the proof it suffices to note that we largely left open the behavior of $\xi_0$ at $0$ (apart 
from \eqref{5.33a}) and hence we may or may not choose $0$ to be a right and/or left Lebesgue 
point of $\x_0$.   
\end{proof}

We add an example kindly sent to us by A.\ Poltoratski \cite{Po13}: 

\begin{remark} \lb{r5.9} 
Again, in the context of  rank-one perturbations, \eqref{5.40}, consider the following 
construction of $\xi_0 = \xi(\, \cdot \,; A_+, A_-)$. Consider the sequences $\{1/n\}_{n \in \bbN}$ 
and $\{- 1/n\}_{n \in \bbN}$ and their midpoints, and suppose $\xi_0$ jumps by $+1$ at 
$\pm 1/n$ and by $-1$ at the sequence of midpoints and equals zero outside $[-1, 1]$. Then 
$A_- = A_0$ will be the operator of multiplication by $\lambda$ in $L^2(\bbR; d \mu_0(\lambda))$  
where $d \mu_0$ now has point masses at $\pm 1/n$ (and we make the choices mentioned in the 
paragraph following \eqref{5.40}). Clearly, $\xi_0$ has neither a right nor a 
left Lebesgue point at $\lambda = 0$. $\Diamond$
\end{remark}

If one leaves the realm of rank-one perturbations, examples involving singular continuous spectral 
measures can very easily be found:

\begin{remark} \lb{r5.10} 
Suppose $A_{\pm}$ are of the type $A_- = A_0 \oplus 0$, $A_+ = A_0 \oplus B_1$, where $\cH$ 
decomposes into $\cH_0 \oplus \cH_1$ with $\cH_1 = \ell^2(\bbN)$ and $A_0$ is a generic example of an operator with singular continuous spectrum in $\cH_0$ and $B_1 \in \cB_1\big(\ell^2(\bbN)\big)$. 
Then $\xi (\, \cdot \, ;A_+, A_-) = \xi(\, \cdot \, ; B_1,0)$ and one can choose $B_1$ in such a way that 
the step function $\xi(\, \cdot \, ; B_1,0)$ does, or does not, have $0$ as a right Lebesgue point. $\Diamond$
\end{remark}

These considerations underscore the independence of the right/left Lebesgue point 
hypothesis on $\xi(\, \cdot \,; A_+, A_-)$ in Theorem \ref{t5.5} from our main Hypothesis \ref{h2.1}.

\section{The Witten Index of $\bsD_\bsA^{}$ in the Special Case $\dim(\cH) < \infty$} 
\lb{s5}

In this section we briefly treat the special finite-dimensional 
case, $\dim(\cH) < \infty$,  and explicitly compute the Witten index of 
$\bsD_\bsA^{}$ irrespective of whether or not $\bsD_\bsA^{}$ is a 
Fredholm operator in $L^2(\bbR; \cH)$.  

We denote by 
\begin{equation}
\#_{>} (S) \, \text{ (resp.\ $\#_{<} (S)$)}
\end{equation}
the number of strictly positive (resp., strictly negative) eigenvalues of a 
self-adjoint operator $S$ in $\cH$, counting multiplicity.   

In the special case $\dim(\cH) < \infty$, Hypothesis \ref{h2.1} considerably simplifies and one is then 
left with the following set of assumptions:

\begin{hypothesis} \lb{h4.1}
Suppose that $\cH$ is a complex Hilbert space with $\dim(\cH) < \infty$. \\
$(i)$ Assume $A_- \in \cB(\cH)$ is self-adjoint matrix in $\cH$. \\
$(ii)$ Suppose there exists a family of bounded self-adjoint matrices $\{B(t)\}_{t\in\bbR}$ 
that is locally absolutely continuous on $\bbR$ such that 
\begin{equation}  \lb{4.1}
\int_\bbR dt \, \big\|B'(t)\big\|_{\cB(\cH)} < \infty.
\end{equation}
\end{hypothesis}

\begin{theorem} \lb{t4.2}
Assume Hypothesis \ref{h4.1}. Then $\xi(\,\cdot\,; A_+,A_-)$ has a
piecewise constant representative on $\bbR$, the right limit  $\xi(0_+; \bsH_2, \bsH_1)$ exists, 
and $\xi(\,\cdot\, ; \bsH_2, \bsH_1)$ has a continuous 
representative on $(0,\infty)$. Moreover, the resolvent and semigroup regularized Witten 
indices $W_r(\bsD_\bsA^{})$ and $W_s(\bsD_\bsA^{})$ exist, and 
\begin{align}
W_r(\bsD_\bsA^{}) = W_s(\bsD_\bsA^{}) &= \xi(0_+; \bsH_2, \bsH_1)    \lb{4.25} \\
&= [\xi(0_+; A_+,A_-) + \xi(0_-; A_+,A_-)]/2    \lb{4.26} \\
&= \f{1}{2} [\#_{>} (A_+) - \#_{>} (A_-)] - 
\f{1}{2} [\#_{<} (A_+) - \#_{<} (A_-)].     \lb{4.27} 
\end{align}
In particular, in the finite-dimensional context, the Witten indices are either  
integer, or half-integer-valued.
\end{theorem}
\begin{proof}
We will focus exclusively on $W_r(\bsD_\bsA^{})$ as its equality with 
$W_s(\bsD_\bsA^{})$ is clear from Theorem \ref{t5.5} (cf.\ \eqref{55e}).

The fact that $n = \dim(\cH)<\infty$ shows that $\xi(\,\cdot\,; A_+,A_-)$ has a piecewise 
constant representative on $\bbR$ as $A_\pm$ have precisely 
$n$ real eigenvalues counting multiplicity. More precisely, denoting by 
$\lambda_j(A_{\pm})$ the eigenvalues of $A_\pm$ and by 
$m(A; \lambda)$, $\lambda \in \bbR$, the multiplicity function associated with the linear operator $A$ 
in $\cH$ (defining $m(A; \lambda) = 0$ if $\lambda$ is not an eigenvalue of $A$), one obtains that 
\begin{equation}
\xi(\lambda_+; A_+,A_-) - \xi(\lambda_-; A_+,A_-) 
= m(A_-; \lambda) - m(A_+; \lambda),   \quad \lambda \in\bbR   \lb{4.28} 
\end{equation} 
(cf., e.g., \cite[Proposition\ 8.2.8]{Ya92}). In addition, by \eqref{2.21} and \eqref{2.22}, 
\begin{equation}
\xi(\,\cdot\,; A_+,A_-) =0 \, \text{ for $\lambda < \inf(\sigma(A_+) \cup 
\sigma(A_-))$ and $\lambda > \sup(\sigma(A_+) \cup \sigma(A_-))$.} 
\lb{4.29} 
\end{equation}
Since every real number $\nu \in \bbR$ is a right and a left Lebesgue point of $\xi(\,\cdot\,; A_+,A_-)$, 
this applies in particular to $\nu = 0$ and hence shows that $0$ is a right Lebesgue point of 
$\xi(\,\cdot\, ; \bsH_2, \bsH_1)$ by Theorem \ref{t5.5}.

The definition of $g_z(\cdot)$ in \eqref{2.16} then implies, 
\begin{align}
\tr_\cH \big(g_z(A_{\pm})\big) 
& = \sum_{\lambda_j(A_{\pm})\in \sigma(A_{\pm})} 
\f{\lambda_j(A_{\pm})}{\big[\lambda_j(A_{\pm})^2 -z\big]^{1/2}}   \no \\ 
& = \sum_{0 \neq \lambda_j(A_{\pm})\in \sigma(A_{\pm})} 
\f{\lambda_j(A_{\pm})}{\big[\lambda_j(A_{\pm})^2 -z\big]^{1/2}}, \quad  
z\in\bbC\backslash [0,\infty),    \no \\
& \underset{z \to 0, \, z \notin [0,\infty)}{\longrightarrow} 
\sum_{0 \neq \lambda_j(A_{\pm})\in \sigma(A_{\pm})} 
\sgn(\lambda_j(A_{\pm}))   \no \\
& = \#_{>} (A_{\pm}) - \#_{<} (A_{\pm}).     \lb{4.30}
\end{align}
Consequently, $W_r(\bsD_\bsA^{})$ exists and \eqref{2.19} and \eqref{4.30} imply  
\begin{align}
& z \tr_{L^2(\bbR;\cH)}\big((\bsH_2 - z \bsI)^{-1}-(\bsH_1 - z \, 
\bsI)^{-1}\big) = \frac{1}{2} \tr_\cH \big(g_z(A_+)-g_z(A_-)\big) \no \\
& \quad \underset{z \to 0, \, z \notin [0,\infty)}{\longrightarrow} 
\f{1}{2} [\#_{>} (A_+) - \#_{>} (A_-)] - 
\f{1}{2} [\#_{<} (A_+) - \#_{<} (A_-)]    \no \\
& \quad = W_r(\bsD_\bsA^{}).     \lb{4.31} 
\end{align} 
Using \eqref{4.28} and \eqref{4.29} then yields 
\begin{align}
\xi(0_+; A_+, A_-) &= \#_{>} (A_+) - \#_{>} (A_-),    \lb{4.32} \\
\xi(0_-; A_+, A_-) &= \#_{<} (A_-) - \#_{<} (A_+),     \lb{4.33}
\end{align}
proving equality of $W_r(\bsD_\bsA^{})$ with \eqref{4.26} and \eqref{4.27}. 

To prove relation \eqref{4.25} one first notes that for a.e. $\lambda > 0$, 
$\xi(\lambda; \bsH_2, \bsH_1)$ equals (up to a constant multiple) the determinant of the 
($n\times n$) scattering matrix by the celebrated Birman--Krein formula (cf.\ \cite{BK62}, 
\cite{BY93}). In the current finite-dimensional setting, the scattering matrix is explicitly related 
to the $n \times n$ matrix-valued Jost functions which yields a continuous representative of 
$\xi(\lambda; \bsH_2, \bsH_1)$ on $(0,\infty)$ and the existence of the right limit 
$\xi(0_+; \bsH_2, \bsH_1)$. The latter fact is standard in the special scalar 
case $n=1$ (cf., e,g., \cite[Ch.\ XVII]{CS89}, \cite[Ch.\ 3]{Ma86}), but it is also well-known to 
extend to the case of matrix-valued Schr\"odinger operators (see, e.g., 
\cite[Ch.\ II]{AM63}, \cite{AW13}, \cite[Sect.\ XVII.4.4]{CS89}, \cite{MO82}, \cite{NJ55}, 
\cite{Ol85}). Employing \eqref{2.28} then yields for some fixed $\varepsilon > 0$, 
\begin{align} 
& z \tr_{L^2(\bbR;\cH)} \big((\bsH_2 - z \bsI)^{-1} - (\bsH_1 - z \, 
\bsI)^{-1}\big)
= - z \int_{[0, \infty)}  \frac{\xi(\lambda; \bsH_2, \bsH_1) \, 
d\lambda}{(\lambda -z)^2}    \no \\
& \quad = - z \int_{[0, \infty)}  \frac{\xi(0_+; \bsH_2, \bsH_1) \, 
d\lambda}{(\lambda -z)^2}   \no \\
& \qquad - z \int_{[0, \infty)}  \frac{[\xi(\lambda; \bsH_2, \bsH_1) 
- \xi(0_+; \bsH_2, \bsH_1)] \, 
d\lambda}{(\lambda -z)^2}   \no \\
& \quad = \xi(0_+; \bsH_2, \bsH_1) 
- z \int_{[0, \varepsilon]}  \frac{[\xi(\lambda; \bsH_2, \bsH_1) 
- \xi(0_+; \bsH_2, \bsH_1)] \, 
d\lambda}{(\lambda -z)^2}   \no \\
& \qquad - z \int_{[\varepsilon, \infty)}  \frac{[\xi(\lambda; \bsH_2, \bsH_1) 
- \xi(0_+; \bsH_2, \bsH_1)] \, 
d\lambda}{(\lambda -z)^2}   \no \\
& \quad = \xi(0_+; \bsH_2, \bsH_1) + \oh(1) \, 
\text{ as $z\to 0$, $z \notin [0,\infty)$.}      \lb{4.34}
\end{align}
Here we used continuity of $\xi(\,\cdot\,; \bsH_2, \bsH_1)$ in the 
$\lambda$-integral times $z$ over any interval $(0,\varepsilon]$ for fixed 
$\varepsilon > 0$, the existence of the right limit $\xi(0_+; \bsH_2, \bsH_1)$, 
and the fact that the $\lambda$-integral over 
$[\varepsilon,\infty)$ times $z$ obviously tends to zero as $z \to 0$, 
$z \notin [0,\infty)$. This proves \eqref{4.25}. The latter,  together with 
\eqref{4.31}, completes the proof of Theorem \ref{t4.2}. 
\end{proof}

\section{Homology and the Witten Index} \lb{s6}

Invariance of the Fredholm index under compact perturbations is key to understanding its topological interpretation in terms of K-theory. As the Witten
index is not invariant under (relatively) compact perturbations there can be 
no direct connection to K-theory. Somewhat surprisingly, the Witten index is related 
instead to cyclic homology, as we now explain by relating the approach of \cite{GS88}, the Carey--Pincus point of view in \cite{CP86}, and the recent (unpublished) thesis of Kaad (for the more advanced results in the thesis see \cite{Ka11}, \cite{Ka11a}).

In \cite{CP86} bounded non-Fredholm operators $S,S^*$ with the property
that the commutator $[S,S^*]$ is trace class are studied. An important tool exploited there is the Carey--Pincus principal function. It is related directly to the spectral shift function. Kaad's thesis provides us with the connection between \cite{CP86} and cyclic homology and, as we show at the end of this section, with the Witten index.

We start with a brief summary of Chapter 1 of Kaad's thesis: The point of departure for this discussion are Banach algebras $A$ and $J$ where $J$ is an ideal in $A$ (not necessarily closed). We let $\cB=A/J$. By the zeroth relative continuous cyclic homology group of the pair $(J,A)$ we will understand the quotient space $HC_0(J,A) = J/ \Im(b)$. Here $b : J \otimes A + A \otimes J\to J$ is an extension of the Hochschild boundary
and is given by
\begin{equation}
b : s\otimes a + a'\otimes t\mapsto sa-as + a't-ta',
\end{equation}
where $s,t\in J$ and $a,a'\in A$.
Note that as a topological vector space $HC_0(J,A)$ is non-Hausdorff in general. That is, 
the image of the extended Hochschild boundary $b$ is not necessarily closed in $J$.
We will also make use of the first cyclic homology group of $\cB$ denoted $HC_1(\cB)$.
In the book of Loday \cite{Lo98}
he introduces the chain complex
$C_m=\cA^{\otimes m+1}$, $m=0, 1, 2, \dots$, and the two boundary maps $b, B$ being 
respectively the Hochschild boundary and the Connes' boundary. These are defined on 
$C_m$ by:
\begin{align}
\begin{split}
b(a_0, a_1, \ldots, a_n) &= (a_0a_1, a_2, \ldots, a_n)
+\sum_{i=1}^{n-1} (-1)^i(a_0,a_1,\ldots, a_ia_{i+1}, \ldots, a_n) \\
& \quad +(-1)^n(a_na_0,a_1,\ldots, a_{n-1}),
\end{split}
\end{align}
with $b:C_m\to C_{m-1}$,
and
\begin{align}
\begin{split}
B(a_0,\ldots, a_n) &= \sum_{i=0}^n (-1)^{ni}(1,a_i,\ldots,a_i,a_0,\ldots,a_{i-1}) \\
& \quad +(-1)^{ni}(a_i,1,a_{i+1},\ldots,a_n,a_0,\dots, a_{i-1}),
\end{split}
\end{align}
where $(a_0, a_1, \ldots, a_n)\in C_m$ and $B:C_m\to C_{m+1}$.
Then it is not difficult to check that we have the relations $b^2=0=B^2,\ 0=bB+Bb$. As in 
\cite{Lo98} we can form a bi-complex
$B(\cA)$ with $B(\cA)_{p,q}= \cA^{\otimes^{q-p+1}}$ with total boundary $b+B$
and hence there are homology groups $\ker (b+B)/{\Im}(b+B)$ in each degree
(being those of periodic cyclic homology)
associated with the total boundary $b+B$. We understand $HC_1(\cB)$ as the first periodic 
cyclic homology group of $\cB$.
Let $\cE$ be the $C^*$-algebra generated by $S$ and $S^*$.
Let $J=\cE\cap \cL^1(\cH)$. Following Kaad, introduce the exact sequence
\begin{equation}
X : 0 \rightarrow J \stackrel{i}{\rightarrow} \cE\stackrel{q}{\rightarrow}\cB\rightarrow 0,
\end{equation}
where $i$ is the obvious inclusion and $q$ is the quotient map to $\cB=\cE/J$.
Kaad in his thesis proves the following result:

\begin{theorem}
The operator trace determines a well defined map on the zeroth continuous
relative cyclic homology group
\begin{equation}
{\rm Tr}_*: HC_0(J, \cE)\to \bbC.
\end{equation}
This can be extended to a map on $HC_1(\cB)$ using the connecting map
\begin{equation}
\partial_X: HC_1(\cB)\to HC_0(J, \cE),
\end{equation}
coming from the above exact sequence labelled $X$.
Notice that $\cB$ is a commutative algebra. The pair $q(S), q(S^*)$ defines a class 
$q(S)\otimes q(S^*)$ in $HC_1(\cB)$. This class maps to the commutator $[S,S^*]$ 
under $\partial_X$.
\end{theorem}
%
\subsubsection*{How the Witten index relates to this}
\newcommand{\bma}{\left(\begin{array}{cc}}
\newcommand{\ema}{\end{array}\right)}
We will adopt a more general viewpoint here than in the earlier parts of the paper.
We suppose we have a complex separable Hilbert space
 $\mathcal K$  and form $\mathcal K^{(2)}= \mathcal K\oplus \mathcal K$.
In the situation of the model operator of the introduction $\cK$ would be $L^2(\bbR;\cH)$. 
As in the introduction we assume we have a linear, closed, densely defined operator $T$ 
on $\mathcal K$ and that $T$ and its adjoint can be combined to form
a self adjoint unbounded operator $\mathcal D=\bma 0 & T^*\\
T & 0\ema$ on $\mathcal K^{(2)}$. We restrict to the case where $\lambda$ is in the 
intersection of the resolvent sets of $TT^*$ and $T^*T$ and we make the assumption that
\begin{equation}
- \lambda
\tr_{\cK}\big((T^*T - \lambda I_{\cK})^{-1} - (TT^* - \lambda I_{\cK})^{-1}\big), 
\quad \lambda < 0,   \lb{7.7}
\end{equation}
is finite. Let us make a change of notation and set $\mu^2=-\lambda$
so that the assumption of the previous equation becomes
$(1+\mu^{-2}T^*T)^{-1}- (1 + \mu^{-2}TT^*)^{-1}$
is trace class. It is clear from this formulation that we can think of $\mu$
as scaling $T$.
This suggests that we make the passage to the bounded picture
by writing 
\begin{equation} 
F^\mu_\cD= \mu^{-1}\cD\big(1+\mu^{-2}\cD^2\big)^{-1/2}.
\end{equation} 
Then
\begin{equation}
1-(F^\mu_\cD)^2= (1+\mu^{-2}\cD^2)^{-1} 
= \bma (1+\mu^{-2}T^*T)^{-1} & 0\\ 0 & (1+\mu^{-2}TT^*)^{-1}\ema.
\end{equation}
For ease of writing let
\begin{equation}
F^\mu_\cD=\bma 0 & S_\mu^*\\ S_\mu & 0 \ema.
\end{equation}
In other words,
\begin{equation}
S_\mu= \mu^{-1}T(1+\mu^{-2}T^*T)^{-1/2},
\end{equation}
and
\begin{equation}
1-(F^\mu_\cD)^2 = \bma 1-S_\mu^*S_\mu & 0 \\
0 & 1-S_\mu S_\mu^*\ema.
\end{equation}
Consequently, our assumption that
$(1+\mu^{-2}T^*T)^{-1}- (1+\mu^{-2}TT^*)^{-1}$ is trace class
translates in the bounded picture
to the assumption that the commutator $[S_\mu, S_\mu^*]$ is trace class.
The Witten index is thus calculating
$\lim_{\mu\to 0} {\rm tr}_\cK ([S_\mu, S^*_\mu])$ whenever this limit exists. 
This puts us into the framework of Kaad's analysis by letting $\mathcal E$ be the 
$C^*$-algebra generated by $S_1$ and $S_1^*$ and using the fact that this algebra contains
the operators $S_\mu, S_\mu^*$ for $\mu> 0$. Thus we see that the 
theorem quoted above from Kaad's thesis
implies  that the Witten index is given by a scaling limit of a functional defined
on $HC_1(\cB)$.
This point of view generalizes extensively and is the topic of \cite{CGK14} and \cite{CK14}.

\appendix
\section{Some Analytical Tools} \lb{sA}
\renewcommand{\theequation}{A.\arabic{equation}}
\renewcommand{\thetheorem}{A.\arabic{theorem}}
\setcounter{theorem}{0} \setcounter{equation}{0}

We briefly collect a variety of useful tools employed at various places in this paper. 

Since Theorem \ref{t3.10}, and particularly our analysis in Section \ref{s4} relies 
on the use of right and left Lebesgue points of spectral shift functions, we start by briefly 
recalling this notion.

\begin{definition} \lb{d5.1} 
Let $f \in L^1_{\loc} (\bbR; dx)$ and $h > 0$. \\
$(i)$ The point  $x \in \bbR$ is called a right Lebesgue point of $f$ if 
there exists an $\alpha_+ \in \bbC$ such that  
\begin{equation} 
\lim_{h \downarrow 0} \f{1}{h} \int_{x}^{x + h} dy \, |f(y) - \alpha_+| = 0.   \lb{5.1} 
\end{equation} 
One then denotes $\alpha_+ = \Lf(x_+)$. \\
$(ii)$ The point $x \in \bbR$ is called a left Lebesgue point of $f$ if 
there exists an $\alpha_- \in \bbC$ such that  
\begin{equation} 
\lim_{h \downarrow 0} \f{1}{h} \int_{x - h}^{x} dy \, |f(y) - \alpha_-| = 0.   \lb{5.2} 
\end{equation}
One then denotes $\alpha_- = \Lf(x_-)$. \\ 
$(iii)$ The point $x \in \bbR$ is called a Lebesgue point of $f$ if 
there exist $\alpha \in \bbC$ such that  
\begin{equation} 
\lim_{h \downarrow 0} \f{1}{2h} \int_{x - h}^{x + h} dy \, |f(y) - \alpha| = 0.   \lb{5.3} 
\end{equation} 
One then denotes $\alpha = \Lf(x)$. That is, $x \in \bbR$ is a Lebesgue point 
of $f$ if and only if it is a left and a right Lebesgue point and 
$\alpha_+ = \alpha_- = \alpha$. 
\end{definition} 

\begin{remark} \lb{r5.2}
$(i)$ We note that the definition \eqref{5.3} of a Lebesgue point of $f$ is not universally adopted. 
For instance, \cite[Sect.\ 18, p.\ 277--278]{HS65} defines $x_0$ to be a Lebesgue point 
of $f$ if 
\begin{equation}
\lim_{h \downarrow 0} \f{1}{h} \int_{0}^h dy \, |f(x + y) + f(x_0 - y) - 2 f(x)| = 0.   \lb{5.4} 
\end{equation}
The elementary example 
\begin{equation}
f(x;\beta) = \begin{cases} 0, & x<0, \\ \beta, & x=0, \\ 1, & x>0, \end{cases} 
\quad \beta \in \bbC,     \lb{5.5} 
\end{equation}
shows that $f(\, \cdot \, ; \beta)$ satisfies \eqref{5.4} for $x_0 = 0$ if and only if $\beta = 1/2$, whereas 
there exists no $\beta \in \bbC$ such that $f(\, \cdot \, ; \beta)$ satisfies \eqref{5.3} for $x_0 = 0$. Clearly, 
$\Lf (0_+ ; \beta) = 1$ and $\Lf(0_- ; \beta) = 0$ show that $x_0 = 0$ is a right and a left Lebesgue 
point of $f(\, \cdot \, ; \beta)$. In this paper we will not exploit \eqref{5.4} but always resort to 
\eqref{5.3} and its one-sided counterparts \eqref{5.1} and \eqref{5.2}. \\
$(ii)$ Introduce 
\begin{equation}
F(x) = \int_0^x dx' \, f(x'), \quad x \in \bbR, \; f \in L^1_{\loc} (\bbR; dx),    \lb{5.6} 
\end{equation}
then clearly $F \in AC([0,R])$ for all $R>0$ and $F'(x) = f(x)$ for (Lebesgue) a.e.\ $x \in [0,\infty)$. 
If $0$ is a right Lebesgue point of $f$, then actually 
\begin{equation}
F'(0_+) = \Lf(0_+).     \lb{5.7}
\end{equation}  
Similarly, if $x$ is a Lebesgue point of $f$ then $F'(x) = \Lf(x)$. (For additional material 
in this context see, e.g., \cite[Chs.\ 15, 16]{Jo01}.) $\Diamond$
\end{remark} 

Next, to facilitate the proof of Proposition \ref{p5.8}, we review a 
few basic facts of self-adjoint rank-one perturbations of a self-adjoint operator 
$A_0$ in $\cH$, following \cite{Ar57}, \cite{AD56}, \cite{DSS94}, \cite{Do65}, \cite{GPS08}, 
\cite{GT00}, \cite{SW86}: Assuming that $f_0$ is a cyclic vector for $A_0$, we consider the family 
of self-adjoint operators $A_{\alpha}$ in $\cH$,
\begin{equation}
A_{\alpha} = A_0 + \alpha (f_0, \, \cdot \,)_{\cH} f_0, \quad \alpha \geq 0.   \lb{5.40}
\end{equation} 
For instance (cf.\ \cite{GT00}) one can choose $A_0$ to be the operator of multiplication with the variable $\lambda$ in the Hilbert space $L^2(\bbR; d\omega_0(\lambda))$, where $\supp (d \omega)$ is compact (rendering $A_0$ a bounded self-adjoint operator) and then take 
$f_0 (\lambda) = 1$ for $\omega$-a.e.\ $\lambda \in \bbR$.

Then introducing the basic object,
\begin{align}
\begin{split} 
& F_{\alpha}(z) = \big(f_0, (A_{\alpha} - z I_{\cH})^{-1} f_0\big)_{\cH} 
= \int_{\sigma(A_{\alpha})} \f{d \omega_{\alpha} (\lambda)}{\lambda - z},    \lb{5.41} \\
& \, d \omega_{\alpha} (\cdot) = d \|E_{A_{\alpha}}(\cdot) f_0\|_{\cH}^2, \; z \in \bbC_+, \; 
\alpha \geq 0, 
\end{split} 
\end{align}
it is well-known that (cf.\ \cite[Appendix\ A]{GPS08})
\begin{align}
& F_{\alpha} (z) = \f{F_0(z)}{1 + \alpha F_0(z)}, \quad 
\Im(F_{\alpha}(z)) = \f{\Im(F_0(z))}{|1 + \alpha F_0(z)|^2}, \quad z \in \bbC_+,     \lb{5.42} \\
& \ln (1 + \alpha F_0(z)) = \int_{\bbR} \f{\xi(\lambda; A_{\alpha}, A_0) \, d\lambda}{\lambda - z},   
\quad z \in \bbC_+,     \lb{5.43} \\
& \xi(\lambda; A_{\alpha}, A_0) = \lim_{\varepsilon \downarrow 0} \pi^{-1} 
\Im(\ln(1 + \alpha F_0(\lambda + i \varepsilon))) \, \text{ for a.e.\ $\lambda \in \bbR$,}    \lb{5.44} \\
& 0 \leq \xi(\, \cdot \,; A_{\alpha}, A_0) \leq 1 \, \text{ a.e.\ on $\bbR$.}    \lb{5.45} 
\end{align}
Introducing
\begin{equation}
G_0(\lambda) = \int_{\bbR} \f{d \omega_0(\lambda')}{(\lambda - \lambda')^2}, \quad 
\lambda \in \bbR,    \lb{5.46}
\end{equation} 
($G_0$ takes values in $[0,\infty) \cup \{\infty\}$), and employing the usual Lebesgue decomposition 
\begin{equation} 
d\omega_{\alpha} = d\omega_{\alpha, ac} + d\omega_{\alpha, s} = d\omega_{\alpha, ac} 
+ d\omega_{\alpha, sc}  + d\omega_{\alpha, p}, \quad \alpha \geq 0,     \lb{5.47}
\end{equation} 
we recall that (cf., \cite{Ar57}, \cite{DSS94}, \cite{Do65}, \cite{SW86})
\begin{align}
& A = \{\lambda \in \bbR \,|\, \lim_{\varepsilon \downarrow 0} 
\Im(F_0(\lambda + i \varepsilon)) \in (0, \infty)\}
\, \text{ supports $d\omega_{\alpha, ac}$ for all $\alpha \geq 0$,}    \lb{5.48} \\  
& S_{\alpha} = \{\lambda \in \bbR \,|\, \lim_{\varepsilon \downarrow 0} F_0(\lambda + i \varepsilon) 
= - 1/\alpha; \, G_0(\lambda) = \infty\} \, \text{ supports $d\omega_{\alpha, sc}$,} 
\; \alpha \in (0,\infty),   \lb{5.49} \\ 
& P_{\alpha} = \{\lambda \in \bbR \,|\, \lim_{\varepsilon \downarrow 0} F_0(\lambda + i \varepsilon) 
= - 1/\alpha; \, G_0(\lambda) < \infty\} \,  \text{ equals $\sigma_{p}(A_{\alpha})$, the set of}     \no \\
& \hspace*{6.5cm} \text{eigenvalues of $A_{\alpha}$, $\alpha \in (0,\infty)$,}    
\lb{5.50} \\
& \, d \omega_{\alpha, p} (\lambda) = \alpha^{-2}\big(1 + \alpha^2\big) \sum_{\lambda_n \in P_{\alpha}} 
G_0(\lambda_n)^{-1} d \theta(\lambda - \lambda_n), \quad \alpha > 0,   \lb{5.51}
\end{align}
where $\quad \theta(x) = \begin{cases} 1, & x > 0, \\ 0, & x < 0. \end{cases}$ 

By definition, the set $S$ supports the measure $d\nu$ if $\nu(\bbR\backslash S) = 0$. In 
particular, $A$, $S_\alpha$, and $P_\alpha$, $\alpha \in (0,\infty)$, are mutually disjoint, and 
for $\beta, \gamma \in (0,\infty)$, $\beta \neq \gamma$, $d\omega_{\beta, s}$ and 
$d\omega_{\gamma, s}$ are mutually singular. 

Moreover, by \cite[Lemma\ 12]{AD56}, 
\begin{align} 
& \text{if there exist $c, d \in [0,1]$ with $0 < d - c <1$, such that}    \no \\ 
& \quad \text{$0 \leq c \leq \xi(\lambda ;A_{\alpha}, A_0) \leq d \leq 1$ a.e.\ on 
$(a,b) \subset \bbR$},     \lb{5.52} \\
& \quad \text{then $d\omega_{0} = d\omega_{0, ac}$ is purely absolutely 
continuous on $(a,b)$.}    \no 
\end{align} 

In addition, by \cite[Theorem\ A.4 and Proposition\ A.8]{GPS08}, if $A_0 \in \cB(\cH)$ and  
a cyclic vector $f_0 \in \cH$ is fixed, then the set of possible functions 
$\xi(\, \cdot \,; A_{\alpha}, A_0)$ associated with the pair of self-adjoint operators 
$(A_{\alpha}, A_0)$, $\alpha \geq 0$, coincides with the set 
\begin{align}
\{\xi \in L^1(\bbR; d\lambda) \,|\, \xi \, \text{real-valued}, \, \supp(\xi) \, \text{compact}, \, 
0 \leq \xi \leq 1 \, \text{a.e.~on\,$\bbR$}\}.     \lb{5.53}
\end{align} 

Moreover, we also need a few facts relating the truncated Hilbert transform,
\begin{equation}
(H_{\varepsilon} f)(\lambda) = \pi^{-1} \int_{(-\infty, -\varepsilon] \cup [\varepsilon,\infty)} 
\f{d\lambda' \, f(\lambda')}{\lambda - \lambda'}, \quad \lambda \in \bbR,  \; \varepsilon > 0, 
\lb{5.54}
\end{equation} 
and the Poisson and conjugate Poisson operator for the open upper half-plane,
\begin{align}
& (P_{\varepsilon} f)(\lambda) = \pi^{-1} \int_{\bbR} 
\f{d\lambda' \,  \varepsilon f(\lambda')}{(\lambda - \lambda')^2 + \varepsilon^2}, 
 \quad \lambda \in \bbR,  \; \varepsilon > 0,    \lb{5.55} \\
& (Q_{\varepsilon} f)(\lambda) = \pi^{-1} \int_{\bbR} 
\f{d\lambda' \, (\lambda - \lambda') f(\lambda')}{(\lambda - \lambda')^2 + \varepsilon^2}, 
 \quad \lambda \in \bbR,  \; \varepsilon > 0,   \lb{5.56}
\end{align} 
for appropriate set of functions $f$, for instance, 
$f \in L^1(\bbR; d\lambda) \cap L^{\infty}(\bbR; d\lambda)$, decaying sufficiently fast near 
$\pm \infty$, and satisfying some local regularity (e.g., H\"older-type) conditions. We note that 
(cf.\, e.g., \cite[Sect.\ 4.1.1]{Gr08}),
\begin{align}
& (H f)(\lambda) = \lim_{\varepsilon \downarrow 0} (H_{\varepsilon} f)(\lambda) \, \text{ exists 
for all $\lambda \in \bbR$ such that $f$ satisfies a H\"older} \no \\
& \quad \text{condition near $\lambda$, that is, for some $C_{\lambda} > 0$, 
$\varepsilon_{\lambda} > 0$, $\delta_{\lambda} > 0$, }    \no \\ 
& \quad \text{$|f(\lambda) - f(\lambda')| \leq C_{\lambda} |\lambda - \lambda'|^{\varepsilon_{\lambda}}$  
whenever $|\lambda - \lambda'| \leq \delta_{\lambda}$.}     \lb{5.57}
\end{align}
In addition, we mention the estimates (cf.\ \cite[p.\ 375--376]{Ki09})
\begin{align}
\begin{split} 
|(H_{\varepsilon} f)(\lambda) - (Q_{\varepsilon} f)(\lambda)| \leq 
\min\bigg((P_{\varepsilon} f)(\lambda), \pi^{-1} \int_{\bbR} 
\f{d\lambda' \, \varepsilon |f(\lambda') - f(\lambda)|}{(\lambda - \lambda')^2 + \varepsilon^2}\bigg),& \\
\lambda \in \bbR,  \; \varepsilon > 0,&     \lb{5.58}
\end{split} 
\end{align}
and (cf.\ \cite[p.\ 327]{AD56})
\begin{align}
\begin{split} 
\liminf_{\varepsilon \downarrow 0} (H_{\varepsilon} f)(\lambda) \leq 
\liminf_{\varepsilon \downarrow 0} (Q_{\varepsilon} f)(\lambda) \leq
\limsup_{\varepsilon \downarrow 0} (Q_{\varepsilon} f)(\lambda) \leq 
\limsup_{\varepsilon \downarrow 0} (H_{\varepsilon} f)(\lambda),&    \\
\lambda \in \bbR, \; \varepsilon > 0.&        \lb{5.59} 
\end{split} 
\end{align}

\section{Trace Relations and the Abel Transform} \lb{sB}
\renewcommand{\theequation}{B.\arabic{equation}}
\renewcommand{\thetheorem}{B.\arabic{theorem}}
\setcounter{theorem}{0} \setcounter{equation}{0}

While \cite{GLMST11} (see also \cite{Pu08}) exclusively focused on a resolvent approach 
to this circle of ideas as evidenced by Theorem \ref{t2.2}, other approaches are of course possible 
and have, in fact, been pursued. In particular, a heat kernel approach was originally used in 
\cite{GS88}, and more recently, it was also used in \cite{St11}. In this section we will identify the 
analog of the fundamental trace formula \eqref{2.19} in terms of general functions $f$ beyond 
resolvents and semigroups. 

To set up the principal result of this appendix we need some preparations and 
recall a number of results from \cite{BR79}, \cite{GLMST11}, \cite{GS88}, \cite{Pe05}, and 
\cite[Ch.~8]{Ya92}.

Suppose that $\bsT_j$, $j=1,2$, are self-adjoint operators in the separable, complex Hilbert space 
$\boldsymbol{\cK}$ satisfying 
\begin{align} 
& \bsT_j \geq 0, \; j=1,2,   \lb{B.1} \\ 
&\big[(\bsT_2 - z \bsI)^{-1}-(\bsT_1 - z \bsI)^{-1}\big] \in \cB_1(\boldsymbol{\cK}), \quad 
z\in\rho(\bsT_1) \cap \rho(\bsT_2).    \lb{B.2}
\end{align} 
Then \cite[Theorem~8.9.1]{Ya92} implies 
\begin{align}
& [f(\bsT_2) - f(\bsT_1)] \in \cB_1(\boldsymbol{\cK}),    \lb{B.3} \\
& \tr_{\boldsymbol{\cK}}\big(f(\bsT_2) - f(\bsT_1)\big)
 = \int_{[0,\infty)} \xi(\lambda; \bsT_2,\bsT_1) \, d\lambda \, f'(\lambda),    \lb{B.4}
\end{align}
for all $f\in C^2(\bbR)$ such that 
\begin{equation}
\text{for some $\varepsilon > 0$, } \, [\lambda^2 f'(\lambda)]' \underset{\lambda \to \infty}{=} 
\Oh\big(\lambda^{-1 - \varepsilon}\big).   \lb{B.5} 
\end{equation}
Here the spectral shift function $\xi(\, \cdot \,; \bsT_2,\bsT_1)$ for the pair $(\bsT_2, \bsT_1)$ 
is normalized by 
\begin{equation}
\xi(\lambda; \bsT_2,\bsT_1) = 0, \quad \lambda < 0.    \lb{B.6}
\end{equation} 

Next, we recall the modified homogeneous Besov 
space $\wti B^1_{\infty,1}(\bbR)$ (cf., e.g., \cite{Pe05}), consisting of the homogeneous Besov space 
$\dot B^1_{\infty,1}(\bbR)$ with all constants removed, 
that is, $\wti B^1_{\infty,1}(\bbR)$ consists of $C^1(\bbR)$-functions $\phi$ satisfying 
\begin{equation}
\phi \in \wti B^1_{\infty,1}(\bbR) \, \text{ if and only if } \, 
\sup_{x \in \bbR} |\phi'(x)| + \int_{\bbR} dx \, x^{-2} \big\|(D_x^2 \phi)(x)\big\|_{L^\infty(\bbR; dx)} 
< \infty,
\end{equation} 
where $(D_x \phi)(y) = \phi(y + x) - \phi(y)$, $x, y \in \bbR$. Then, with $\cK$ a separable, complex 
Hilbert space, one has the following result. 

\begin{theorem} [\cite{Pe85}, \cite{Pe90}] \lb{t6.2} 
Let $S_\pm$ be self-adjoint operators in $\cK$ satisfying 
\begin{align} 
& \dom(S_+) = \dom(S_-),      \\ 
& \, [S_+ - S_-] \in \cB_1(\cK). 
\end{align} 
In addition, assume that $F \in \wti B^1_{\infty,1}(\bbR)$. Then
\begin{equation}
[F(S_+) - F(S_-)] \in \cB_1(\cK).     \lb{B.10}
\end{equation} 
\end{theorem}

Following \cite{GS88}, a somewhat easier verifiable condition can be derived as follows. We denote
\begin{equation}
\big(\wti F\big)(k) = (2 \pi)^{-1/2} \slim_{R \to \infty} \int_{[-R, R]} d\nu \, F(x) e^{- i k \nu} \, 
\text{ for a.e.\ $k \in \bbR$}, \; F \in L^2(\bbR; d\nu),     
\end{equation}
and use the same symbol for the Fourier transform if $F \in L^1(\bbR; d\nu)$.

\begin{theorem} [\cite{GS88}] \lb{t6.3} 
Let $S_\pm$ be self-adjoint operators in $\cK$ satisfying  
\begin{align} 
& \dom(S_+) = \dom(S_-),     \\
& \, [S_+ - S_-] \in \cB_p(\cH) \, \text{ for some $p \in [1,\infty) \cup \{\infty\}$.} 
\end{align} 
In addition, assume that 
$(1 + |\cdot|) \wti F \in L^1(\bbR; dk)$. Then
\begin{equation}
\|F(S_+) - F(S_-)\|_{\cB_p(\cK)} \leq \bigg[(2 \pi)^{-1/2} 
\int_{\bbR} dk \, |k| \big|\wti F(k)\big|\bigg]
\|S_+ - S_-\|_{\cB_p(\cK)}.   \lb{B.14}
\end{equation} 
Finally, 
\begin{equation}
F, F', F'' \in L^2(\bbR; d\nu) \, \text{ implies } \,  (1 + |\cdot|) \wti F \in L^1(\bbR; dk).     \lb{B.15}  
\end{equation} 
\end{theorem}
\begin{proof}
To prove \eqref{B.14} it suffices to note that 
\begin{equation}
F(S_+) - F(S_-) = (2 \pi)^{-1/2} \int_{\bbR} dk \, \wti F(k) \big[e^{i k S_+} - e^{i k S_-}\big]
\end{equation}
and 
\begin{equation}
e^{i k S_+} - e^{i k S_-} = i k \int_0^1 ds \, e^{i s S_+} [S_+ - S_-] e^{i (1-s)S_-}, \quad 
k \in \bbR,
\end{equation}
imply \eqref{B.14}. 

If $F, F', F'' \in L^2(\bbR; d\nu)$, \cite[Lemma~7]{PS09} (see also \cite{PS10}) yields, 
\begin{align} 
& \big\|\wti F\big\|_{L^1(\bbR; dk)} = \int_{[-1,1]} dk \, \big| \wti F (k)\big| + 
\int_{\bbR\backslash [-1,1]} dk \, |k|^{-1} \big|k \wti F (k)\big|   \no \\ 
& \quad \leq 2^{1/2} \bigg(\int_{[-1,1]} dk \, \big| \wti F (k)\big|^2 \bigg)^{1/2}   \no \\
& \qquad + \bigg(\int_{\bbR \backslash [-1,1]} dk \, |k|^{-2} \bigg)^{1/2} 
\bigg(\int_{\bbR \backslash [-1,1]} dk \, \big|k \wti F (k)\big|^2 \bigg)^{1/2}    \no \\
& \quad \leq 2^{1/2} \big[\|F\|_{L^2(\bbR; d\nu)} + \|F'\|_{L^2(\bbR; d\nu)}\big],      \lb{B.17}  \\
& \big\|k \wti F\big\|_{L^1(\bbR; dk)} = \int_{[-1,1]} dk \, |k| \big| \wti F (k)\big| + 
\int_{\bbR\backslash [-1,1]} dk \, |k|^{-1} \big|k^2 \wti F (k)\big|   \no \\ 
& \quad \leq 2^{1/2} \bigg(\int_{[-1,1]} dk \, \big| \wti F (k)\big|^2 \bigg)^{1/2}   \no \\
& \qquad + \bigg(\int_{\bbR \backslash [-1,1]} dk \, |k|^{-2} \bigg)^{1/2} 
\bigg(\int_{\bbR \backslash [-1,1]} dk \, \big|k^2 \wti F (k)\big|^2 \bigg)^{1/2}    \no \\
& \quad \leq 2^{1/2} \big[\|F\|_{L^2(\bbR; d\nu)} + \|F''\|_{L^2(\bbR; d\nu)}\big],     \lb{B.18}
\end{align} 
implying $(1 + |\cdot|) \wti F \in L^1(\bbR; dk)$.
\end{proof}

We note that estimate \eqref{B.17} was earlier derived in \cite[Corollary 3.233]{BR79}, with the 
constant $2^{1/2}$ replaced by $\pi^{1/2}$. 

More generally, assuming 
\begin{equation}
h \in C^1(\bbR) \cap L^{\infty}(\bbR) \, \text{ to be real-valued,} \quad h'(\nu) > 0, \; \nu \in \bbR, 
\lb{B.21} 
\end{equation}
if one knows that 
\begin{equation}
[h(S_+) - h(S_-)] \in \cB_1(\cK)    \lb{B.22}
\end{equation} 
(rather than $[S_+ - S_-] \in \cB_1(\cK)$), to conclude $[F(S_+) - F(S_-)] \in \cB_1(\cK)$ 
one needs to verify 
\begin{equation}
F \circ h^{-1} \in \wti B^1_{\infty,1}(\bbR)    \lb{B.23}
\end{equation}
with $h^{-1}$ denoting the inverse function to $h$. By \eqref{B.14} and \eqref{B.15} it suffices to verify 
\begin{equation}
F \circ h^{-1}\in L^2(h(\bbR);dx), \; (F \circ h^{-1})' \in L^2(h(\bbR); dx), 
\; (F \circ h^{-1})'' \in L^2(h(\bbR); dx).    \lb{B.23*}
\end{equation}
In the particular case (cf.\ \eqref{2.16})
\begin{equation}
h(\nu) = g_{-1}(\nu) = \f{\nu}{(\nu^2 +1)^{1/2}}, \quad \nu \in \bbR, \lb{B.24}
\end{equation}
one verifies that \eqref{B.23*} is equivalent to 
\begin{align}
\begin{split} 
& (\nu^2 + 1)^{-3/4}F(\cdot) \in L^2(\bbR; d\nu), \; (\nu^2 + 1)^{3/4} F'(\cdot) \in L^2(\bbR; d \nu)),   \\
& (\nu^2 + 1)^{9/4}\big|F''(\cdot) + 3 \nu (\nu^2 + 1)^{-1} F'(\cdot) \big| \in L^2(\bbR; d \nu),    
\lb{B.25}
\end{split} 
\end{align}
and hence conditions \eqref{B.22} and \eqref{B.25} indeed guarantee that 
\begin{equation} 
[F(S_+) - F(S_-)] \in \cB_1(\cK).    \lb{B.25a} 
\end{equation} 

Continuing to assume \eqref{B.22}, we now introduce the spectral shift function associated with 
the pair $(S_+, S_-)$ via the invariance principle: Since $h(S_{\pm})$ are self-adjoint, Krein's 
trace formula in its simplest form yields (cf.\ \cite[Theorem\ 8.2.1]{Ya92})
\begin{equation}
\tr_{\cK}\big(h(S_+) - h(S_-)\big)
= \int_{[-1,1]} \xi(\omega; h(S_+), h(S_-)) \, d\omega,    \lb{B.26}
\end{equation} 
where
$\xi(\,\cdot\,; h(S_+), h(S_-))$ is the spectral shift function associated with the pair
$(h(S_+), h(S_-))$ uniquely determined by the requirement (cf.\ \cite[Sects.\ 8.9.1, 8.9.2]{Ya92})
\begin{equation}
\xi(\,\cdot\,; h(S_+), h(S_-)) \in L^1(\bbR; d\omega).  \lb{B.27}
\end{equation} 
Defining (cf.\ \cite[eq.\ 8.11.4]{Ya92})
\begin{equation}
\xi(\nu; S_+, S_-) := \xi(h(\nu); h(S_+), h(S_-)), \quad \nu\in\bbR,  \lb{B.28}
\end{equation}
this then implies (cf.\ \cite[Lemma~8.11.3]{Ya92})
\begin{equation}
\tr_{\cK}\big(\phi(S_+) - \phi(S_-)\big)
  = \int_{\bbR} \xi(\nu; S_+, S_-) \, d\nu \, \phi'(\nu),
\quad \phi \in C_0^{\infty}(\bbR).         \lb{B.29}
\end{equation}

Next, assuming in addition to \eqref{B.22} that also
\begin{equation}
\big[(S_+ - z I_{\cK})^{-1} - (S_- - z I_{\cK})^{-1}\big] \in \cB_1(\cK), \quad z \in \rho(S_+) \cap \rho(S_-), 
\lb{B.30}
\end{equation}
we contrast \eqref{B.26}--\eqref{B.29} with the following alternative approach to a trace formula of the type 
\eqref{B.29}: Introducing the spectral shift function $\widehat \xi(\,\cdot\,; S_+, S_-)$ associated with the 
pair $(S_+, S_-)$ as in \cite[eq.\ (8.7.4)]{Ya92}, one notes that $\widehat \xi(\,\cdot\,; S_+, S_-)$ is nonunique 
(in fact, it is unique up to an arbitrary additive constant) and that 
\begin{equation}
\widehat \xi(\,\cdot\,; S_+, S_-) \in L^1\big(\bbR; (|\nu| + 1)^{-2} d\nu\big)  \lb{B.31}
\end{equation}
(cf. \cite[Sect.\ 8.7]{Ya92}). At this point we arbitrarily fix the undetermined 
constant in $\widehat \xi(\,\cdot\,; S_+, S_-)$, and for simplicity, keep denoting the corresponding 
spectral shift function by the same symbol $\widehat \xi(\,\cdot\,; S_+, S_-)$.~As shown in \cite[Lemma~7.1]{GLMST11}, the functions $\xi(\,\cdot\,\; S_+, S_-)$ and $\widehat \xi(\,\cdot\,\; S_+, S_-)$ differ at most by a constant, that is, 
there exists a $C\in\bbR$ such that
\begin{equation}
\widehat \xi(\nu; S_+, S_-)=\xi(\nu; S_+, S_-) + C \, \text{ for a.e.\ } \, \nu\in \bbR.    \lb{B.32}
\end{equation}
Moreover, by \cite[Theorem\ 8.7.1]{Ya92} the trace formula 
\begin{equation}
{\tr}_{\cK} (G(S_+) - G(S_-)) = \int_{\bbR} \widehat\xi(\nu;S_+,S_-) \, d\nu \, G'(\nu)    \lb{B.33} 
\end{equation}
holds for the class of functions $G \in C^2(\bbR) \cap L^{\infty}(\bbR)$ satisfying the conditions
\begin{equation}
\text{for some $\varepsilon > 0$, } \, 
 [\nu^2 G'(\nu)]' \underset{|\nu| \to \infty}{=} O(|\nu|^{-1-\varepsilon}),   \lb{B.34} 
 \end{equation}
 and 
 \begin{align} 
 \begin{split} 
& \text{the limits } \, \lim_{\nu\to -\infty} G(\nu) = \lim_{\nu\to +\infty} G(\nu) \, \text{ and } \,    \\ 
& \quad \lim_{\nu\to -\infty} \nu^2 G'(\nu)=\lim_{\nu\to +\infty} \nu^2 G'(\nu) \, 
\text{ exist finitely.}     \lb{B.35}
\end{split} 
\end{align}
Next, one observes that \eqref{B.31} and \eqref{B.32} imply that
\begin{equation}
\xi(\,\cdot\,; S_+, S_-) \in L^1\big(\bbR; (|\nu| + 1)^{-2} d\nu\big), \lb{B.36}
\end{equation}
and combining \eqref{B.32}, \eqref{B.33}, and \eqref{B.35} one also obtains
\begin{equation}
{\tr}_{\cK} (G(S_+) - G(S_-)) = \int_{\bbR} \xi(\nu; S_+, S_-) \, d\nu \, G'(\nu).    \lb{B.37}
\end{equation}

These considerations permit one to extend the class of admissible functions $G$ in the trace 
formula \eqref{B.37} as follows.

\begin{lemma} \lb{lB.3}
Let $S_\pm$ be self-adjoint operators in $\cK$ satisfying 
\begin{align} 
& \dom(S_+) = \dom(S_-),     \lb{B.38} \\ 
& \, \big[(S_+ - z I_{\cK})^{-1} - (S_- - z I_{\cK})^{-1}\big] \in \cB_1(\cK), 
\quad z \in \rho(S_+) \cap \rho(S_-),     \lb{B.39} 
\end{align} 
Moreover, suppose that $k \in C^2(\bbR) \cap L^{\infty}(\bbR)$ is such that 
\begin{align}
& \text{for some $\varepsilon > 0$, } \, 
 [\nu^2 k'(\nu)]' \underset{|\nu|\to \infty}{=} O(|\nu|^{-1-\varepsilon}),   \lb{B.40} \\
\begin{split} 
& k_{\pm} = \lim_{\nu \to \pm \infty} k(\nu) \, \text{ exist finitely,}   \\ 
& \quad \text{the limits } \, \lim_{\nu\to -\infty} \nu^2 k'(\nu)=\lim_{\nu\to +\infty} \nu^2 k'(\nu) \, 
\text{ exist finitely,}     \lb{B.41}
\end{split} 
\end{align}
and the trace formula 
\begin{equation}
\tr_{\cK}(k(S_+) - k(S_-)) = \int_{\bbR} \xi(\nu; S_+, S_-) \, d\nu \, k'(\nu)         \lb{B.42}
\end{equation}
holds. In addition, assume that $H \in C^2(\bbR)$ satisfies the conditions
\begin{equation}
\text{for some $\varepsilon > 0$, } \, 
 [\nu^2 H'(\nu)]' \underset{|\nu|\to \infty}{=} O(|\nu|^{-1-\varepsilon}),   \lb{B.43} 
 \end{equation}
 and for some constant $c \in \bbR$, 
 \begin{align} 
 \begin{split} 
& H_{\pm} = \lim_{\nu \to \pm \infty} H(\nu) \, \text{ exist finitely, with } \, H_+ - H_- = c (k_+ - k_-),   \\ 
& \quad \text{the limits } \, \lim_{\nu\to -\infty} \nu^2 H'(\nu)=\lim_{\nu\to +\infty} \nu^2 H'(\nu) \, 
\text{ exist finitely,}     \lb{B.44}
\end{split} 
\end{align}
and 
\begin{equation}
[H(S_+) - H(S_-)] \in \cB_1(\cK).     \lb{B.45} 
\end{equation}
Then, 
\begin{equation}
{\tr}_{\cK} (H(S_+) - H(S_-)) = \int_{\bbR} \xi(\nu; S_+, S_-) \, d\nu \, H'(\nu).    \lb{B.46} 
\end{equation} 
\end{lemma}
\begin{proof}
Introduce the function $\hatt H$ by
\begin{equation}
\hatt H(\nu) = H(\nu) - c k(\nu), \quad \nu \in \bbR.    \lb{B.47}
\end{equation}
Then $\hatt H$ satisfies the conditions \eqref{B.33}--\eqref{B.35} and hence \eqref{B.37} applies, 
\begin{align}
& {\tr}_{\cK} (H(S_+) - H(S_-)) - c \tr_{\cK}(k(S_+) - k(S_-)) 
= {\tr}_{\cK} \big(\hatt H(S_+) - \hatt H(S_-)\big)    \no \\
& \quad = \int_{\bbR} \hatt \xi(\nu;S_+,S_-) \, d\nu \, \hatt H'(\nu)  
= \int_{\bbR} \xi(\nu;S_+,S_-) \, d\nu \, \hatt H'(\nu)    \no \\
& \quad = \int_{\bbR} \xi(\nu;S_+,S_-) \, d\nu \, H'(\nu) 
- c \int_{\bbR} \xi(\nu;S_+,S_-) \, d\nu \, k'(\nu), \lb{B.48} 
\end{align}
implying \eqref{B.46}. (Here we used the fact that $\int_{\bbR} d\nu \, \hatt H'(\nu) = 0$ in the second line of \eqref{B.48} on account of the first line in \eqref{B.44}.) 
\end{proof}

Next, consider $f \in C^1([0,\infty)) \cap L^{\infty}([0,\infty); d\lambda)$, satisfying for some $C>0$, 
\begin{equation}
|f'(\lambda)| \leq C (\lambda + 1)^{-1}, \; \lambda \geq 0,    \lb{B.49} 
\end{equation}
and introduce 
\begin{equation}
F'(\nu) := \f{1}{\pi} \int_{[\nu^2,\infty)} 
\f{d\lambda \, f'(\lambda)}{(\lambda-\nu^2)^{1/2}} = \f{1}{\pi} \int_{[0,\infty)} d\tau \, \tau^{-1/2} f'(\tau + \nu^2), 
\quad \nu \in \bbR.    \lb{B.50}
\end{equation}
Using the fact 
\begin{equation}
\int_{[0,\infty)} dx \, x^{\alpha - 1} (1 + r x)^{-\beta} = 
\f{\Gamma(\alpha) \Gamma(\beta-\alpha)}{\Gamma(\beta)} r^{-\alpha}, \quad r>0, \; \beta > \alpha > 0 
\lb{B.51}
\end{equation}
(cf.\ \cite[No.~3.1943, p.~285]{GR80}, with $\Gamma(\cdot)$ the gamma function, 
see \cite[Sect.~6.1]{AS72}), one verifies that $F'$ is well-defined and satisfies
\begin{equation}
F' \in C(\bbR) \cap L^{\infty}(\bbR; d \nu), \quad |F'(\nu)| \leq C (\nu^2 + 1)^{-1/2}, \; \nu \in \bbR. 
\lb{B.52}
\end{equation}
In particular, $F'$ is an even function 
\begin{equation}
F'(- \nu) = F'(\nu), \quad \nu \in \bbR,    \lb{B.53} 
\end{equation}
and hence to determine $F$ we use without loss of generality the normalization 
\begin{equation}
F(0) = 0    \lb{B.54} 
\end{equation}
rendering $F$, given by $F(\nu) = \int_0^{\nu} d\omega \, F'(\omega)$, an odd function, 
\begin{equation}
F(- \nu) = - F(\nu), \quad \nu \in \bbR.    \lb{B.55} 
\end{equation}

One notices that \eqref{B.50} represents an Abel-type transformation. \\
We claim that 
\begin{align}
\begin{split} 
F(\nu) &= \f{\nu}{2\pi} \int_{[\nu^2,\infty)} d\lambda \, \lambda^{-1} (\lambda - \nu^2)^{-1/2} 
[f(\lambda) - f(0)]     \\
&= \f{\nu}{2\pi} \int_{[0,\infty)} d\tau \, \tau^{-1/2} (\tau + \nu^2)^{-1} [f(\tau + \nu^2) - f(0)].     \lb{B.56} 
\end{split} 
\end{align}
This can either directly been proved by integrating \eqref{B.50} with respect to $\nu$, or alternatively,
denoting the right-hand side of \eqref{B.56} temporarily by $\wti F$, by proving that 
$\wti F'(\nu) = F'(\nu)$ with $F'$ given by \eqref{B.50}, and noting that $\wti F$ is an odd function. 
We follow the latter route and obtain
\begin{align}
\wti F'(\nu) &= \f{1}{2 \pi} \int_{[0,\infty)} d\tau \, \tau^{-1/2} \big[(\tau + \nu^2)^{-1} 
- 2 \nu^2 (\tau + \nu^2)^{-2}\big] [f(\tau + \nu^2) - f(0)]     \no \\
& \quad - \f{1}{\pi} \int_{[0,\infty)} d \tau \, \tau^{1/2} (\tau + \nu^2)^{-1} f'(\tau + \nu^2) 
+ \f{1}{\pi} \int_{[0,\infty)} d \tau \, \tau^{-1/2} f'(\tau + \nu^2)    \no \\
&= \f{1}{2 \pi} \int_{[0,\infty)} d\tau \, \tau^{-1/2} \big[(\tau + \nu^2)^{-1} 
- 2 \nu^2 (\tau + \nu^2)^{-2}\big] [f(\tau + \nu^2) - f(0)]     \no \\
& \quad - \f{1}{\pi} \tau^{1/2} (\tau + \nu^2)^{-1} f(\tau + \nu^2)\bigg|_{\tau = 0}^{\infty} 
+ \f{1}{\pi} \int_{[0,\infty)} d \tau \, \tau^{-1/2} f'(\tau + \nu^2)    \no \\
& \quad + \f{1}{\pi} \int_{[0,\infty)} d\tau \, \tau^{-1/2} \big[2^{-1} (\tau + \nu^2)^{-1} 
- \tau (\tau + \nu^2)^{-2}\big] f(\tau + \nu^2)    \no \\
&= - \f{f(0)}{2 \pi} \int_{[0,\infty)} d\tau \, \tau^{-1/2} \big[(\tau + \nu^2)^{-1} 
- 2 \nu^2 (\tau + \nu^2)^{-2}\big]     \no \\ 
& \quad + \f{1}{\pi} \int_{[0,\infty)} d \tau \, \tau^{-1/2} f'(\tau + \nu^2)     \no \\
& = \f{1}{\pi} \int_{[0,\infty)} d \tau \, \tau^{-1/2} f'(\tau + \nu^2) = F'(\nu), \quad \nu \in \bbR,  
\end{align}  
since $\int_{[0,\infty)} d\tau \, \tau^{-1/2} \big[(\tau + \nu^2)^{-1} 
- 2 \nu^2 (\tau + \nu^2)^{-2}\big] = 0$ by \eqref{B.51}. 

Applying \eqref{B.51} once again, this yields 
\begin{equation}
F \in C^1(\bbR) \cap L^{\infty}(\bbR; d\nu), \quad F(\pm \infty) = \mp f(0)/2.     \lb{B.57} 
\end{equation}

Given these preparations, one obtains the following result (assuming Hypothesis \ref{h2.1} for the 
remainder of this appendix and hence having Theorem \ref{t2.2} at our disposal).

\begin{lemma} \lb{lB.4}
Assume Hypothesis \ref{h2.1}. \\
$(i)$ Suppose $F \in C^2(\bbR) \cap L^{\infty}(\bbR; d\nu)$ satisfies 
\begin{align}
\begin{split}
& (\nu^2 + 1)^{-3/4}F(\cdot) \in L^2(\bbR; d\nu), \quad (\nu^2 + 1)^{3/4} F'(\cdot) \in L^2(\bbR; d \nu),   
\lb{B.58} \\
& (\nu^2 + 1)^{9/4}\big|F''(\cdot) + 3 \nu (\nu^2 + 1)^{-1} F'(\cdot) \big|  \in L^2(\bbR; d \nu).     
\end{split} 
\end{align} 
Then
\begin{equation}
[F(A_+) - F(A_-)] \in \cB_1(\cH).     \lb{B.59}
\end{equation}
If in addition, for some $\varepsilon > 0$,  
\begin{equation} 
 [\nu^2 F'(\nu)]' \underset{|\nu|\to \infty}{=} O(|\nu|^{-1-\varepsilon}),   \lb{B.60} 
 \end{equation}
 and 
 \begin{align} 
 \begin{split} 
& F_{\pm} = \lim_{\nu \to \pm \infty} F(\nu) \, \text{ exist finitely, with } \, F_+ = - F_- ,   \\ 
& \quad \text{the limits } \, \lim_{\nu\to -\infty} \nu^2 F'(\nu)=\lim_{\nu\to +\infty} \nu^2 F'(\nu) \, 
\text{ exist finitely,}     \lb{B.61}
\end{split} 
\end{align}
then, 
\begin{equation}
{\tr}_{\cH} (F(A_+) - F(A_-)) = \int_{\bbR} \xi(\nu; A_+, A_-) \, d\nu \, F'(\nu).    \lb{B.62} 
\end{equation} 
$(ii)$ Assume $f \in C^2([0,\infty)) \cap L^{\infty}([0,\infty); d\lambda)$ satisfies   
for some $\varepsilon > 0$, $C>0$, 
\begin{equation} 
|f'(\lambda)| \leq C (\lambda + 1)^{-(5 + \varepsilon)/2},  \quad 
|f''(\lambda)| \leq C (\lambda + 1)^{-(7 + \varepsilon)/2},  \; \lambda \geq 0.      \lb{B.63}
\end{equation}
Then $F$ defined by \eqref{B.56} with $F'$ given by \eqref{B.50}, satisfies 
$F \in C^2(\bbR) \cap L^{\infty}(\bbR; d\nu)$ and the inclusions \eqref{B.58}, and hence 
\eqref{B.59} and \eqref{B.62} hold in this case. 
\end{lemma} 
\begin{proof}
$(i)$ Identifying $\cK$ with $\cH$ and $S_{\pm}$ with $A_{\pm}$,  given the inclusions \eqref{B.58}, 
assertion \eqref{B.59} is a direct consequence of \eqref{2.5}, \eqref{2.7}, \eqref{2.18}, and 
\eqref{B.21}--\eqref{B.25a}. 
The fact that $\lim_{\nu \to \pm \infty} g_{-1}(\nu) = \pm 1$, combined with \eqref{2.24} permits us 
to apply Lemma \ref{lB.3}, proving the trace formula \eqref{B.62}.  \\
$(ii)$ Assumptions \eqref{B.63} combined with \eqref{B.51} yield
\begin{align}
& (\nu^2 + 1)^{-3/4} F(\nu) \underset{|\nu| \to \infty}{=} \Oh \big(|\nu|^{-3/2}\big), \quad 
(\nu^2 + 1)^{3/4} F'(\nu) \underset{|\nu| \to \infty}{=} \Oh \big(|\nu|^{-(5/2) - \varepsilon}\big),  \no \\
& (\nu^2 + 1)^{9/4} \big|F''(\nu) + 3 \nu (\nu^2 + 1)^{-1} F'(\nu)\big| 
\underset{|\nu| \to \infty}{=} \Oh \big(|\nu|^{-(1/2) - \varepsilon}\big),    \lb{B.64} 
\end{align}
and hence the inclusions \eqref{B.58} are valid. An application of \eqref{B.57} and item $(i)$
then completes the proof. 
\end{proof}

Thus, we arrive at the principal result of this appendix: 

\begin{theorem} \lb{tB.5}
Assume Hypothesis \ref{h2.1} and suppose that for some $\varepsilon > 0$, $C>0$, 
$f \in C^2([0,\infty)) \cap L^{\infty}([0,\infty); d\lambda)$ satisfies 
\begin{equation} 
|f'(\lambda)| \leq C (\lambda + 1)^{-(5 + \varepsilon)/2},  \quad 
|f''(\lambda)| \leq C (\lambda + 1)^{-(7 + \varepsilon)/2},  \; \lambda \geq 0.     \lb{B.65}
\end{equation}
Define $F \in C^2(\bbR) \cap L^{\infty}(\bbR; d\nu)$ by \eqref{B.56} with $F'$ given by \eqref{B.50}. 
Then 
\begin{align}
\begin{split} 
& \tr_{L^2(\bbR;\cH)} (f(\bsH_2) - f(\bsH_1))
 = \int_{[0,\infty)} \xi(\lambda; \bsH_2,\bsH_1) \, d\lambda \, f'(\lambda)     \\
& \quad = \int_{\bbR} \xi(\nu; A_+,A_-) \, d\nu \, F'(\nu) 
= \tr_{\cH} (F(A_+) - F(A_-)).    \lb{B.66} 
\end{split} 
\end{align} 
\end{theorem}
\begin{proof}
Identifying $\boldsymbol{\cK}$ with $L^2(\bbR;\cH)$ and $\bsT_j$ with $\bsH_j$, $j=1,2$, an 
application of \eqref{B.1}--\eqref{B.6} implies the first line in \eqref{B.66}. 
Employing Pushnitski's formula \eqref{2.37a} and transforming the integral over the region
\begin{equation}
(\lambda, \nu) \in [0,\infty) \times \big[-\lambda^{1/2}, \lambda^{1/2}\big]
\end{equation}
into the region
\begin{equation}
(\nu,\lambda) \in \bbR \times \big[\nu^2,\infty), 
\end{equation} 
yields 
\begin{align}
&\int_{[0,\infty)} \xi(\lambda; \bsH_2,\bsH_1) \, d\lambda \, f'(\lambda) 
 \no \\
& \quad = \int_{[0,\infty)} d \lambda \, f'(\lambda) 
\bigg[\frac{1}{\pi}\int_{-\lambda^{1/2}}^{\lambda^{1/2}}
\frac{\xi(\nu; A_+,A_-) \, d\nu}{(\lambda-\nu^2)^{1/2}} \bigg]    \no \\
& \quad = \int_{\bbR} \xi(\nu; A_+,A_-) \, d\nu \, \frac{1}{\pi} \int_{[\nu^2,\infty)} 
\frac{d\lambda \, f'(\lambda)}{(\lambda-\nu^2)^{1/2}}. 
\end{align}
Finally, \eqref{B.50} and \eqref{B.62} imply the equalities
\begin{align}
& \int_{\bbR} \xi(\nu; A_+,A_-) \, d\nu \, \frac{1}{\pi} \int_{[\nu^2,\infty)} 
\frac{d\lambda \, f'(\lambda)}{(\lambda-\nu^2)^{1/2}}   \no \\
& \quad = \int_{\bbR} \xi(\nu; A_+,A_-) \, d\nu \, F'(\nu) 
= \tr_{\cH} (F(A_+) - F(A_-)),  
\end{align}
completing the proof of \eqref{B.66}. 
\end{proof}

We conclude this appendix by isolating the special resolvent and heat kernel cases: 

\begin{example} \lb{eB.6}
Assume Hypothesis \ref{h2.1}. \\
$(i)$ Consider the {\bf resolvent case:} Given
\begin{equation}
f(\lambda) = (\lambda - z)^{-1}, \quad f'(\lambda) = - (\lambda - z)^{-2}, 
\quad \lambda \in [0,\infty), \; z \in \bbC\backslash [0,\infty), 
\end{equation}
one indeed verifies $($but not without tears; it requires the computation of some ``nasty'' 
integrals to be found in \cite[Nos.\ 3.1962, 3.2231]{GR80}$)$ that equations \eqref{B.50} and 
\eqref{B.56} imply 
\begin{equation}
F(\nu) = \f{1}{2 z} \f{\nu}{(\nu^2 - z)^{1/2}} = g_z(\nu), \quad 
F'(\nu) = - \f{1}{2} \f{1}{(\nu^2 - z)^{3/2}}, \quad 
\nu \in \bbR, \; z \in \bbC\backslash [0,\infty).    \lb{B.72} 
\end{equation}  
One verifies that for $z=-1$ $($it suffices to consider that case only\,$)$ \eqref{B.3}--\eqref{B.6} 
clearly apply to $f(\cdot) = (\, \cdot \, - z)^{-1}$, and so does Lemma \ref{B.4} to $F(\cdot)$ in 
\eqref{B.72}, since, for instance,   
\begin{equation}
F''(\nu) + 3 \nu (\nu^2 + 1)^{-1} F'(\nu) = 0, \quad \nu \in \bbR, 
\end{equation} 
$($for $z=-1$$)$, and also condition \eqref{B.60} and \eqref{B.61} are satisfied. 
We emphasize that in order to arrive at the inclusions \eqref{B.58}, one had to rely on  
\eqref{B.21}--\eqref{B.24}, which in turn are based on double operator integral $($DOI\,$)$ 
techniques. In particular, \eqref{B.22} with $h$ given by $g_{-1}$ in \eqref{B.24} had to 
be assumed to hold, and the latter, to date, can only be verified by DOI methods.  \\[1mm]
$(ii)$ Consider the {\bf heat kernel case:} Given
\begin{equation}
f(\lambda) = e^{-s \lambda}, \quad f'(\lambda) = - s e^{- s \lambda}, 
\quad \lambda \in [0,\infty), \; s \in (0,\infty), 
\end{equation}
one verifies $($with the help of \cite[No.\ 9.2541]{GR80}$)$ that 
\begin{equation}
F(\nu) = - \f{1}{2 } \erf\big(s^{1/2} \nu\big), \quad 
F'(\nu) = - \bigg(\f{s}{\pi}\bigg)^{1/2} e^{- s \nu^2}, \quad 
\nu \in \bbR, \; s \in (0,\infty),
\end{equation}  
where we used the standard abbreviation for the error function, 
\begin{equation}
\erf(x) = \f{2}{\pi^{1/2}} \int_0^x dy \, e^{- y^2}, \quad x \in \bbR.
\end{equation}
It is easily seen that Theorem \ref{tB.5} applies. 
One notes that this extends some of the formulas recently obtained in \cite{St11}. 
\end{example}

\section{An Alternative Proof of the Estimate \eqref{3.42}} \lb{sC}
\renewcommand{\theequation}{C.\arabic{equation}}
\renewcommand{\thetheorem}{C.\arabic{theorem}}
\setcounter{theorem}{0} \setcounter{equation}{0}

The purpose of this appendix is to offer a real variable method alternative to the 
complex analytic proof of the important estimate \eqref{3.42} in Section \ref{s3}:

\begin{theorem} \lb{tB.1} 
Suppose that $0 \leq S_j$, $j=1,2$, are nonnegative, self-adjoint operators 
in $\cH$ and assume that 
\begin{equation} 
\big[(S_2 - z_0 I_{\cH})^{-1} - (S_1 - z_0 I_{\cH})^{-1}\big] \in \cB_1(\cH)  
\, \text{ for some } \, z_0 \in \rho(S_1) \cap \rho(S_2).   
\end{equation} 
Then
\begin{equation}
\big[e^{- z S_2} - e^{- z S_1}\big] \in \cB_1(\cH), \quad \Re(z) >0.
\end{equation}
Moreover, for every $\varphi \in (0,\pi/2)$, there exists
$C_\varphi>0$ such that
\begin{equation} \lb{C.3}
\big\|e^{- z S_2} - e^{- z S_1}\big\|_{\cB_1(\cH)} \leq C_\varphi
\big(|z|^{-1} + |z|\big), \quad  z \in S_{\varphi}.
\end{equation}
\end{theorem}
\begin{proof}
For $z \in S_{\varphi}$ consider the function
\begin{equation} 
f(t) = \begin{cases} e^{-[z/|z|](t^{-1}-1)},  & t>0, \\
0, & t=0. \end{cases} 
\end{equation}
Then for every self-adjoint operator $S\geq 0$ in $\cH$ one has $e^{-zS}=f((I_\cH+|z|S)^{-1})$. Next, 
set $A_z :=(I_\cH+|z|S_1)^{-1}$ 
and $B_z :=(I_\cH+|z|S_2)^{-1}$. By (a slight variation of) Theorem \ref{t6.2} (see e.g. 
\cite[Theorem 4]{PS09} and \cite[Lemma 6]{PS10}), one has 
\begin{equation} 
\|e^{-zS_1}-e^{-zS_2}\|_{\cB_1(\cH)} = \|f(A_z)-f(B_z)\|_{\cB_1(\cH)} 
\leq c \|f\|_{C^2}\|A_z - B_z\|_{\cB_1(\cH)},
\end{equation} 
where $c>0$ is a constant and $\|f\|_{C^2}=\max_{n \in \{0,1,2\}}\|f^{(n)}\|_{L^{\infty}((0,\infty); dt)}$.
We will prove that $\|f\|_{C^2} \leq C_\varphi$ for some constant $0 < C_\varphi < \infty$, depending only 
on $\varphi$. One observes that $\|f\|_{C^2}$ does not change if one replaces $z/|z|$ with 
$\Re(z)/|z|$ in the definition of the function $f$. Using this replacement one obtains 
\begin{equation} 
\|f\|_{L^{\infty}((0,\infty); dt)} = \sup_{t>0}e^{-[\Re(z)/|z|](t^{-1}-1)}=e^{\Re(z)/|z|}\leq e.
\end{equation} 
Moreover, since 
\begin{equation} 
f'(t) = \begin{cases} e^{-[\Re(z)/|z|](t^{-1}-1)} [\Re(z)/|z|] t^{-2}, & t>0, \\ 
0, & t=0, \end{cases} 
\end{equation} 
it follows that 
\begin{align} 
\|f'\|_{L^{\infty}((0,\infty); dt)} &= \frac{\Re(z)}{|z|}e^{\Re(z)/|z|} 
\sup_{t>0} \Big(e^{-[\Re(z)/|z|]t^{-1}}t^{-2}\Big)    \no \\ 
& = \bigg(\frac{\Re(z)}{|z|}\bigg)^{-1} 
e^{[\Re(z)/|z|]} \sup_{t>0} \big(t^2e^{-t}\big)     \no \\ 
& \leq c [|z|/ \Re(z)] \leq C_\varphi.
\end{align} 
Estimating $\|f''\|_{L^{\infty}((0,\infty); dt)}$ in a similar way implies the required estimate.

To finish the proof of Theorem \ref{tB.1}, it remains to estimate the norm $\|A_z - B_z\|_{\cB_1(\cH)}$. 
An application of \eqref{3.5a} (replacing $z$ by $- |z|^{-1}$) yields
\begin{align}
\begin{split}
A_z - B_z = |z| (S_1-z_0I_\cH)(I_\cH+|z|S_1)^{-1}&\big[(S_1-z_0I_\cH)^{-1}-(S_2-z_0I_\cH)^{-1}\big] \\
& \times (S_2-z_0I_\cH)(I_\cH+|z|S_2)^{-1}.
\end{split}
\end{align}
Consequently,
\begin{align}
\|A_z - B_z\|_{\cB_1(\cH)}&\leq |z| \big\|(S_1-z_0I_\cH)^{-1}-(S_2-z_0I_\cH)^{-1}\big\|_{\cB_1(\cH)}    \no \\
& \quad \times \big\|(S_1-z_0I_\cH)(I_\cH+|z|S_1)^{-1}\big\|_{\cB(\cH)}    \\
& \quad \times \big\|(S_2-z_0I_\cH)(I_\cH+|z|S_2)^{-1}\big\|_{\cB(\cH)}.    \no 
\end{align}
Observing that 
\begin{align}
\big\|(S_j-z_0I_\cH)(I_\cH+|z|S_j)^{-1}\big\|_{\cB(\cH)} \leq 
\sup_{\lambda>0} \bigg(\frac{|\lambda-z_0|}{1+\lambda|z|}\bigg)  
\leq \sup_{\lambda>0} \bigg(\frac{\lambda+|z_0|}{1+\lambda|z|}\bigg)    \no \\
\le\sup_{\lambda>0} \bigg(\frac{\lambda}{1+\lambda|z|}\bigg) 
+\sup_{\lambda>0} \bigg(\frac{|z_0|}{1+\lambda|z|}\bigg) 
= |z|^{-1}+|z_0|, \quad j = 1,2, 
\end{align}
one arrives at
\begin{equation}
\|A_z - B_z\|_{\cB_1(\cH)}\leq |z|\big[|z|^{-1} + |z_0|\big]^2 \big\|(S_1-z_0I_\cH)^{-1}
- (S_2-z_0I_\cH)^{-1}\big\|_{\cB_1(\cH)}.
\end{equation} 

Combining this estimate with the estimate of $\|f\|_{C^2}$, yields \eqref{C.3}. 
\end{proof}

\medskip

\noindent 
{\bf Acknowledgments.} We are indebted to Jens Kaad, Yuri Latushkin, Galina Levitina, 
Konstantin Makarov, Michael Pang, Alexander Pushnitski, Marcus Waurick, and Dmitriy 
Zanin for helpful discussions, and to E.\ Brian Davies and Alexei Poltoratski for valuable correspondence.

F.G.\ gratefully acknowledges the extraordinary hospitality and stimulating atmosphere during 
his five-week visit to the Australian National University (ANU), Canberra, and to the University 
of New South Wales (UNSW), Sydney, in July/August of 2012, and the funding of his visit by 
the Australian Research Council. 

A.C., D.P., and F.S.\ gratefully acknowledge financial support from the Australian Research 
Council. A.C.\ also thanks the Alexander von Humboldt Stiftung and colleagues at the University 
of M\"unster. 

Y.T.\ was partially supported by the NCN grant DEC-2011/03/B/ST1/00407 and by the EU 
Marie Curie IRSES program, project ``AOS'', No.\ 318910.

Finally A.C., F.G., F.S., and Y.T.\ thank the Erwin Schr\"odinger International Institute for 
Mathematical Physics (ESI), Vienna, Austria, for funding support for this collaboration.

 
\end{document}